\documentclass[preprint,
3p,
times loads txfonts
]{elsarticle}

\usepackage{amsmath}
\usepackage{amsthm}
\usepackage{cases}
\usepackage{enumitem}              % Improved itemize & enumerate environments.
\usepackage{mathrsfs}
\usepackage{mathtools}
\usepackage[OMLmathsfit]{isomath}  % Loads also fixmath.sty (upper case italic greek letters, definition of \upDelta etc).
\usepackage{graphicx}
\usepackage{grffile}               % Fixes errors with images whose file names contain more than one dot, e.g., *.0000.png.
\usepackage[T1]{fontenc}

\usepackage{xcolor}

\usepackage[cmintegrals]{newpxmath}  % The fitting math symbols for newpxtext.

\usepackage{bm}                      % Provides \boldsymbol that we use in \vec{}.

\usepackage{dsfont}                                              % Double striked symbols.
  \newcommand{\IQ}{\ensuremath\mathds{Q}}                        % Set of rational numbers.
  \newcommand{\IR}{\ensuremath\mathds{R}}                        % Set of real numbers.
  \newcommand{\IP}{\ensuremath\mathds{P}}                        % Set of polynomials.

\renewcommand*{\vec}[1]{{\boldsymbol{#1}}}                       % Vector.
\DeclareMathAlphabet{\mathbfsf}{\encodingdefault}{\sfdefault}{bx}{n}
\newcommand*{\vecc}[1]{\mathbfsf{#1}}                            % Tensor or matrix.
\newcommand*{\transpose}[1]{{#1}^\mathrm{T}}                     % Transposition.
\newcommand*{\grad}{\vec{\nabla}}                                % Gradient.
\renewcommand*{\div}{\vec{\nabla}\cdot}                          % Divergence.
\newcommand*{\abs}[1]{\ensuremath{|#1|}}                         % Absolute value or volume.
\newcommand*{\norm}[2]{\|#1\|_{#2}}                              % Norm.
\newcommand*{\on}[2]{\left.#1\right\vert_{#2}}                   % restriction on a set, i.e., eg u|G.
\newcommand*{\prox}{\mathtt{prox}}                               % The proximal.

\usepackage{longtable}
\usepackage{multirow}
\usepackage{tabularx}                                            % Specify total length of tabular.
  \newcolumntype{R}{>{\raggedleft\arraybackslash}X}
  \newcolumntype{L}{>{\raggedright\arraybackslash}X}
  \newcolumntype{C}{>{\centering\arraybackslash}X}
\usepackage{booktabs}                                            % Improve tabulars, enable \top-, \mid-, \bottomrule.
\usepackage{rotating}                                            % Provide command for rotating table contents.
          
\newtheorem{lemma}{Lemma}
\newtheorem{theorem}{Theorem}
\newtheorem{remark}{Remark}

\journal{ }

\begin{document}

\begin{frontmatter}
\title{Efficient optimization-based invariant-domain-preserving limiters in solving gas dynamics equations}
\author[label1]{Chen Liu}\ead{chenl@uark.edu}
\author[label4]{Dionysis Milesis}\ead{dmilesis@bu.edu}
\author[label2]{Chi-Wang Shu}\ead{chi-wang_shu@brown.edu}
\author[label3]{Xiangxiong Zhang\corref{cor1}}\ead{zhan1966@purdue.edu}
\address[label1]{Department of Mathematical Sciences, University of Arkansas, Fayetteville, Arkansas 72701, USA.}
\address[label4]{Department of Mathematics and Statistics, Boston University, Boston, MA 02215, USA.}
\address[label2]{Division of Applied Mathematics, Brown University, Providence, Rhode Island 02912, USA.}
\address[label3]{Department of Mathematics, Purdue University, West Lafayette, Indiana 47907, USA.}

\begin{abstract}
We introduce effective splitting methods for implementing optimization-based limiters to 
enforce the invariant domain in gas dynamics in high order accurate numerical schemes.
The key ingredients include an easy and efficient explicit formulation of the projection onto the invariant domain set, and also proper applications of the classical Douglas-Rachford splitting
and its more recent extension Davis-Yin splitting. 
Such an optimization-based approach can be applied to many numerical schemes to construct high order accurate, globally conservative, and invariant-domain-preserving schemes for compressible flow equations.
As a demonstration, we apply it to high order discontinuous Galerkin schemes and test it 
on demanding benchmarks to validate the robustness and performance of both $\ell^1$-norm minimization limiter and $\ell^2$-norm minimization limiter.
\end{abstract}

\begin{keyword}
optimization-based limiters \sep invariant-domain-preserving \sep positivity-preserving \sep discontinuous Galerkin \sep Davis-Yin splitting \sep Douglas-Rachford splitting 

\vspace{.5\baselineskip}
\MSC 65K05 \sep 65K10 \sep 65M60 \sep 90C25 
\end{keyword}
\end{frontmatter}

%% Content
%%%%%%%%%%%%%%%%%%%%%%%%%%%%%%%%%%%%%%%%%%%%%%%%%%%%%%%%%%%%%%%%%%%%%%%%%%%%%%%%%%%%%%%%%%%%%%%%%%%%%%%%%%%%%%%%%%%%%%%
\section{Introduction}
%%%%%%%%%%%%%%%%%%%%%%%%%%%%%%%%%%%%%%%%%%%%%%%%%%%%%%%%%%%%%%%%%%%%%%%%%%%%%%%%%%%%%%%%%%%%%%%%%%%%%%%%%%%%%%%%%%%%%%%
\subsection{Background and motivation}
The compressible  Euler and Navier-Stokes (NS) equations are fundamental models in gas dynamics, with broad applications in aeronautics, astronautics, and astrophysics. 
When solving gas dynamics equations,  it is necessary for numerical methods to preserve positivity of density and internal energy, not only to produce physically meaningful solutions, but more importantly to achieve nonlinear stability, especially for demanding applications involving low density and pressure such as very high speed shocks and blast waves \cite{ha2005numerical, wang2012robust}. 

The conserved variables in the gas dynamics equations are density, momentum, and total energy denoted by $\rho, \vec{m}, E$. Then the internal energy can be written as $\rho e =E-\frac12 \frac{\|\vec m\|^2}{\rho}$, and   for ideal gas the pressure is $p=(\gamma-1)\rho e$ with a constant $\gamma>1$. Let $\vec{m}$ be a column vector, then the  set of  admissible states of positive density and positivity internal energy (or pressure) can be written as
\begin{equation}
    G = \left\{\vec{U} = \transpose{[\rho,\transpose{\vec{m}},E]}\!:~ \rho > 0,~~ \rho e(\vec{U}) = E - \frac{\norm{\vec{m}}{}^2}{2\rho} > 0\right\}.
    \label{invariant-domain-original}
\end{equation}
Since the function $\rho e(\vec{U})$ is concave with respect to $\vec{U}$, by Jensen's inequality, the set $G$ is convex.
In the literature,  numerical schemes with solutions staying in the convex set $G$ are positivity-preserving schemes
\cite{einfeldt1991godunov,zhang2010positivity}, and such schemes are also called invariant-domain-preserving schemes \cite{guermond2021second}
since well-posed solutions should stay in this invariant domain set $G$.
 For numerical implementation, a numerical admissible set or invariant domain $G^\varepsilon$ is often considered with a small $\varepsilon > 0$:
\begin{align}\label{invariant-domain}
G^\varepsilon = \left\{\vec{U} = \transpose{[\rho, \transpose{\vec{m}}, E]}\!:~ \rho \geq \varepsilon,~~ \rho e(\vec{U}) = E - \frac{\norm{\vec{m}}{2}^2}{2\rho} \geq \varepsilon\right\},
\end{align} which is a convex and closed set. We refer to \cite{wu2025high} for a recent comprehensive review on invariant-domain-preserving schemes.

For enforcing bounds or positivity, there are quite a few approaches developed in the past 15 years for  high order accurate schemes,  most of which use fully explicit time discretizations \cite{zhang2010maximum,guermond2017invariant,kuzmin2024property}. For the compressible NS equations, the fully explicit method \cite{zhang2017positivity} typically suffers from restrictive time step size constraints like $\Delta t = \mathcal{O}(\operatorname{Re}\Delta{x}^2)$, which is suitable only for large Reynolds numbers $\operatorname{Re}\gg 1$. To use larger time step sizes like $\Delta t = \mathcal{O}(\Delta{x})$, 
 semi-implicit \cite{guermond2021second,LZ2022CNS} and the fully implicit method \cite{grapsas2016unconditionally} can be considered,  
 but   extensions of such positivity-preserving schemes to arbitrarily high order accuracy are in general quite difficult. For a more detailed review, see \cite[Section~1.2]{LZ2022CNS}.

\par
For avoiding small time steps and easy extensions to arbitrarily high order accuracy, one approach is the optimization-based method \cite{guba2014optimization,van2019positivity,ruppenthal2023optimal,liu2024simple}. 
A minimization is solved to seek minimal modification of the given numerical solution while enforcing the conservation and positivity or bounds as constraints. In this paper, we focus on how to design such an optimization-based method with the invariant domain $G^\varepsilon$ as a constraint, which has not been well studied in previous efforts of designing optimization-based positivity-preserving limiters in the literature.

\subsection{Optimization-based invariant-domain-preserving limiters}\label{sec:compute_prox}

An optimization-based limiter can be applied to any numerical schemes. In this paper, we focus on discontinuous Galerkin (DG) methods as an example, and the generalization to other popular schemes such as finite difference, finite volume, and continuous finite element method is straightforward. For simplicity, we only consider DG on uniform meshes.
For reducing computational cost, we only apply optimization-based limiters to post process the cell averages of DG solutions so that cell averages are ensured to stay in the invariant domain $G^\varepsilon$, after which the simple scaling limiter in \cite{zhang2010positivity} can be applied to each DG polynomial in each cell to further ensure quadrature point values to be in the invariant domain $G^\varepsilon$. Another alternative is to apply optimization-based limiters to all the  quadrature point values of the DG solution, to which the efficient operator splitting methods in this paper can also be used.
   
Given a DG solution $\vec{U}_h$, define the piecewise constant function $\overline{\vec{U}_h}$ as the cell average of $\vec{U}_h$ in each cell $K$, that is, $\overline{\vec{U}_h}|_{K} = \frac{1}{\abs{K}}\int_K \vec{U}_h$. If there exists a bad cell $K$ with $\overline{\vec{U}_h}|_{K}\notin G^\varepsilon$, then we seek a piecewise constant $\vec{X}_h$ which minimizes the distance to $\overline{\vec{U}_h}$ under a chosen norm with the constraints of preserving global conservation and invariant domain:
\begin{align}\label{eq:introduction_opt_model}
\min_{\vec{X}_h} \norm{\vec{X}_h - \overline{\vec{U}_h}}{}
\quad\text{subject to}\quad
\int_\Omega \vec{X}_h = \int_\Omega \vec{U}_h ~~\text{and}~~ 
\on{\vec{X}_h}{K_i}\in G^\varepsilon~~\text{for any cell}~K_i.
\end{align}
One can certainly add more constraints to the  minimization above for enforcing more desired physical properties. On the other hand, one major computational challenge is how to efficiently solve such a constrained minimization, which is a limiter to be applied to each time step of high order schemes solving a time dependent problem. As will be shown in this paper, the two constraints in  \eqref{eq:introduction_opt_model} can be efficiently handled by splitting methods. 

Let $\overline{\vec{W}_h}$ denote the minimizer to \eqref{eq:introduction_opt_model}. Then after applying the optimization-based limiter  \eqref{eq:introduction_opt_model}, the postprocessed DG polynomial can be written as
\begin{equation}\label{eq:introduction_opt_model1_post} 
\widehat{\vec{U}}_h = (\vec{U}_h - \overline{\vec{U}_h}) + \overline{\vec{W}_h},
\end{equation}
and it preserves global conservation $\int_\Omega \widehat{\vec{U}}_h = \int_\Omega \vec{U}_h$ and 
$\widehat{\vec{U}}_h$ has  cell averages $\overline{\vec{W}_h}|_{K_i}\in G^\varepsilon$.

The modification \eqref{eq:introduction_opt_model1_post} does not destroy the approximation order  \cite{liu2024simple,LZ2024CNS,LHTZ2024FP}.
In particular,
let $\overline{\vec{U}}$ denote the cell average of the exact solution, and assume the numerical solution has the same integral as the exact solution, then $\overline{\vec{U}}$ satisfies both constraints in \eqref{eq:introduction_opt_model}, thus the minimizer $\overline{\vec{W}_h}$ to \eqref{eq:introduction_opt_model} 
 satisfies $\|\overline{\vec{W}_h}-\overline{\vec{U}_h}\|\leq \|\overline{\vec{U}}-\overline{\vec{U}_h}\|$, and 
 \begin{equation}
     \|\overline{\vec{W}_h}-\overline{\vec{U}}\|\leq \|\overline{\vec{W}_h}-\overline{\vec{U}_h}\|+\|\overline{\vec{U}_h}-\overline{\vec{U}}\|\leq 2\|\overline{\vec{U}_h}-\overline{\vec{U}}\|.
     \label{accuracy-proof}
 \end{equation}  This simple argument shows that the limiter \eqref{eq:introduction_opt_model} does not destroy the order of accuracy w.r.t. the chosen norm in \eqref{eq:introduction_opt_model}. 
As will be shown in Theorem \ref{theorem-limiter} in Section \ref{sec:limiter_Euler},  we can get a better accuracy result than   \eqref{accuracy-proof} if $L^2$-norm is used, i.e., taking $\norm{\cdot}{} = \norm{\cdot}{L^2}$ in \eqref{eq:introduction_opt_model},
\begin{align} 
\|\overline{\vec{W}_h}-\overline{\vec{U}}\|_{L^2}<\|\overline{\vec{U}_h}-\overline{\vec{U}}\|_{L^2}.
\end{align}

Once the DG polynomial cell averages are in the invariant domain $G^\varepsilon$,
 the simple scaling limiter in \cite{zhang2010maximum, zhang2010positivity} can be applied to further enforce quadrature point values to be in $G^\varepsilon$. 
In the literature, such a limiter is sometimes called Zhang-Shu limiter, which does not affect  accuracy either \cite{zhang2017positivity}.

\subsection{Operator splitting methods in convex optimization}

In \eqref{eq:introduction_opt_model}, we still need to specify a norm, and
a popular choice is to employ the $L^2$ norm, i.e. taking $\norm{\cdot}{} = \norm{\cdot}{L^2}$. Meanwhile, the $L^1$ norm $\norm{\cdot}{} = \norm{\cdot}{L^1}$ has also been often considered in the literature. As will be shown in our numerical tests, $L^1$ norm limiter may be more desirable for certain problems, such as the astrophysical jet. 
In the convex optimization literature,
the indicator function for a closed convex $\Lambda$ is defined as: 
\begin{align}
\iota_\Lambda(\vec{X}) = \begin{cases}
    0, & \vec{X}\in \Lambda,\\
    +\infty, & \vec{X}\notin \Lambda.
\end{cases}
\end{align}
We remark that  $+\infty$ must be used in the definition of the indicator function, instead of a large enough positive number, so that $\iota_\Lambda(\vec{X})$ can be a proper convex function, 
see \cite[Section 3.2]{chen2025enforcing}.
With respect to the $L^2$ and $L^1$ norms, the optimization problem \eqref{eq:introduction_opt_model} can be equivalently written as the following two models, respectively,
\begin{align}
\text{($L^2$~model)}:&\quad~
\min_{\vec{X}_h} \norm{\vec{X}_h - \overline{\vec{U}_h}}{L^2}^2 + \iota_{\Lambda_1}(\vec{X}_h) + \iota_{\Lambda_2}(\vec{X}_h),\label{eq:opt_model_L2_org}\\
\text{($L^1$~model)}:&\quad~
\min_{\vec{X}_h} \norm{\vec{X}_h - \overline{\vec{U}_h}}{L^1} + \iota_{\Lambda_1}(\vec{X}_h) + \iota_{\Lambda_2}(\vec{X}_h),\label{eq:opt_model_L1_org}
\end{align}
where the two closed convex sets 
\begin{align}
\Lambda_1 = \{\vec{X}_h\!: \int_\Omega \vec{X}_h = \int_\Omega \vec{U}_h\},\qquad \Lambda_2 = \{\vec{X}_h\!: \vec{X}_h|_{K_i}\in G^\varepsilon,~\forall~i\},
\end{align}
are associated to conservation and invariant-domain-preserving constraints.

For solving \eqref{eq:opt_model_L2_org} and \eqref{eq:opt_model_L1_org}, which are a composite of a few convex closed proper functions, the scalability of first order operator splitting methods suits well for large problems. 
For two operator splitting convex optimization scheme, the most robust splitting is the Douglas-Rachford splitting (DRS), which was first introduced in 1950s for solving heat equations in two dimensions \cite{peaceman1955numerical,douglas1956numerical}. 
In 1979, Lions and Mercier extended DRS to composite convex optimization \cite{lions1979splitting}. It has been well known that quite a few popular splitting methods such as the alternating direction method of multipliers (ADMM) and the dual split Bregman method can be exactly equivalent to DRS under proper choices of parameters. 
See \cite{demanet2016eventual,torres2025asymptotic} and references therein for the exact equivalent relation. 
In \cite{raguet2019note,raguet2013generalized}, it was shown that DRS and forward-backward splitting can be integrated, followed by its extension in \cite{briceno2015forward}. Such a method was proven convergent by Davis and Yin \cite{davis2017three}, and such a three operator splitting is also referred to as Davis-Yin splitting (DYS).
\par

We emphasize that it is critical to design solvers to efficiently solve  \eqref{eq:introduction_opt_model}, which is needed in each time step for solving a time-dependent PDE. 
The  DRS in \cite{lions1979splitting} has been shown to be   efficient for solving optimization-based limiters for enforcing global conservation and bounds of scalar variables in high order DG schemes solving challenging PDEs including Cahn-Hilliard-Navier-Stokes system \cite{liu2024simple}, compressible Navier-Stokes \cite{LZ2024CNS}, and highly anisotropic diffusion \cite{LHTZ2024FP}.  
In the rest of the paper, we focus on how to construct splitting methods to solve \eqref{eq:opt_model_L2_org} and \eqref{eq:opt_model_L1_org} for invariant-domain-preserving of a vector variable $\vec{U}$, which is a more difficult problem than the simpler problem of enforcing bounds or positivity for a scalar variable in 
\cite{liu2024simple,LZ2024CNS,LHTZ2024FP}. 

\paragraph{\bf Two-operator Douglas-Rachford splitting}

\par
Let $\partial{g}$ and $\partial{h}$ be the subdifferentials of convex functions $g$ and $h$. 
Let $\mathrm{I}$ denote an identity operator. The proximal operator with a parameter $\gamma>0$ for the convex function $g$ is defined as 
\begin{align}
\prox_g^\gamma(\vec{X}) = (\mathrm{I}+\gamma \partial{g})^{-1}(\vec{X})=\mathrm{argmin}_{\vec{Z}}  g(\vec{Z}) + \frac{1}{2\gamma}\norm{\vec{Z}-\vec{X}}{2}^2,
\end{align}
and $\prox_h^\gamma$ is defined similarly.
To find a minimizer $\vec{X}^\ast$ of $g(\vec{X}) + h(\vec{X})$, the generalized DRS iteration with a step size $\gamma>0$ and a relaxation parameter $\lambda\in(0,2)$ is given by:
\begin{align}\label{eq:DR_algorithm}
\text{(DRS)}~
\begin{cases}
\vec{Y}^{k+1} &\hspace{-0.75em}= \lambda\,\prox_g^\gamma(2\vec{X}^k - \vec{Y}^k) + \vec{Y}^k - \lambda\vec{X}^k, \\
\vec{X}^{k+1} &\hspace{-0.75em}= \prox_h^\gamma(\vec{Y}^{k+1}),
\end{cases}
\end{align}
in which $\vec{Y}$ is only an auxiliary variable, and $\vec{X}^k$ will converge to the minimizer for any fixed $\gamma>0$ and   $\lambda\in(0,2)$ if the two functions $g$ and $h$ are convex closed proper, e.g., $L^1$ function and indicator functions of a closed convex set. 
If one of the functions is also strongly convex, then \eqref{eq:DR_algorithm} also converges for $\lambda = 2$, and
the special case  $\lambda = 2$ is also called Peaceman-Rachford splitting. 
 For using DRS for solving \eqref{eq:opt_model_L2_org} for enforcing bounds of a scalar variable,  in \cite{liu2024simple} the asymptotic linear convergence of \eqref{eq:DR_algorithm} has been analyzed, from which a simple formula for selecting optimal parameters $\gamma$ and $\lambda$ is derived, but neither the analysis nor the parameter formula in \cite{liu2024simple} can be extended to the vector variable case for the invariant domain \eqref{invariant-domain}.

\paragraph{\bf Three-operator Davis-Yin splitting}
For solving a problem of the form $\min_{\vec{X}}f(\vec{X}) + g(\vec{X}) + h(\vec{X})$,  DYS in \cite{davis2017three}  is given by:
\begin{align}\label{eq:DY_algorithm}
\text{(DYS)}~
\begin{cases}
\vec{X}^{k+1/2} &\hspace{-0.75em}= \prox_g^\gamma(\vec{Z}^k),\\
\vec{X}^{k+1} &\hspace{-0.75em}= \prox_f^\gamma\big(2\vec{X}^{k+1/2} - \vec{Z}^{k} - \gamma\grad{h}(\vec{X}^{k+1/2})\big),\\
\vec{Z}^{k+1} &\hspace{-0.75em}= \vec{Z}^{k} + \vec{X}^{k+1} - \vec{X}^{k+1/2}.
\end{cases}
\end{align}
For proper closed convex functions $f$, $g$, and $h$, where $\grad{h}$ is Lipschitz continuous with the Lipschitz constant $L$, iteration \eqref{eq:DY_algorithm} converges for any constant step size $\gamma\in(0, 2/L)$. 
We remark that three-block ADMM methods can also be used here, but 
in practice DYS $\gamma=\frac{1}{L}$ performs much better than all other alternatives, see a numerical comparison for solving \eqref{eq:opt_model_L2_org} for enforcing bounds of a scalar variable in \cite{anshika2024three}.

\subsection{The main results of this paper}

For solving \eqref{eq:opt_model_L2_org}, both DRS \eqref{eq:DR_algorithm} and DYS  \eqref{eq:DY_algorithm} can be used, if all the proximal operators are available.
 For the indicator function $\iota_\Omega$ for a closed convex set $\Omega$, its proximal operator is simply the Euclidean projection to the set $\Omega$. To implement DRS and DYS, the only nontrivial operator is the projection onto the invariant domain. Even though the definition of the set $G^\varepsilon$ in \eqref{invariant-domain} seems simple, its projection formula is no longer straightforward. The first key result of this paper is to derive an efficient explicit projection formula  for convex set $G^\varepsilon$ in \eqref{invariant-domain}, with which 
both DRS \eqref{eq:DR_algorithm} and DYS  \eqref{eq:DY_algorithm} can be easily implemented for solving \eqref{eq:opt_model_L2_org}. We are able to derive the projection formula as a cubic root. For the ease of presentation, we list the projection algorithms in \ref{sec:1D_proj} to \ref{sec:3D_proj}.

For solving \eqref{eq:opt_model_L1_org}, neither DRS \eqref{eq:DR_algorithm} nor DYS  \eqref{eq:DY_algorithm} can be directly used. As will be shown in this paper, \eqref{eq:opt_model_L2_org} can be regarded as a subproblem in DRS solving \eqref{eq:opt_model_L1_org}, and this subproblem can be solved efficiently by DYS  \eqref{eq:DY_algorithm}. In other words, we propose to use DRS nested with DYS for solving  \eqref{eq:opt_model_L1_org}, which is an efficient way for solving $L^1$-norm optimization-based invariant-domain-preserving limiter. 

Our numerical results indicate that the $L^1$-norm limiter \eqref{eq:opt_model_L1_org} is not better than the $L^2$-norm limiter \eqref{eq:opt_model_L2_org} in terms of sparsity of the minimizer, unlike   $\ell^1$-norm minimizations in many other applications. Moreover, the $L^2$-norm limiter \eqref{eq:opt_model_L2_org} is cheaper to solve, and it improves the accuracy of the DG solution,  as shown in Theorem \ref{theorem-limiter} in Section \ref{sec:limiter_Euler}. 
On the other hand, as will be shown in our numerical tests, for certain important time dependent problems, the $L^1$-norm limiter \eqref{eq:opt_model_L1_org} may be more desirable. For instance, as shown in Figure \ref{fig:astrophysical_jet_density} for the high speed astrophysical jet problem, the optimization-based limiter is triggered less during the time evolution if using  $L^1$-norm limiter \eqref{eq:opt_model_L1_org}, compared to  using $L^2$-norm limiter \eqref{eq:opt_model_L2_org}.

\subsection{Related work}

In the past 15 years, there have been active research on constructing  invariant-domain-preserving limiters  for high order schemes solving hyperbolic and related systems.
One popular approach introduced by Zhang and Shu \cite{zhang2010maximum,zhang2010positivity,wu2025high} is to only modify the approximation polynomial by a simple scaling limiter to eliminate overshoots and undershoots at certain quadrature points, with which DG and finite volume type schemes using strong stability preserving (SSP) time discretizations can ensure the cell averages to be invariant-domain-preserving. See also \cite[Appendix~C]{zhang2017positivity} for the accuracy proof of this limiter.
Another widely used approach is the flux corrected transport (FCT) type method \cite{zalesak1979fully}, which is to modify the high order numerical flux by a convex combination with a low order flux \cite{guermond2017invariant,guermond2018second,kuzmin2020monolithic,hajduk2021monolithic}. The FCT is flexible and applicable to a broad range of spatial discretizations. Recent developments include convex limiting  and subcell convex limiting \cite{hajduk2021monolithic,pazner2021sparse,rueda2022subcell,lin2024high}, etc. It is in general  difficult to establish rigorous accuracy proof for FCT type methods, though the accuracy for the 1D linear scalar case can be proven \cite{xu2014parametrized}.
All these approaches rely on existence of a low order invariant-domain-preserving numerical flux, which is however not always available for complicated systems, e.g., higher order PDEs like phase-field equations \cite{liu2024simple}.
\par

As a more flexible alternative to the traditional methods above, optimization-based limiters like  \eqref{eq:introduction_opt_model} can be considered,
and it is mostly studied for a scalar variable, for which the admissible set $G = [m,M]$ means bound-preserving, or $G = [0,+\infty]$ means positivity-preserving.
The same or similar optimization-based limiters  for scalar quantities have been considered under different contexts, e.g., \cite{guba2014optimization,bochev2012constrained,bradley2019communication,yee2020quadratic,bochev2020optimization,peterson2024optimization}.
See also \cite{van2019positivity,ruppenthal2023optimal} for different optimization-based approaches for enforcing bounds.  In \cite{liu2024simple,LZ2024CNS,LHTZ2024FP}, 
 the optimization-based limiter was considered for some challenging equations and systems but the limiter was only applied to scalar variables. 

To the best of our knowledge, all these existing optimization-based approaches consider the treatment of a scalar variable, rather than directly tackling the invariant domain set $G^\varepsilon$ defined in \eqref{invariant-domain}. On the other hand, if $\vec{X}$ is a scalar variable in \eqref{eq:introduction_opt_model}, it has been proven in \cite{bradley2019communication} that the unique minimizer to the $L^2$ minimization \eqref{eq:opt_model_L2_org} is one of the minimizers to the $L^1$ minimization \eqref{eq:opt_model_L2_org}, and an explicit cheap construction to one of the $L^1$ minimizers to \eqref{eq:opt_model_L1_org} was given in  \cite{bradley2019communication}, which will be reviewed in Section~\ref{sec:scalar}. In other words, for a scalar variable $\vec{X}$ in  $L^1$ minimization \eqref{eq:opt_model_L1_org}, the best solution is the explicit construction of one particular minimizer in  \cite{bradley2019communication}, which is however not a minimizer to \eqref{eq:opt_model_L2_org}, and one gets the accuracy justification in \eqref{accuracy-proof} under $L^1$-norm. For using  $L^2$ minimization  \eqref{eq:opt_model_L2_org} for a scalar variable $\vec{X}$, in the literature there are many efficient solvers that were used, for which however the computational complexity is usually not quantified. For solving 
$L^2$ minimization  \eqref{eq:opt_model_L2_org} for a scalar variable $\vec{X}$, when using the optimal parameters given in \cite{liu2024simple},  DRS has a complexity around $80 N$, if the bad cell ratio is small, which is the case for many good schemes solving many practical problems.  

 Among all the efficient solvers considered and used for  \eqref{eq:opt_model_L2_org} for a scalar variable, only the operator splitting methods can be easily extended to enforcing the invariant domain for a vector variable. 
 For example, it was shown in \cite{bochev2013fast}
 that \eqref{eq:opt_model_L2_org} for scalar variables
 can be efficiently solved by the solver in  \cite{dai2006new}, which was designed for enforcing bounds  rather than an invariant domain like \eqref{invariant-domain}.

\subsection{Contributions and organization of the paper}

In this paper, we provide a simplified explicit formula for projecting a point onto the admissible set \eqref{invariant-domain} for gas dynamics equations such as compressible Euler and NS equations.  We further use this projection to construct 
efficient operator splitting schemes for solving 
optimization-based post-processing limiters  \eqref{eq:opt_model_L2_org}
and  \eqref{eq:opt_model_L1_org}.
To the best of our knowledge, this is the first study on constructing optimization-based limiters for directly enforcing 
the invariant domain  \eqref{invariant-domain},
and this is also the first time to use in this context a three-operator splitting scheme like DYS, which performs quite well numerically for solving $L^2$-norm limiter  \eqref{eq:opt_model_L2_org}.
  The proposed invariant-domain-preserving limiter also works in enforcing positivity of cell averages in DG methods, even if the time discretization is not SSP type, as will be shown in our numerical tests.  
\par
The rest of this paper is organized as follows. In Section~\ref{sec:scalar}, we review some existing results then explain how to design splitting methods for solving the $L^2$ and $L^1$ optimization-based limiter for scalar variables. In Section~\ref{sec:limiter_Euler}, we describe the designed splitting methods for optimization-based limiters for the enforcing invariant domain. DYS method is used for solving the $L^2$-norm invariant-domain-preserving limiter, and DRS nested with DYS can be used for  the $L^1$-norm invariant-domain-preserving limiter.
Both limiters involve the projection to the invariant domain, and formulae are given in the Appendix. Numerical tests are shown in Section~\ref{sec:numerical_experiments}. Concluding remarks are given in Section~\ref{sec:concluding_remarks}.

%%%%%%%%%%%%%%%%%%%%%%%%%%%%%%%%%%%%%%%%%%%%%%%%%%%%%%%%%%%%%%%%%%%%%%%%%%%%%%%%%%%%%%%%%%%%%%%%%%%%%%%%%%%%%%%%%%%%%%%
\section{Bound-preserving limiters for scalar conservation laws}\label{sec:scalar}
%%%%%%%%%%%%%%%%%%%%%%%%%%%%%%%%%%%%%%%%%%%%%%%%%%%%%%%%%%%%%%%%%%%%%%%%%%%%%%%%%%%%%%%%%%%%%%%%%%%%%%%%%%%%%%%%%%%%%%%

In this section, we first review the result in \cite{bradley2019communication}, then discuss splitting methods for  preserving bounds of scalar variable, i.e., the invariant domain is $G = [m, M]$, where $m$ and $M$ are the lower and upper bounds satisfied by the exact or physical solutions.  

\subsection{The problem setup and existing results for the scalar case}
 For a scalar-valued function $u$ in solving a scalar equation, a DG scheme   may violate the  bound-preserving property with cell averages out of the interval $G = [m, M]$, which may occur when using a high order scheme with a large time step size or a high order accurate implicit scheme.
The $L^2$ model \eqref{eq:opt_model_L2_org} enforces global conservation and bounds without compromising high order accuracy, see \cite{LHTZ2024FP}.
\par 
We introduce a vector $\vec{u}\in\IR^N$ to represent cell averages of the scalar-valued DG polynomial solution $U_h$, that is, the $i$-th entry of $\vec{u}$ equals $\overline{U_h}|_{K_i} = \frac{1}{\abs{K_i}}\int_{K_i} U_h$. 
Define a matrix $\vecc{A}=[1~1~\cdots~1]\in\IR^{1\times N}$ and $b = \vecc{A}\vec{u}$.
Then, the limiter \eqref{eq:opt_model_L2_org} is equivalent to the following  problem with a parameter $\alpha > 0$ in matrix-vector form:
\begin{align}\label{eq:opt_model_l2_scalar}
\min_{\vec{x}\in\IR^N}&\, \frac{1}{2\alpha}\norm{\vec{x}-\vec{u}}{2}^2 + \iota_{\Lambda_1}(\vec{x}) + \iota_{\Lambda_2}(\vec{x}),\nonumber\\
\text{where}~~& \Lambda_1 = \{\vec{x}\in\IR^N: \vecc{A}\vec{x} = b\}, \\
\text{and}~~& \Lambda_2 = \{\vec{x}\in\IR^N: x_i \in [m, M],~ \forall i = 1,\cdots, N\}.\nonumber
\end{align}
Notice that we have a freedom to choose a different $\alpha>0$, which may affect performance of solvers for \eqref{eq:opt_model_l2_scalar}.
Similarly, the $L^1$ model \eqref{eq:opt_model_L1_org} for a scalar variable with $G=[m,M]$ is equivalent to
\begin{align}\label{eq:opt_model_l1_scalar}
\min_{\vec{x}\in\IR^N}&\, \norm{\vec{x}-\vec{u}}{1} + \iota_{\Lambda_1}(\vec{x}) + \iota_{\Lambda_2}(\vec{x}),\nonumber\\
\text{where}~~& \Lambda_1 = \{\vec{x}\in\IR^N: \vecc{A}\vec{x} = b\} \\
\text{and}~~& \Lambda_2 = \{\vec{x}\in\IR^N: x_i \in [m, M],~ \forall i = 1,\cdots, N\}.\nonumber
\end{align}
By \cite[Proposition 4.1]{bradley2019communication}, the unique minimizer to the strongly convex $\ell^2$ problem \eqref{eq:opt_model_l2_scalar} also minimizes the $\ell^1$ problem  \eqref{eq:opt_model_l1_scalar}, which however has multiple minimizers.
Moreover, it is proven in \cite{bradley2019communication} that
one particular $\ell^1$ minimizer can be constructed explicitly by $\mathtt{ClipAndAssuredSum}$ (\cite[Algorithm~3.1]{bradley2019communication}):
\begin{subequations}\label{eq:ClipAndAssuredSum}
\begin{align}
\mathtt{clip}(\vec{u}) &= \min{(\max{(\vec{u}, m), M})},\\
\vec{d} &= \begin{cases}
\frac{M\vec{1} - \mathtt{clip}(\vec{u})}{MN - \vec{1}\cdot\mathtt{clip}(\vec{u})}, & \text{if}~\vec{1}\cdot\mathtt{clip}(\vec{u}) - b < 0,\\
\vec{0}, & \text{if}~\vec{1}\cdot\mathtt{clip}(\vec{u}) - b = 0,\\
\frac{\mathtt{clip}(\vec{u}) - m\vec{1}}{\vec{1}\cdot\mathtt{clip}(\vec{u}) - mN}, & \text{if}~\vec{1}\cdot\mathtt{clip}(\vec{u}) - b > 0,
\end{cases}\\
\vec{x}^\ast &= \mathtt{clip}(\vec{u}) - (\vec{1}\cdot\mathtt{clip}(\vec{u}) - b)\vec{d}.
\end{align}
\end{subequations}
If one wants an $\ell^1$ minimizer, the $\mathtt{ClipAndAssuredSum}$ is the best solution because of its low cost. On the other hand, the $\vec{x}^\ast$ in \eqref{eq:ClipAndAssuredSum} is not necessarily an $\ell^2$ minimizer. Moreover, it seems quite difficult to extend the proof and results in  \cite{bradley2019communication}  to solving \eqref{eq:opt_model_L1_org} in the context of the invariant domain \eqref{invariant-domain}.  

We emphasize that the minimizers to \eqref{eq:opt_model_l1_scalar} are usually not sparse, unlike many other $\ell^1$ minimization problems in other applications. In terms of sparsity,  the minimizers to $\ell^1$ minimization \eqref{eq:opt_model_l1_scalar} are often only marginally better than the minimizer to  $\ell^2$ minimization \eqref{eq:opt_model_l2_scalar}.

\begin{remark}
The objective function of the $\ell^1$ minimization problem \eqref{eq:opt_model_l1_scalar} is convex but not strongly convex, which implies the non-uniqueness of the minimizer.
To see this, let us consider the following simple example with four variables:
\begin{subequations}\label{eq:simple_model}
\begin{align}
\min_{x_1, x_2, x_3, x_4}~ &\abs{x_1 - 1} + \abs{x_2 - 1} + \abs{x_3 - 2} + \abs{x_4 - 2.1} \label{eq:simple_model_1a}\\
\mathrm{subject~to}~ &x_1 + x_2 + x_3 + x_4 = 6.1 
~\mathrm{and}~ x_i \in [1,2]~\mathrm{for}~i = 1, \cdots, 4. \label{eq:simple_model_1b}
\end{align}
\end{subequations}
The vector $\transpose{[x_1, x_2, x_3, x_4]} = \transpose{[1, 1.1, 2, 2]}$ is a minimizer, which yields the objective function value $0.2$. And $\transpose{[1.05, 1.05, 2, 2]}$ achieves the same objective value, i.e., it is also a minimizer.
When interpreting the values $\transpose{[1,1,2,2.1]}$ in \eqref{eq:simple_model_1a} as cell averages, the first minimizer $\transpose{[1, 1.1, 2, 2]}$ modifies fewer cells (two cells) than the second minimizer $\transpose{[1.05, 1.05, 2, 2]}$, which modifies three.
In this simple example \eqref{eq:simple_model}, both $\mathtt{ClipAndAssuredSum}$ and splitting methods listed in this section produce the second minimizer, which is not the ``sparsest solution'' in the sense of modifying the fewest number of cells.
\end{remark}

\subsection{Davis-Yin splitting for $\ell^2$ minimization}

Both sets $\Lambda_1$ and $\Lambda_2$ in \eqref{eq:opt_model_l2_scalar} are convex and closed, which implies the objective function of \eqref{eq:opt_model_l2_scalar} is a proper closed strongly convex function and thus has a unique minimizer.
In \cite{LHTZ2024FP}, DRS \eqref{eq:DR_algorithm} was used to solve \eqref{eq:opt_model_l2_scalar} with optimal parameters. Here we provide an easier alternative approach based on DYS method, which serves as an important component in constructing the splitting method for $L^1$ model in this paper. 
To apply DYS \eqref{eq:DY_algorithm}, we choose 
\begin{align}\label{eq:scalar_func_fgh}
f(\vec{x}) = \iota_{\Lambda_1}(\vec{x}), \quad 
g(\vec{x}) = \iota_{\Lambda_2}(\vec{x}), \quad\text{and}\quad
h(\vec{x}) = \frac{1}{2\alpha}\norm{\vec{x}-\vec{u}}{2}^2. 
\end{align}
Let $\vecc{A}^+ = \transpose{\vecc{A}}(\vecc{A}\transpose{\vecc{A}})^{-1}$ denote the pseudo inverse of the matrix $\vecc{A}$.  The proximal operator for the function $f$ defined in \eqref{eq:scalar_func_fgh} is given by
\begin{align}\label{eq:scalar_DY_prox_f}
\prox_f^\gamma(\vec{x}) = \vecc{A}^+(b - \vecc{A}\vec{x}) + \vec{x}.
\end{align}
The proximal operator of an indicator of a set is the Euclidean projection onto that set. Thus, the $\prox_g^\gamma$ projects each point $x_i$, $i=1,\cdots,N$, onto the admissible set $[m,M]$, i.e., it is the cut-off operation:
\begin{align}\label{eq:scalar_DY_prox_g}
[\prox_g^\gamma(\vec{x})]_i = \min{(\max{(x_i, m)},M)},\quad \forall i = 1, \cdots, N.
\end{align}
Here, the subscript $i$ denotes the $i$-th component in corresponding vector.
In addition, we have
\begin{align}\label{eq:scalar_DY_prox_h}
\grad{h} = \frac{1}{\alpha}(\vec{x} - \vec{u})
\quad\text{and}\quad
\prox_h^\gamma(\vec{x}) =
\frac{\alpha}{\gamma+\alpha}\vec{x}
+
\frac{\gamma}{\gamma+\alpha}\vec{u}
\end{align}
Substituting \eqref{eq:scalar_DY_prox_f}-\eqref{eq:scalar_DY_prox_h} into \eqref{eq:DY_algorithm} and taking the step size $\gamma =\frac{1}{L}=\alpha$, where $L=\frac{1}{\alpha}$ is the Lipschitz constant of $\grad{h}$, we obtain:
\begin{subequations}\label{eq:DY_l2_implement_formula}
\begin{align}
\vec{X}^{k+1/2} &= \min{(\max{(\vec{Z}^k, m)},M)}, \label{eq:DY_l2_implement_formula_1}\\
\vec{Y}^{k+1} &= \vec{X}^{k+1/2} - \vec{Z}^{k} + \vec{u},\label{eq:DY_l2_implement_formula_2}\\
\vec{X}^{k+1} &= \vecc{A}^+(b - \vecc{A}\vec{Y}^{k+1}) + \vec{Y}^{k+1},\label{eq:DY_l2_implement_formula_3}\\
\vec{Z}^{k+1} &= \vec{Z}^{k} + \vec{X}^{k+1} - \vec{X}^{k+1/2}.\label{eq:DY_l2_implement_formula_4}
\end{align}
\end{subequations}
In \eqref{eq:DY_l2_implement_formula_1}, the $\min$ and $\max$ denote entrywise operation.
We initialize the algorithm above by setting $\vec{Z}^{0} = \vec{u}$. To terminate the iteration, we can choose a small positive tolerance $\epsilon$ near machine zero and employ the stopping criterion $\norm{\vec{Z}^{k+1}-\vec{Z}^{k}}{2h} < \epsilon$, where norm $\norm{\cdot}{2h} = h^{d/2}\norm{\cdot}{2}$.
\par
A practical advantage of the DYS \eqref{eq:DY_l2_implement_formula} is that it does not require any parameter tuning. Although an optimal parameter selection guideline exists for DRS in the scalar cases \cite{liu2024simple}, no established principle is currently known for choosing parameters in a DRS when applied to an invariant domain \eqref{invariant-domain}. In such cases, a parameter-free DYS is preferable. In practice, DYS with the step size
$\gamma=\frac{1}{L}$ is empirically quite efficient, often converges to machine zero accuracy within a few iterations.
See a numerical comparison of DYS with other similar methods for solving \eqref{eq:opt_model_l2_scalar} in \cite{anshika2024three}. 
Thus, for compressible Euler equations, we propose using the DYS method to solve the $L^2$ minimization problem \eqref{eq:opt_model_L2_org}, as will be shown in Section~\ref{sec:limiter_Euler}. 

\subsection{Douglas-Rachford splitting for $\ell^1$ minimization only for scalar variables}

We first consider the DRS method \eqref{eq:DR_algorithm} to solve $\ell^1$ minimization  \eqref{eq:opt_model_l1_scalar}, with functions $g$ and $h$ chosen as follows:
\begin{align}\label{eq:scalar_l1_gh_alter}
g(\vec{x}) = \norm{\vec{x}-\vec{u}}{1} + \iota_{\Lambda_2}(\vec{x}) \quad\text{and}\quad
h(\vec{x}) = \iota_{\Lambda_1}(\vec{x}).
\end{align}
The proximal of $h$ above is the same as \eqref{eq:scalar_DY_prox_f}. As shown in \ref{appendix:compute_prox},
the proximal operator  of $g$ is,
\begin{subequations}
\begin{align}
[\prox_g^\gamma(\vec{x})]_i &= \max\{\min\{u_i + \mathrm{S}_\gamma(x_i-u_i), M\}, m\},\\
\mathrm{S}_\gamma(a) &= \mathrm{sgn}(a)\max\{\abs{a}-\gamma, 0\}.
\end{align}
\end{subequations}
The DRS method can easily be written out with these explicit formulae of proximal operators, and it converges with any positive step size $\gamma>0$ and any relaxation parameter $\lambda\in(0,2)$.

However, when extending from $G=[m, M]$ to \eqref{invariant-domain} for systems of conservation laws such as the Euler equations,
the formula above and its derivation no longer hold, and there is no simple analytical expression for the proximal operator when combining the $\ell^1$ norm with an invariant-domain-preserving constraint \eqref{invariant-domain}.

\subsection{Douglas-Rachford splitting combined with Davis-Yin splitting for $\ell^1$ minimization}
 
We now consider the DRS method \eqref{eq:DR_algorithm} to solve $\ell^1$ minimization  \eqref{eq:opt_model_l1_scalar}, with functions $g$ and $h$ chosen as follows:
\begin{align}\label{eq:scalar_l1_gh}
g(\vec{x}) = \norm{\vec{x}-\vec{u}}{1} \quad\text{and}\quad
h(\vec{x}) = \iota_{\Lambda_1}(\vec{x}) + \iota_{\Lambda_2}(\vec{x}).
\end{align}
Define a shrinkage operator $\mathrm{S}_\gamma(x) = \mathrm{sgn}(x)\max\{\abs{x}-\gamma, 0\}$. Associated with the function $g$ above, the proximal is given by
\begin{align}\label{eq:scalar_l1_prox_g}
[\prox_g^\gamma(\vec{x})]_i 
= u_i + \mathrm{S}_\gamma(x_i - u_i)
= u_i + \mathrm{sgn}(x_i - u_i)\max\{\abs{x_i - u_i} - \gamma, 0\},
\end{align}
where the subscript $i$ denotes the $i$-th component in corresponding vector.
The proximal of the function $h$ in \eqref{eq:scalar_l1_gh} lacks a tractable analytical form. To overcome this difficulty, we design a nested approach, where the inner DYS iteration computes $\prox_h^\gamma$ numerically. 
Notice that we have the following relation by definition of a proximal operator,  
\begin{align}\label{eq:scalar_l1_inner_problem}
\prox_h^\gamma(\vec{x})=\operatorname{argmin}_{\vec{z}\in\IR^N}\, \frac{1}{2\gamma}\norm{\vec{z}-\vec{x}}{2}^2 + \iota_{\Lambda_1}(\vec{z}) + \iota_{\Lambda_2}(\vec{z}),
\end{align}
which is nothing but the   problem \eqref{eq:opt_model_l2_scalar}.
 Therefore, to implement the proximal operator \eqref{eq:scalar_l1_inner_problem},  it can be solved numerically by applying DYS \eqref{eq:DY_l2_implement_formula}. After obtaining $\prox_h^\gamma$, we use the DRS \eqref{eq:DR_algorithm} to compute a minimizer of \eqref{eq:opt_model_l1_scalar}.
We initialize the algorithm above by setting $\vec{Y}^{0} = \vec{u}$ and $\vec{X}^{0} = \prox_g^\gamma(\vec{Y}^{0})$. To terminate the iteration, we choose a small positive tolerance $\epsilon$ and employ the stopping criterion $\norm{\vec{Y}^{k+1}-\vec{Y}^{k}}{2h} < \epsilon$.

In other words, this is a nested method with DRS as an outer loop and DYS as the inner loop for the subproblem \eqref{eq:scalar_l1_inner_problem}. This approach can be easily extended to the invariant domain \eqref{invariant-domain}.

%%%%%%%%%%%%%%%%%%%%%%%%%%%%%%%%%%%%%%%%%%%%%%%%%%%%%%%%%%%%%%%%%%%%%%%%%%%%%%%%%%%%%%%%%%%%%%%%%%%%%%%%%%%%%%%%%%%%%%%
\section{An efficient implementation of invariant-domain-preserving limiters by splitting methods}\label{sec:limiter_Euler}
%%%%%%%%%%%%%%%%%%%%%%%%%%%%%%%%%%%%%%%%%%%%%%%%%%%%%%%%%%%%%%%%%%%%%%%%%%%%%%%%%%%%%%%%%%%%%%%%%%%%%%%%%%%%%%%%%%%%%%%
As an example, consider solving compressible Euler equations in a bounded spatial domain $\Omega\subset\IR^d$ over the time interval $t \in [0,T]$. Let $\vec{U} = \transpose{[\rho,\transpose{\vec{m}},E]}$ denote the conservative variables for density, momentum, and total energy. The  compressible Euler equations are
\begin{align}\label{eq:compressible_Euler}
\partial_t{\vec{U}} + \div{\vec{F}^\mathrm{a}}(\vec{U}) = \vec{0},
\quad\text{where}\quad 
\vec{F}^\mathrm{a} = \begin{bmatrix}
\vec{m} \\
\frac{1}{\rho}\vec{m}\otimes\vec{m}+ p \vecc{I}\\
(E+p)\frac{\vec{m}}{\rho}
\end{bmatrix}.
\end{align}
The total energy is $E = \rho e + \frac{\norm{\vec{m}}{2}^2}{2\rho}$, where $e$ denotes the internal energy. %and $\norm{\cdot}{2}$ is the vector $2$-norm.
For the ideal gas, the equation of state is $p = (\gamma-1)\rho e$, where $\gamma>1$ is a constant, e.g., $\gamma = 1.4$ for air.
When vacuum does not occur, a physically meaningful solution to \eqref{eq:compressible_Euler} should have positive density and positive internal energy, namely it should stay in the admissible set $G$.
We will discuss optimization-based limiters for enforcing a numerical invariant domain $G^\varepsilon$ in \eqref{invariant-domain}.

\subsection{An $\ell^2$ minimization  invariant-domain-preserving limiter for  cell averages}
High order explicit schemes with large time steps, as well as implicit schemes, may produce negative cell averages. 
We propose the optimization-based approach that incorporates preserving global conservation and invariant domain as constraints to postprocess the cell averages in DG solution. 
\par
In limiter \eqref{eq:opt_model_L2_org}, the vector-valued piecewise constant polynomial $\vec{X}_h$ minimizes the $L^2$ distance to the cell average of the DG polynomial $\vec{U}_h = \transpose{[\rho_h, \transpose{\vec{m}_h}, E_h]}$. 
We introduce a matrix $\overline{\vecc{U}} \in \IR^{N\times(2+d)}$ to represent cell averages of the DG polynomial $\vec{U}_h$, that is, the $i$-th row of $\overline{\vecc{U}}$ equals \begin{align}
\transpose{\overline{\vec{U}_h}}|_{K_i} = \begin{bmatrix}
\displaystyle \frac{1}{\abs{K_i}}\int_{K_i} \rho_h~~ & 
\displaystyle \frac{1}{\abs{K_i}}\int_{K_i} \transpose{\vec{m}_h}~~ & 
\displaystyle \frac{1}{\abs{K_i}}\int_{K_i} E_h
\end{bmatrix}.
\end{align}
Then, model \eqref{eq:opt_model_L2_org} is equivalent to the following unconstrained minimization problem in matrix-vector form:
\begin{align}\label{eq:invariant_domain_limiter2}
\min_{\vecc{X}\in\IR^{N\times(2+d)}}& \frac{1}{2\alpha}\norm{\vecc{X} - \overline{\vecc{U}}}{F}^2 + \iota_{\Lambda_1}(\vecc{X}) + \iota_{\Lambda_2}(\vecc{X}),\nonumber\\
\text{where}~~&
\Lambda_1 = \{\vecc{X}\!:~ \vecc{A}\vecc{X} = \transpose{\vec{b}}\}\\
\text{and}~~&
\Lambda_2 = \{\vecc{X}\!:~ \text{the $i$-th row}~\transpose{\vecc{X}_i}\in G^\varepsilon,~ \forall i\}.\nonumber
\end{align}
Here, $\norm{\cdot}{F}$ is the matrix Frobenius norm.
On a uniform mesh, we have $\vecc{A} = [1,1,\cdots,1]\in\IR^{1\times N}$ and $\transpose{\vec{b}} = \vecc{A}\overline{\vecc{U}}$. 
The closed convex sets $\Lambda_1$ and $\Lambda_2$ represent global conservation and invariant-domain-preserving constraints, respectively. 
The functions $\iota_{\Lambda_1}$ and $\iota_{\Lambda_2}$ are convex, indicating that \eqref{eq:invariant_domain_limiter2} is a strongly convex minimization problem. Therefore, the solution of \eqref{eq:invariant_domain_limiter2} is unique.
\par
When all cell averages of the exact solution belong to the numerical admissible set, provided that $\varepsilon$ is sufficiently small.
The limiter \eqref{eq:invariant_domain_limiter2} improves the accuracy of the DG solution in the following sense:
\begin{theorem}
\label{theorem-limiter}
Let $\vecc{X}^\ast$ denote the minimizer of \eqref{eq:invariant_domain_limiter2} and let $\overline{\vecc{U}^\mathrm{exact}}\in \IR^{N\times (2+d)}$ be a matrix storing the cell averages of an exact solution. If the DG solution has the same integrals over the whole domain as the exact solution, and the DG cell averages violate the invariant domain (i.e., $\overline{\vecc{U}}\notin \Lambda_2$), then the minimizer $\vecc{X}^\ast$ improves the accuracy:
\begin{align}
\norm{\vecc{X}^\ast - \overline{\vecc{U}^\mathrm{exact}}}{F} < \norm{\overline{\vecc{U}} - \overline{\vecc{U}^\mathrm{exact}}}{F}.
\end{align}
\end{theorem}
\begin{proof}
For two matrices $\vecc{B} = [b]_{ij}$ and $\vecc{C} = [c]_{ij}$ of same sizes, define the inner product $\vecc{B}:\vecc{C} = \sum_{ij} b_{ij}c_{ij}$.
The sets $\Lambda_1$ and $\Lambda_2$ are convex and closed, thus $\Lambda_1\cap\Lambda_2$ is a convex closed set. Thus, $\overline{\vecc{U}^\mathrm{exact}}$ and $\vecc{X}^\ast$ belongs to $\Lambda_1\cap\Lambda_2$, which implies 
\begin{align}
\eta \overline{\vecc{U}^\mathrm{exact}} + (1-\eta)\vecc{X}^\ast \in \Lambda_1\cap\Lambda_2,\quad \forall\eta\in[0,1].
\end{align}
Define
\begin{align}
\phi(\eta) &:= \norm{\overline{\vecc{U}} - (\eta \overline{\vecc{U}^\mathrm{exact}} + (1-\eta)\vecc{X}^\ast)}{F}^2 =\eta^2\norm{\overline{\vecc{U}^\mathrm{exact}} - \vecc{X}^\ast}{F}^2 - 2\eta(\overline{\vecc{U}} - \vecc{X}^\ast):(\overline{\vecc{U}^\mathrm{exact}} - \vecc{X}^\ast) + \norm{\overline{\vecc{U}} - \vecc{X}^\ast}{F}^2.
\end{align}
If $\norm{\overline{\vecc{U}^\mathrm{exact}} - \vecc{X}^\ast}{F} = 0$, then $\norm{\overline{\vecc{U}^\mathrm{exact}}-\vecc{X}^\ast}{F} =0< \norm{\overline{\vecc{U}^\mathrm{exact}}-\overline{\vecc{U}}}{F}$, because
$\overline{\vecc{U}^\mathrm{exact}}\neq \overline{\vecc{U}}$ due to $\overline{\vecc{U}}\notin \Lambda_2$. Otherwise, it is obvious that $\phi(\eta)$ is a quadratic function with respect to $\eta$. From \eqref{eq:invariant_domain_limiter2}, we know $\vecc{X}^\ast$ minimizes $\norm{\overline{\vecc{U}}-\vecc{X}}{F}^2$ for all $\vecc{X}\in \Lambda_1\cap\Lambda_2$. Thus, $\phi(\eta)$ achieves its minimum at $\eta = 0$, which implies
$(\vecc{X}^\ast - \overline{\vecc{U}}):(\overline{\vecc{U}^\mathrm{exact}}- \vecc{X}^\ast) \geq 0.$
Therefore, we have 
\begin{align}
\norm{\overline{\vecc{U}^\mathrm{exact}}-\overline{\vecc{U}}}{F}^2
&= \norm{\overline{\vecc{U}^\mathrm{exact}}-\vecc{X}^\ast + \vecc{X}^\ast - \overline{\vecc{U}}}{F}^2 \nonumber\\
&= \norm{\overline{\vecc{U}^\mathrm{exact}}-\vecc{X}^\ast}{F}^2 + 2(\vecc{X}^\ast - \overline{\vecc{U}}):(\overline{\vecc{U}^\mathrm{exact}}- \vecc{X}^\ast) + \norm{\vecc{X}^\ast - \overline{\vecc{U}}}{F}^2
> \norm{\overline{\vecc{U}^\mathrm{exact}}-\vecc{X}^\ast}{F}^2.
\end{align} 
\end{proof}
 
Finally, we present an equivalent reformulation of problem \eqref{eq:invariant_domain_limiter2} to facilitate the design of splitting methods.
Partition  columns in the matrix $\vecc{X} = [\vec{\rho}, \vec{m}_1, \cdots, \vec{m}_d, \vec{E}]$ such that, the first column $\vec{\rho}$ contains the cell averages of density, the next $d$ columns $\vec{m}_1$ to $\vec{m}_d$ represent the cell averages of the momentum components, and the last column $\vec{E}$ corresponds to the cell averages of total energy. Similarly, we define the matrix $\overline{\vecc{U}} = [\vec{u}, \vec{v}_1, \cdots, \vec{v}_d, \vec{w}]$ and let $\transpose{\vec{b}} = [b_\rho, b_{m_1}, \cdots, b_{m_d}, b_E]$. Then, the minimization problem \eqref{eq:invariant_domain_limiter2} is equivalent to
\begin{align}\label{eq:invariant_domain_limiter3}
\min_{\vecc{X}}&~ \frac{1}{2\alpha}\Big(\norm{\vec{\rho} - \vec{u}}{2}^2 + \sum_{i=1}^{d}\norm{\vec{m}_i - \vec{v}_i}{2}^2 + \norm{\vec{E} - \vec{w}}{2}^2\Big) + \iota_{\Lambda_1}(\vec{\rho}, \vec{m}_1, \cdots, \vec{m}_d, \vec{E}) + \iota_{\Lambda_2}(\vec{\rho}, \vec{m}_1, \cdots, \vec{m}_d, \vec{E}),\nonumber\\
\text{where}&~ \Lambda_1 = \{\vecc{X}\!:\, \vecc{A}\vec{\rho} = b_\rho,~ \vecc{A}\vec{m}_1 = b_{m_1},~ \cdots,~ \vecc{A}\vec{m}_d = b_{m_d},~ \vecc{A}\vec{E} = b_E\}\\
\text{and}&~ \Lambda_2 = \{\vecc{X}\!:\,\transpose{[\rho_i, m_{1i}, \cdots, m_{di}, E_i]} \in G^\varepsilon,~ \forall i\}.\nonumber
\end{align}
\begin{remark}
We emphasize that the work in \cite{liu2024simple,LZ2024CNS,LHTZ2024FP}  focuses solely on a scalar variable and is therefore applicable only to compressible NS equations with Strang splitting \cite{LZ2022CNS,LZ2024CNS}, but not for Euler equations. The proposed limiter in this section is much more general than the method in \cite{LZ2024CNS}. In contrast, the invariant-domain-preserving limiter \eqref{eq:invariant_domain_limiter2} directly preserves the convex invariant domain thus is more broadly applicable in compressible flow simulations.
\end{remark}

\subsection{Splitting methods for $\ell^2$ minimization limiter}
An efficient implementation of splitting methods to solve \eqref{eq:invariant_domain_limiter3} requires computing proximal operators easily and efficiently. We split the objective function in a manner that facilitates the derivation of explicit  formulae. 
\begin{itemize}
\item To apply DRS method \eqref{eq:DR_algorithm}, we choose 
\begin{subequations}
\begin{align}
g(\vecc{X}) &= \frac{1}{2\alpha}\Big(\norm{\vec{\rho} - \vec{u}}{2}^2 + \sum_{i=1}^{d}\norm{\vec{m}_i - \vec{v}_i}{2}^2 + \norm{\vec{E} - \vec{w}}{2}^2\Big) + \iota_{\Lambda_1}(\vec{\rho}, \vec{m}_1, \cdots, \vec{m}_d, \vec{E}),\label{eq:euler_DR_g}\\
h(\vecc{X}) &= \iota_{\Lambda_2}(\vec{\rho}, \vec{m}_1, \cdots, \vec{m}_d, \vec{E}).\label{eq:euler_DR_h}
\end{align}
\end{subequations}
\item To apply DYS method \eqref{eq:DY_algorithm}, we choose
\begin{subequations}
\begin{align}
f(\vecc{X}) &= \iota_{\Lambda_1}(\vec{\rho}, \vec{m}_1, \cdots, \vec{m}_d, \vec{E}), \label{eq:euler_DY_f}\\
g(\vecc{X}) &= \iota_{\Lambda_2}(\vec{\rho}, \vec{m}_1, \cdots, \vec{m}_d, \vec{E}), \label{eq:euler_DY_g}\\
h(\vecc{X}) &= \frac{1}{2\alpha}\Big(\norm{\vec{\rho} - \vec{u}}{2}^2 + \sum_{i=1}^{d}\norm{\vec{m}_i - \vec{v}_i}{2}^2 + \norm{\vec{E} - \vec{w}}{2}^2\Big).\label{eq:euler_DY_h}
\end{align}
\end{subequations}
\end{itemize}
Let us derive the proximal associated with function $g$ in \eqref{eq:euler_DR_g}.
The proximal associated with function $f$ in \eqref{eq:euler_DY_f} can be derived similarly or simply set $\alpha = +\infty$.
By definition of the proximal operator, for any given $\vecc{X} = [\vec{\rho}, \vec{m}_1, \cdots, \vec{m}_d, \vec{E}]$, we need to find $\vecc{Z} = [\vec{\xi}, \vec{\eta}_1, \cdots, \vec{\eta}_d, \vec{\zeta}]$ to minimize the following function:
\begin{multline}\label{eq:minimize_problem_prox_g}
\frac{\gamma}{2\alpha}\Big(\norm{\vec{\xi}-\vec{u}}{2}^2 + \sum_{i=1}^d\norm{\vec{\eta}_i-\vec{v}_i}{2}^2 + \norm{\vec{\zeta}-\vec{w}}{2}^2\Big) + \gamma\iota_{\Lambda_1}(\vec{\xi},\vec{\eta}_1,\cdots,\vec{\eta}_d,\vec{\zeta}) \\
+ \frac{1}{2}\norm{\vec{\xi}-\vec{\rho}}{2}^2 + \frac{1}{2}\sum_{i=1}^{d}\norm{\vec{\eta}_i-\vec{m}_i}{2}^2 + \frac{1}{2}\norm{\vec{\zeta}-\vec{E}}{2}^2.
\end{multline}
By definition of the indicator function $\iota_{\Lambda_1}$, minimizing \eqref{eq:minimize_problem_prox_g} is equivalent to solving the following decoupled constraint optimization problems:
\begin{subequations}
\begin{align}
\min_{\vec{\xi}\in\IR^N}&\, \frac{\gamma}{2\alpha}\norm{\vec{\xi}-\vec{u}}{2}^2 + \frac{1}{2}\norm{\vec{\xi}-\vec{\rho}}{2}^2 \quad\text{subject~to}\quad \vecc{A}\vec{\xi} = b_\rho, \label{eq:derive_prox_min_1}\\
\min_{\vec{\eta}_i\in\IR^N}&\, \frac{\gamma}{2\alpha}\norm{\vec{\eta}_i-\vec{v}_i}{2}^2 + \frac{1}{2}\norm{\vec{\eta}_i-\vec{m}_i}{2}^2 \quad\text{subject~to}\quad \vecc{A}\vec{\eta}_i = b_{m_i},\quad\text{for}~i=1,\cdots,d, \label{eq:derive_prox_min_2}\\
\min_{\vec{\zeta}\in\IR^N}&\, \frac{\gamma}{2\alpha}\norm{\vec{\zeta}-\vec{w}}{2}^2 + \frac{1}{2}\norm{\vec{\zeta}-\vec{E}}{2}^2 \quad\text{subject~to}\quad \vecc{A}\vec{\zeta} = b_E. \label{eq:derive_prox_min_3}
\end{align}
\end{subequations}
Let us start with \eqref{eq:derive_prox_min_1}. Define a Lagrange multiplier as follows:
\begin{align}
\mathcal{L} = \frac{\gamma}{2\alpha}\norm{\vec{\xi}-\vec{u}}{2}^2 + \frac{1}{2}\norm{\vec{\xi}-\vec{\rho}}{2}^2 + \lambda(\vecc{A}\vec{\xi} - b_\rho).
\end{align}
Notice, the matrix $\vecc{A} = [1~1~\cdots~1]\in\IR^{1\times N}$. Take partial derivatives, we get
\begin{subequations}
\begin{align}
\frac{\partial\mathcal{L}}{\partial \vec{\xi}} &= \Big(\frac{\gamma}{\alpha} + 1\Big)\vec{\xi} - \frac{\gamma}{\alpha}\vec{u} - \vec{\rho} + \lambda\transpose{\vecc{A}} = \vec{0}, \label{eq:prox_g_1}\\
\frac{\partial\mathcal{L}}{\partial \lambda} &= \vecc{A}\vec{\xi} - b_\rho = 0. \label{eq:prox_g_2}
\end{align}
\end{subequations}
Left multiply $\vecc{A}$ on both side of \eqref{eq:prox_g_1}, from the constraints $\vecc{A}\vec{\xi} = b_\rho$ and $\vecc{A}\vec{u} = b_\rho$, we have
\begin{align}
b_\rho - \vecc{A}\vec{\rho} + \lambda\vecc{A}\transpose{\vecc{A}} = 0
\quad\Rightarrow\quad
\lambda = (\vecc{A}\transpose{\vecc{A}})^{-1}(\vecc{A}\vec{\rho} - b_\rho).
\end{align}
Substituting the expression for $\lambda$ above into \eqref{eq:prox_g_1} and recalling the pseudo inverse of the matrix $\vecc{A}$ is given by $\vecc{A}^+ = \transpose{\vecc{A}}(\vecc{A}\transpose{\vecc{A}})^{-1}$. We obtain
\begin{align}
\vec{\xi} = \frac{\alpha}{\gamma + \alpha}\big(\vecc{A}^+(b_\rho - \vecc{A}\vec{\rho}) + \vec{\rho}\big) + \frac{\gamma}{\gamma + \alpha}\vec{u}.
\end{align}
The \eqref{eq:derive_prox_min_2} and \eqref{eq:derive_prox_min_3} can be addressed similarly.
Therefore, associated with the function $g$ in \eqref{eq:euler_DR_g}, the proximal $\prox_g^\gamma$ maps
\begin{subequations}
\begin{align}
\vec{\rho} \quad&\rightarrow\quad \frac{\alpha}{\gamma + \alpha}\big(\vecc{A}^+(b_\rho - \vecc{A}\vec{\rho}) + \vec{\rho}\big) + \frac{\gamma}{\gamma + \alpha}\vec{u},\\
\vec{m}_i \quad&\rightarrow\quad \frac{\alpha}{\gamma + \alpha}\big(\vecc{A}^+(b_{m_i} - \vecc{A}\vec{m}_i) + \vec{m}_i\big) + \frac{\gamma}{\gamma + \alpha}\vec{v}_i,\quad\text{for}~i=1,\cdots,d,\\
\vec{E} \quad&\rightarrow\quad \frac{\alpha}{\gamma + \alpha}\big(\vecc{A}^+(b_E - \vecc{A}\vec{E}) + \vec{E}\big) + \frac{\gamma}{\gamma + \alpha}\vec{w}.
\end{align}
\end{subequations}
Notice, the proximal of an indicator of a set is the Euclidean projection onto that set. To this end, we need the projection point onto the numerical admissible set $G^\varepsilon$, i.e., computing the proximal for $h$ in \eqref{eq:euler_DR_h} an $g$ in \eqref{eq:euler_DY_g}. 
Such a projection is by no means trivial. For the ease of presentation, the projection formula to $G^\varepsilon$ and its derivation are given in the Appendix.

\subsection{An $\ell^1$ invariant-domain-preserving cell average limiter}
Consider cell averages $\overline{\vecc{U}} = [\vec{u}, \vec{v}_1, \cdots, \vec{v}_d, \vec{w}]$ of a DG polynomial.
For compressible Euler equations, the $L^1$ limiter \eqref{eq:opt_model_L1_org} is equivalent to the following unconstrained optimization problem in matrix-vector form:
\begin{align}\label{eq:opt_model_l1_Euler}
\min_{\vecc{X}}&~ \norm{\vec{\rho} - \vec{u}}{1} + \sum_{i=1}^{d}\norm{\vec{m}_i - \vec{v}_i}{1} + \norm{\vec{E} - \vec{w}}{1} + \iota_{\Lambda_1}(\vec{\rho}, \vec{m}_1, \cdots, \vec{m}_d, \vec{E}) + \iota_{\Lambda_2}(\vec{\rho}, \vec{m}_1, \cdots, \vec{m}_d, \vec{E}),\nonumber\\
\text{where}&~ \Lambda_1 = \{\vecc{X}\!:\, \vecc{A}\vec{\rho} = b_\rho,~ \vecc{A}\vec{m}_1 = b_{m_1},~ \cdots,~ \vecc{A}\vec{m}_d = b_{m_d},~ \vecc{A}\vec{E} = b_E\}\\
\text{and}&~ \Lambda_2 = \{\vecc{X}\!:\,\transpose{[\rho_i, m_{1i}, \cdots, m_{di}, E_i]} \in G^\varepsilon,~ \forall i\}.\nonumber
\end{align}
The objective function of \eqref{eq:opt_model_l1_Euler} is convex but not strongly convex, which implies the non-uniqueness of the minimizer.
We consider the DRS method \eqref{eq:DR_algorithm} to solve above minimization problem, where functions $g$ and $h$ are chosen as follows:
\begin{subequations}\label{eq:euler_l1_gh}
\begin{align}
g(\vecc{X}) &= \norm{\vec{\rho} - \vec{u}}{1} + \sum_{i=1}^{d}\norm{\vec{m}_i - \vec{v}_i}{1} + \norm{\vec{E} - \vec{w}}{1}\label{eq:euler_l1_gh_a}\\
h(\vecc{X}) &= \iota_{\Lambda_1}(\vec{\rho}, \vec{m}_1, \cdots, \vec{m}_d, \vec{E}) + \iota_{\Lambda_2}(\vec{\rho}, \vec{m}_1, \cdots, \vec{m}_d, \vec{E}).\label{eq:euler_l1_gh_b}
\end{align}
\end{subequations}
Let us derive the proximal associated with function $g$ in \eqref{eq:euler_l1_gh_a}. By definition of the proximal operator, for any given $\vecc{X} = [\vec{\rho}, \vec{m}_1, \cdots, \vec{m}_d, \vec{E}]$, we need to find $\vecc{Z} = [\vec{\xi}, \vec{\eta}_1, \cdots, \vec{\eta}_d, \vec{\zeta}]$ to minimize the following function:
\begin{align}\label{eq:euler_l1_minimize_problem_prox_g}
\gamma\Big(\norm{\vec{\xi}-\vec{u}}{1} + \sum_{i=1}^d\norm{\vec{\eta}_i-\vec{v}_i}{1} + \norm{\vec{\zeta}-\vec{w}}{1}\Big)
+ \frac{1}{2}\Big(\norm{\vec{\xi}-\vec{\rho}}{2}^2 + \sum_{i=1}^{d}\norm{\vec{\eta}_i-\vec{m}_i}{2}^2 + \norm{\vec{\zeta}-\vec{E}}{2}^2\Big).
\end{align}
Minimizing the function above is equivalent to solving the following decoupled optimization problems:
\begin{subequations}
\begin{align}
\min_{\vec{\xi}\in\IR^N}&\, \gamma\norm{\vec{\xi}-\vec{u}}{1} + \frac{1}{2}\norm{\vec{\xi}-\vec{\rho}}{2}^2\,, \label{eq:euler_l1_derive_prox_min_1}\\
\min_{\vec{\eta}_i\in\IR^N}&\, \gamma\norm{\vec{\eta}_i-\vec{v}_i}{1} + \frac{1}{2}\norm{\vec{\eta}_i-\vec{m}_i}{2}^2\,, \quad\text{for}~i=1,\cdots,d, \label{eq:euler_l1_derive_prox_min_2}\\
\min_{\vec{\zeta}\in\IR^N}&\, \gamma\norm{\vec{\zeta}-\vec{w}}{1} + \frac{1}{2}\norm{\vec{\zeta}-\vec{E}}{2}^2\,, \label{eq:euler_l1_derive_prox_min_3}
\end{align}
\end{subequations}
for which solutions were established in Section~\ref{sec:scalar}, see \eqref{eq:scalar_l1_gh} and \eqref{eq:scalar_l1_prox_g}.
Recall the shrinkage operator $\mathrm{S}_\gamma(x) = \mathrm{sgn}(x)\max\{\abs{x}-\gamma, 0\}$. Associated with the function $g$ in \eqref{eq:euler_l1_gh_a}, the proximal $\prox_g^\gamma$ maps
\begin{subequations}
\begin{align}
\rho_i &\rightarrow u_i + \mathrm{S}_\gamma(\rho_i - u_i), \\
m_{1i} &\rightarrow v_{1i} + \mathrm{S}_\gamma(m_{1i} - v_{1i}), ~\quad~ \cdots, ~\quad~
m_{di} \rightarrow v_{di} + \mathrm{S}_\gamma(m_{di} - v_{di}), \\
E_i &\rightarrow w_i + \mathrm{S}_\gamma(E_i - w_i),
\end{align}
\end{subequations}
where the subscript $i$ denotes the $i$-th component in corresponding vector.
The proximal of the function $h$ in \eqref{eq:euler_l1_derive_prox_min_2} lacks a tractable analytical form. To overcome this difficulty, we design a nested approach, where the inner DYS iteration computes $\prox_h^\gamma$ numerically. 
By definition of a proximal operator, the $\prox_h^\gamma$ gives the solution to the following minimization problem: given $\vecc{X}\in\IR^{N\times(2+d)}$, find $\vecc{Z}\in\IR^{N\times(2+d)}$ that
\begin{align}\label{eq:euler_l1_inner_problem}
\min_{\vecc{Z}\in\IR^{N\times(2+d)}}\, \frac{1}{2\gamma}\norm{\vecc{Z}-\vecc{X}}{F}^2 + \iota_{\Lambda_1}(\vecc{Z}) + \iota_{\Lambda_2}(\vecc{Z}).
\end{align}
Notice, problem \eqref{eq:euler_l1_inner_problem} is exactly \eqref{eq:invariant_domain_limiter2} with parameter $\alpha = \gamma$. Therefore, it can be solved by applying DYS method. After obtaining $\prox_h^\gamma$, we utilize the DRS \eqref{eq:DR_algorithm} to compute a minimizer of \eqref{eq:opt_model_l1_Euler}.
To start the algorithm, we set $\vec{Y}^0 = [\vec{u}, \vec{v}_1, \cdots, \vec{v}_d, \vec{w}]$ and $\vec{X}^0 = \prox_h^\gamma(\vec{Y}^{0})$.
To terminate the iteration, we choose the stopping criterion $\norm{\vec{Y}^{k+1}-\vec{Y}^{k}}{2h} < \epsilon$ with a small positive tolerance $\epsilon$.

%%%%%%%%%%%%%%%%%%%%%%%%%%%%%%%%%%%%%%%%%%%%%%%%%%%%%%%%%%%%%%%%%%%%%%%%%%%%%%%%%%%%%%%%%%%%%%%%%%%%%%%%%%%%%%%%%%%%%%%
\section{Numerical experiments}\label{sec:numerical_experiments}
%%%%%%%%%%%%%%%%%%%%%%%%%%%%%%%%%%%%%%%%%%%%%%%%%%%%%%%%%%%%%%%%%%%%%%%%%%%%%%%%%%%%%%%%%%%%%%%%%%%%%%%%%%%%%%%%%%%%%%%
In this section, we show results of applying the optimization-based limiters to a series of representative tests. These include studies of using one-dimensional manufactured data and two-dimensional simulations of demanding gas dynamic benchmarks for testing robustness of positivity-preserving schemes.

For solving $\ell^2$-norm minimizations \eqref{eq:opt_model_l2_scalar} and  \eqref{eq:invariant_domain_limiter3} in all the numerical tests in this paper, we use DYS with a fixed step size $\gamma=\frac{1}{L}$, where the Lipschitz constant is $L=\frac{1}{\alpha}$, 

For the $\ell^1$ minimizations \eqref{eq:opt_model_l1_scalar} and \eqref{eq:opt_model_l1_Euler}, we use DRS. In each iteration of DRS, the proximal operator for summation of two indicator functions, such as \eqref{eq:euler_l1_inner_problem}, is numerically solved by DYS. The step size in the DYS is still $\frac{1}{L}$ where $L$ is the Lipschitz constant for the $\ell^2$ term. The parameters in DRS are taken as follows.
We use a relaxation parameter $\lambda=1$ and a step size $\gamma>0$ obtained from tuning on a few data sets. For instance, for using the optimization based $\ell^1$-norm limiter in each time step in a time dependent problem, we first run DRS with a different choices of $\gamma>0$ at a few time steps when the limiter is invoked, then choose the best value of $\gamma$ in the sense of the fastest convergence of DRS for all future time steps. 

We remark that optimal parameters of $\lambda$ and $\gamma$ in DRS for solving  \eqref{eq:opt_model_l2_scalar} are given in \cite{liu2024simple}, which was derived from convergence rate analysis. Unfortunately, neither the analysis nor the same parameter formulae can be extended to the invariant domain \eqref{invariant-domain}.

\subsection{Synthetic tests in one dimension}\label{sec:numercal_experiment:tests_1D}
In this part, our primary objective is to examine the optimization methods through simple tests, including linear advection governed traveling waves and manufactured perturbation of Lax shock tube data for Euler equations. We generate out-of-bound data by the baseline DG scheme without any limiters, then apply the limiters to data, without integrating optimization limiters into baseline PDE simulators, so that we only study the performance of optimizers here. 
\par
In all one-dimensional tests, we chose $\varepsilon = 10^{-13}$ for numerical admissible set $G^\varepsilon$. In addition, the convergence tolerance for both DYS and DRS methods is set to $\epsilon = 10^{-13}$.

\paragraph{\bf Example~4.1 (Traveling triangle and square waves)}
The linear advection equation $\partial_t u + \partial_x u = 0$ enjoys the maximum principle: if initial data is in $[m,M]$, then so is the solution belongs at any later time $t > 0$.
The computational domain  is $\Omega=[0,3]$. We initialize right-moving triangle and square waves as follows and evolve to time $T = 1$.  
\begin{align}
u^0 = \begin{cases}
4x, & \text{if}~x \in (0.25, 0.5], \\
-4x + 4, & \text{if}~x \in (0.5, 0.75], \\
2, & \text{if}~x \in (1.25, 1.75], \\
1, & \text{otherwise}.
\end{cases}
\end{align}
Thus, the exact solution $u(x,t)\in[1,2]$ for any $x \in \Omega$ and $t \in [0,T]$.
To generate out-of-bound data, we solve the equation on a uniform mesh consisting of $300$ cells. We employ a fourth-order Runge-Kutta (RK) DG method without any limiter. The time step size is set to $0.001$, resulting in $1000$ out-of-bound data sets for postprocessing.
% four-stage fourth-order RK + $\IP^3$ Lagrange basis + $4$-point Gauss-Lobatto quadrature.
\par
We apply DYS method to solve the $\ell^2$ model \eqref{eq:opt_model_l2_scalar} and apply direct method ($\mathtt{ClipAndAssuredSum}$) \eqref{eq:ClipAndAssuredSum} and DRS method to solve the $\ell^1$ model \eqref{eq:opt_model_l1_scalar}, respectively. After postprocessing, all cell averages are enforced within the bounds, see Figure~\ref{fig:linear_advection_1D_1}. 
Notice that $\ell^1$ minimizers are not necessarily unique, and it has been proven that $\ell^2$ minimizer is one of $\ell^1$ minimizers \cite{bradley2019communication}. In general, the $\ell^1$ minimizers found by either $\mathtt{ClipAndAssuredSum}$ or different splitting methods may be different from one another, and also different from $\ell^2$ minimizer.
Nonetheless, for this particular test, up to $\epsilon$, there is  no difference in the three results in Figure~\ref{fig:linear_advection_1D_1}.
\par
In terms of performance of these optimization solvers, for the $\ell^2$ model, DYS method converges within $60$ iterations on all data sets at different time steps. In contrast, the $\ell^1$ model requires more DRS iterations, i.e., converges within $200$ iterations on all data sets at different time steps. 
 Hereinafter, we use the {\it number of projection to admissible set} to denote the total count of computations of the proximal operator for indicator function $\iota_{\Lambda_2}(\vec{x})$.
The total number of projections to the admissible set $[1,2]$ is significantly higher for the $\ell^1$ model because of its inner DYS iterations, see Figure~\ref{fig:linear_advection_1D_2}.
We use step size $\gamma = 10^{-10}$ in DRS for the $\ell^1$ model. 
When the iteration sequences are sufficiently close to the fixed point, both the DYS and DRS methods exhibit asymptotic linear convergence, as illustrated in Figure~\ref{fig:linear_advection_1D_2}.
\begin{figure}[ht!]
\centering
\begin{tabularx}{0.95\linewidth}{@{}c@{~~}c@{}}
\includegraphics[width=0.475\textwidth]{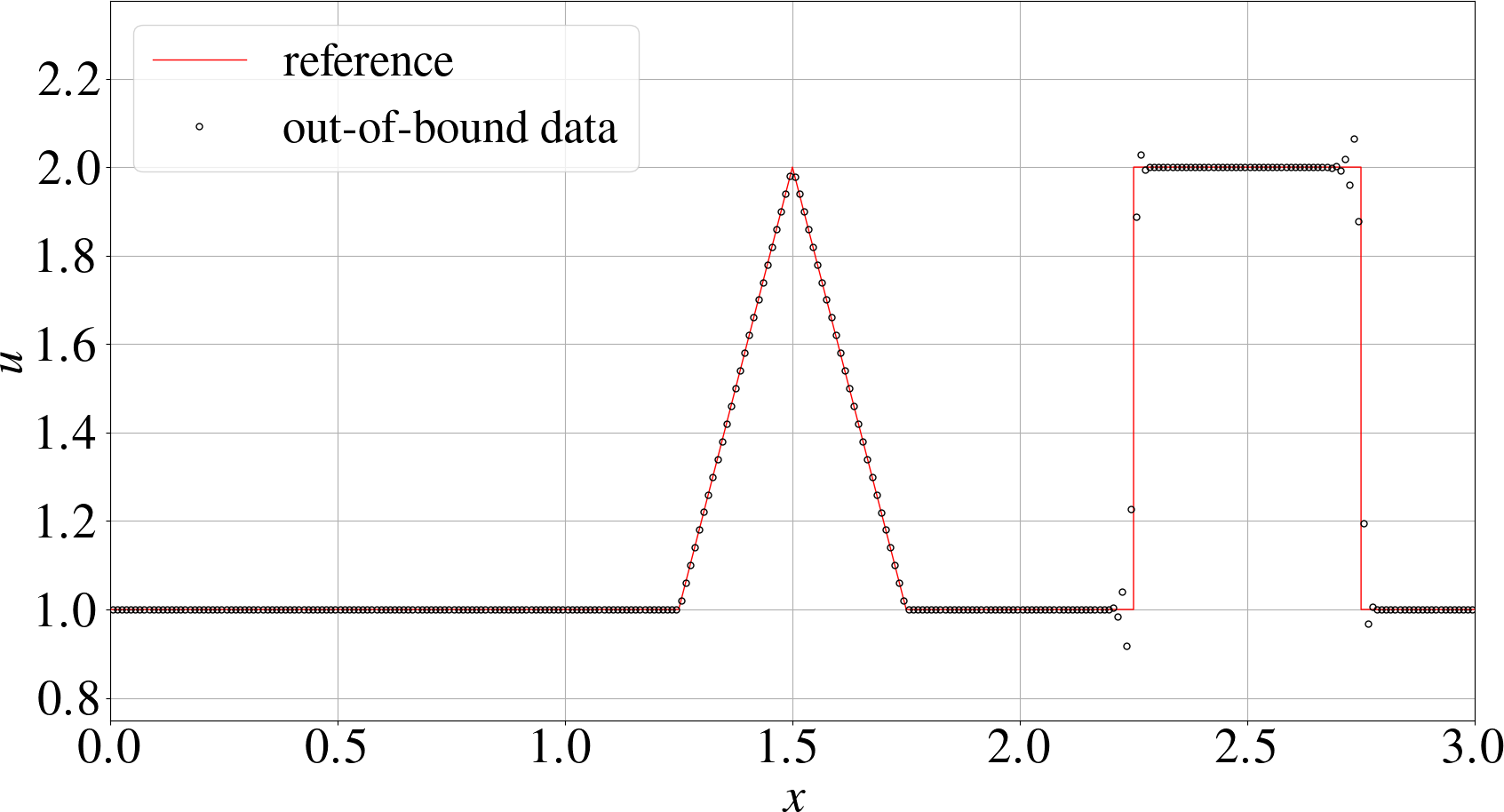} &
\includegraphics[width=0.475\textwidth]{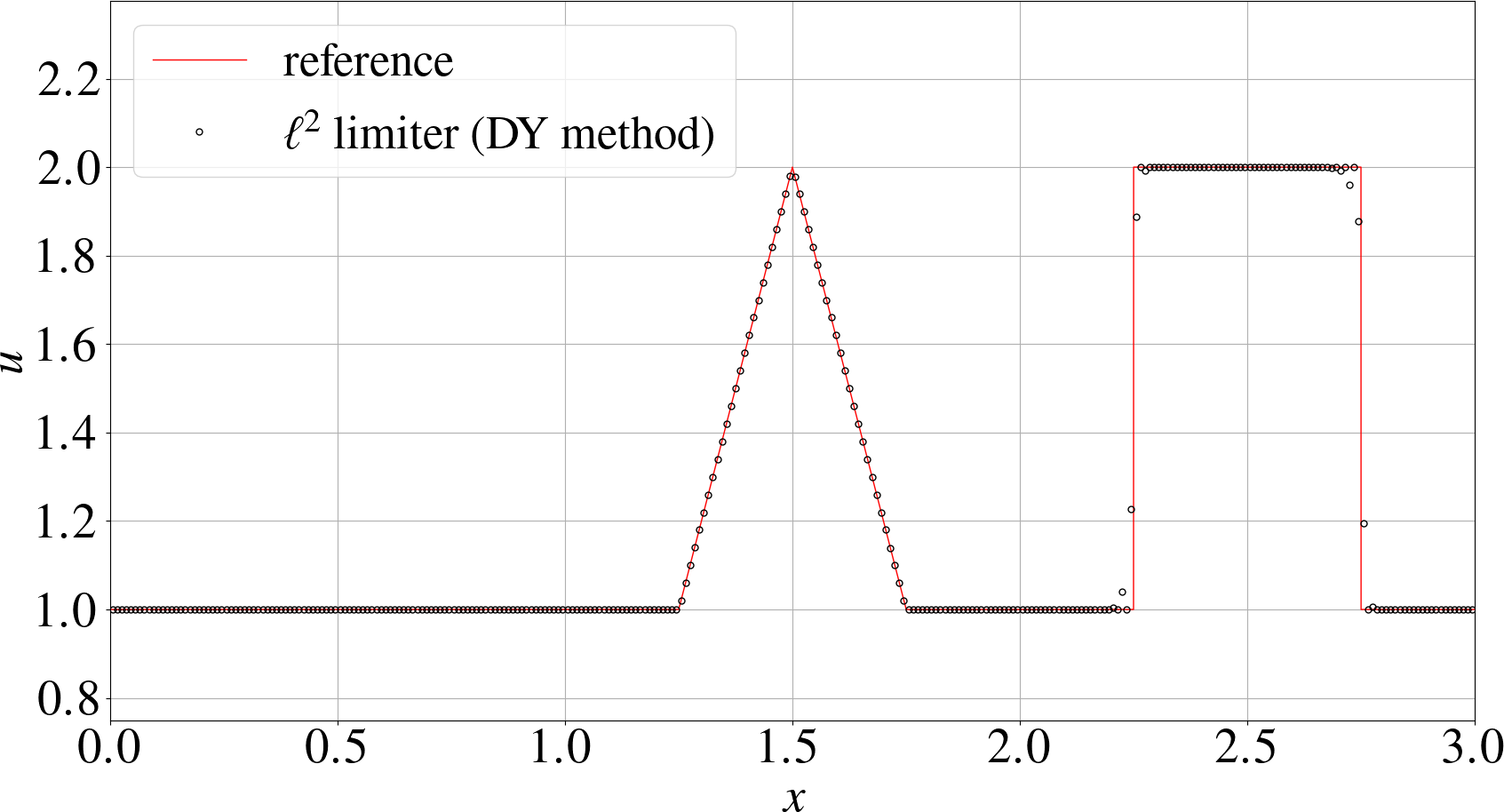} \\
\includegraphics[width=0.475\textwidth]{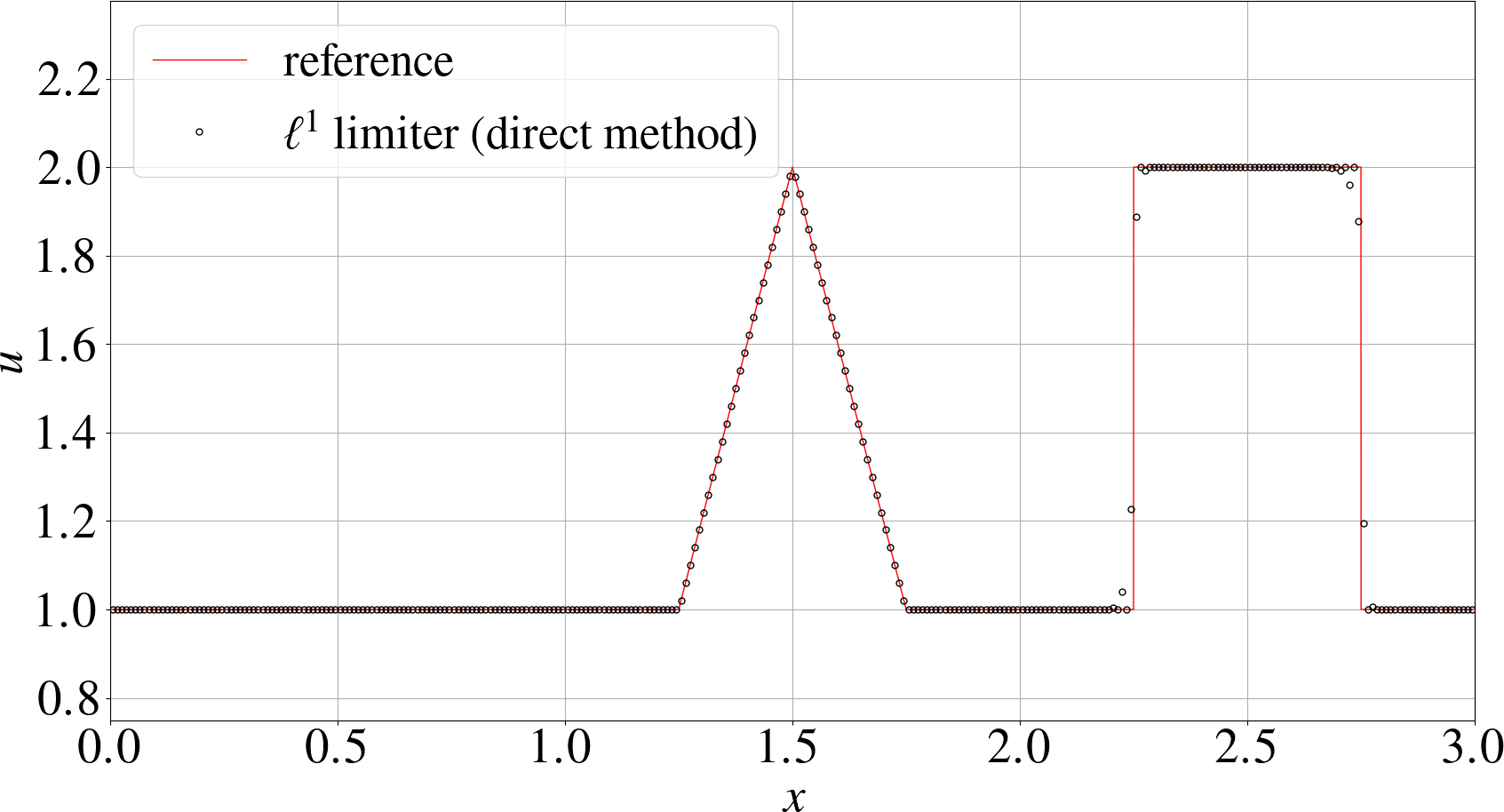} &
\includegraphics[width=0.475\textwidth]{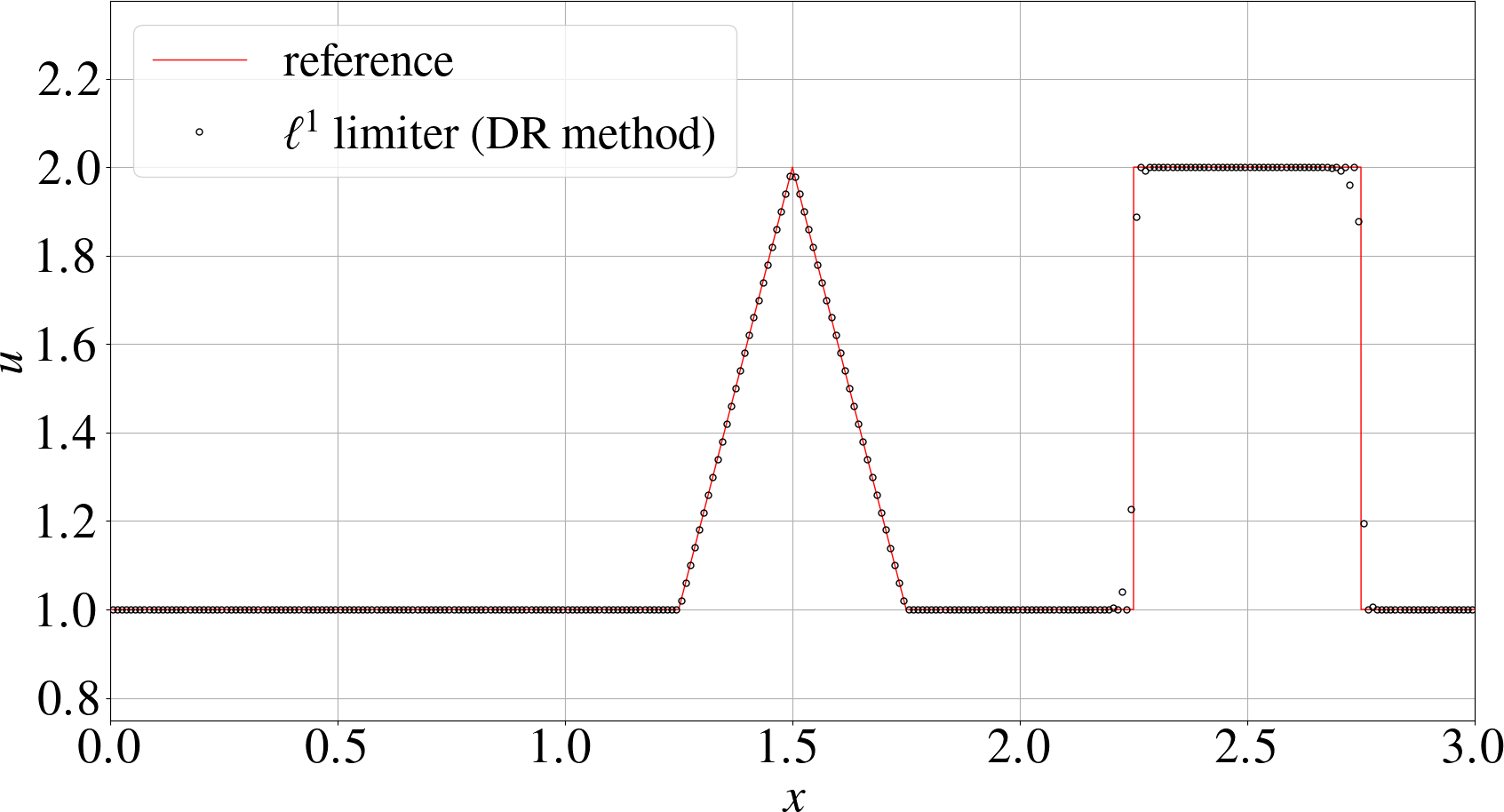} \\
\end{tabularx}
\caption{Top left: out-of-bound cell averages of data set $1000$ (generated in time step $1000$). Top right: the postprocessed result from the $\ell^2$ limiter with DYS method. Bottom: the postprocessed results from the $\ell^1$ limiter with direct method ($\mathtt{ClipAndAssuredSum}$) and DRS method. There is no   difference between $\ell^2$ limited result and two $\ell^1$ limited results up to $\epsilon$.}
\label{fig:linear_advection_1D_1}
\end{figure}
\begin{figure}[ht!]
\centering
\begin{tabularx}{\linewidth}{@{}c@{~}c@{~}c@{~}c@{}}
\begin{sideways}{$\hspace{0.95cm} \text{\small $\ell^2$ limiter} \quad$}\end{sideways} &
\includegraphics[width=0.32\textwidth]{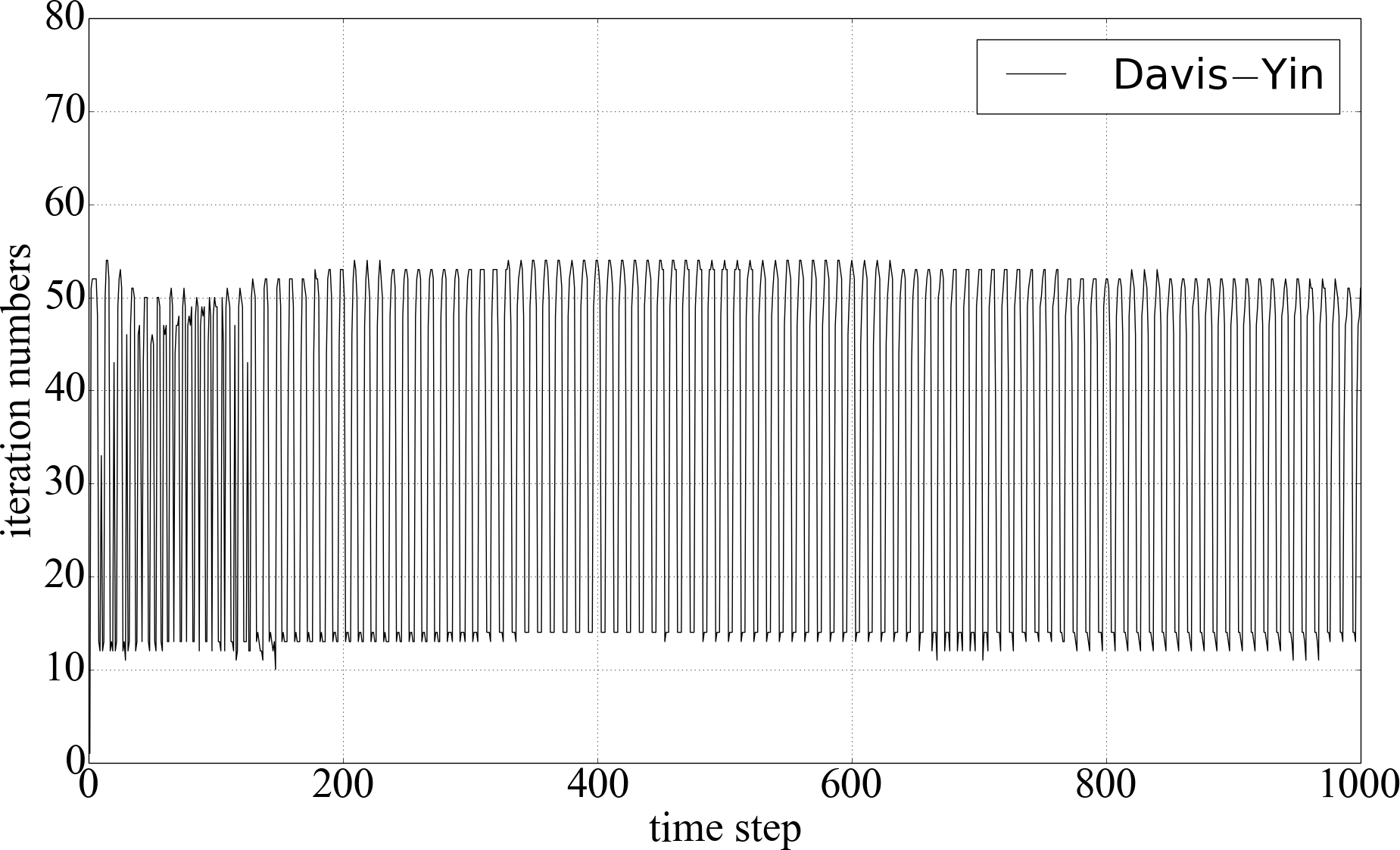} &
\includegraphics[width=0.32\textwidth]{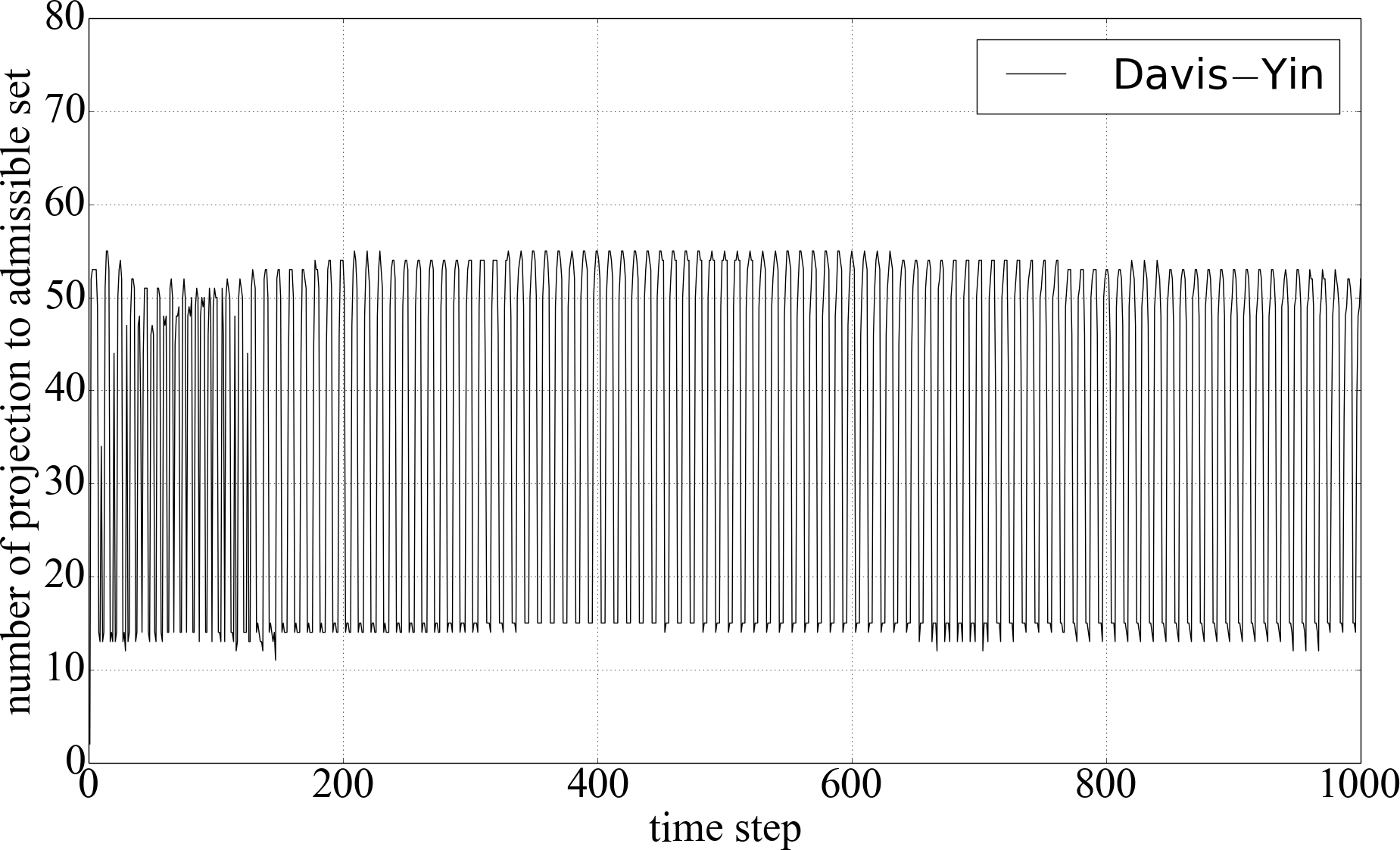} &
\includegraphics[width=0.32\textwidth]{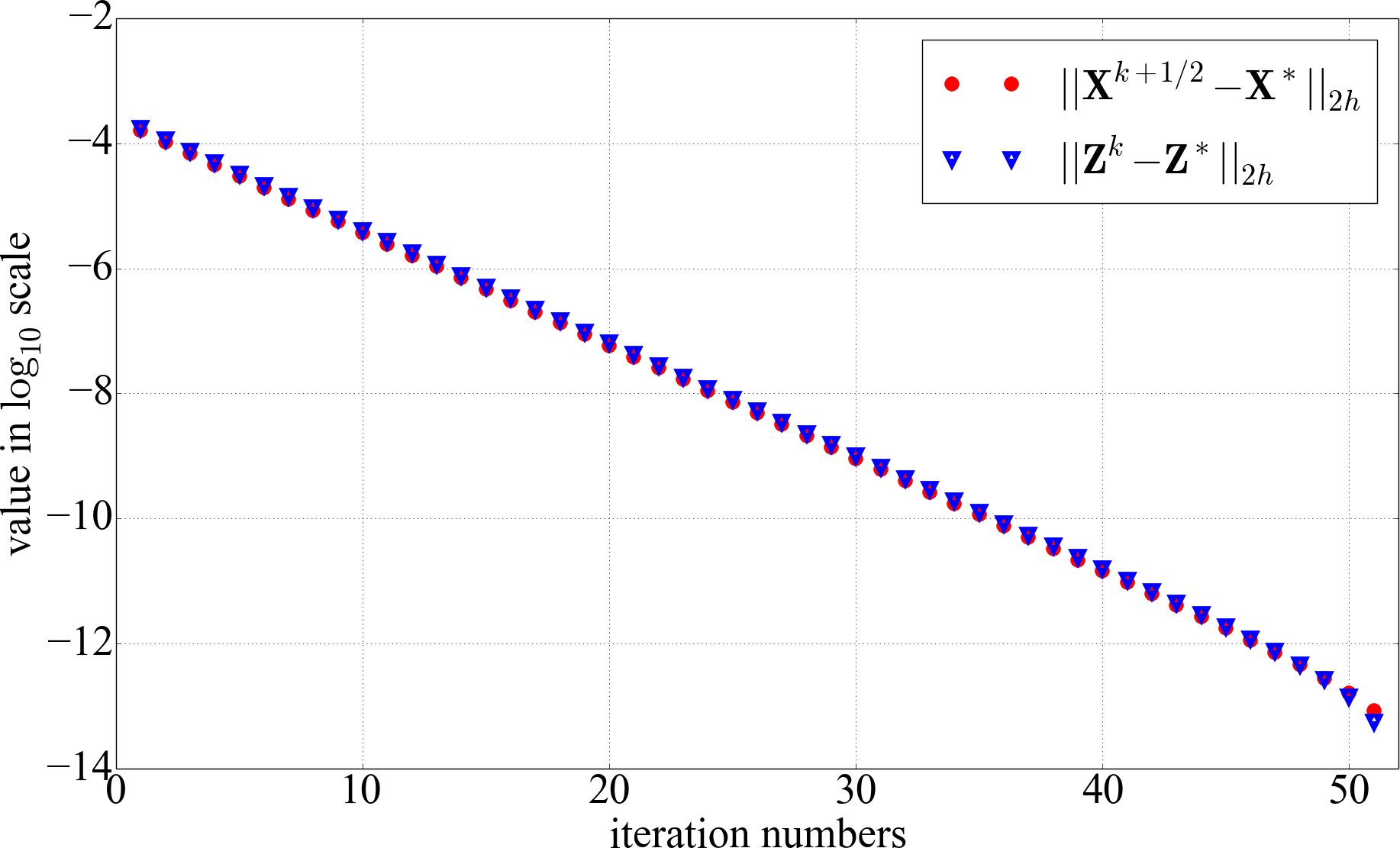} \\
\begin{sideways}{$\hspace{0.95cm} \text{\small $\ell^1$ limiter} \quad$}\end{sideways} &
\includegraphics[width=0.32\textwidth]{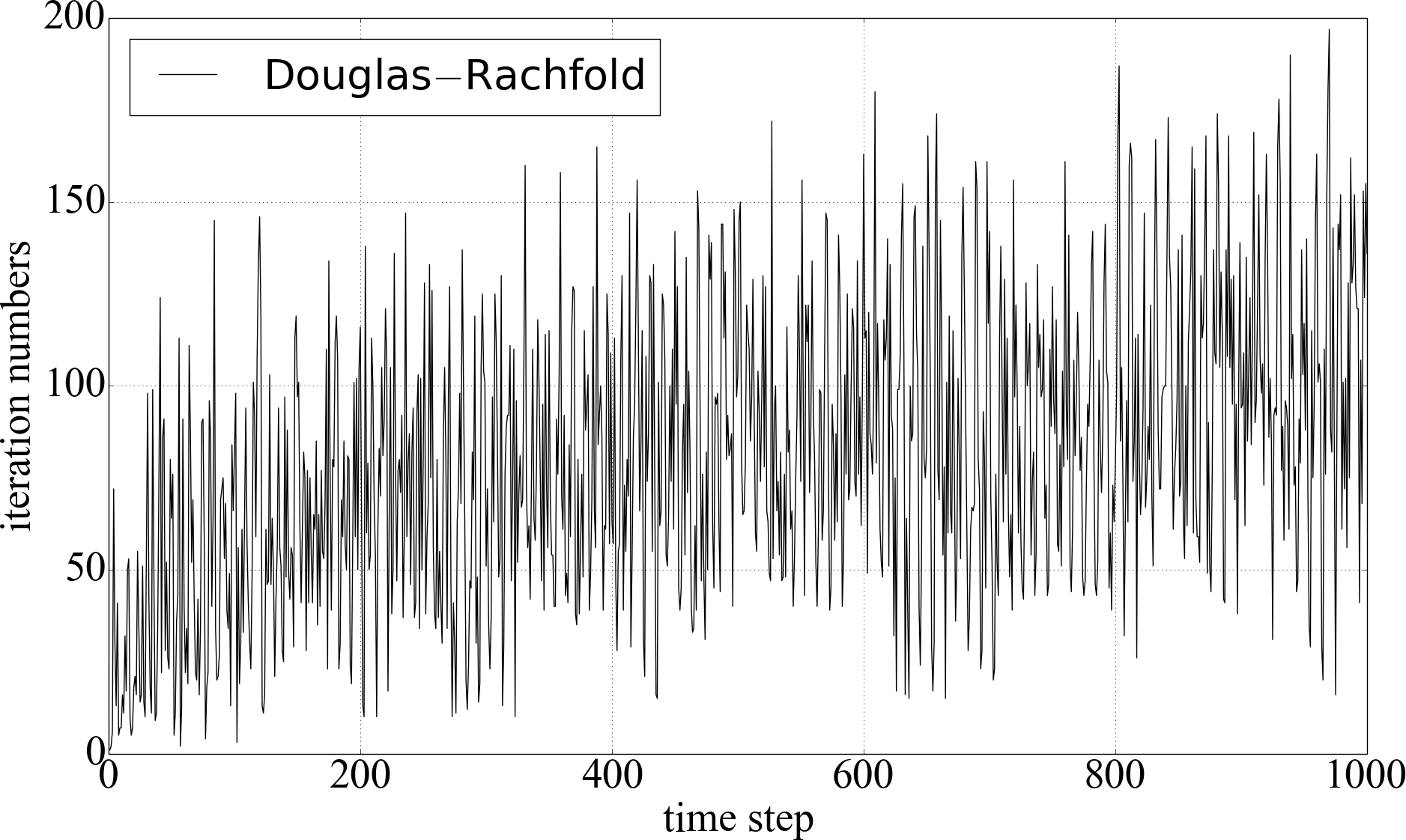} &
\includegraphics[width=0.32\textwidth]{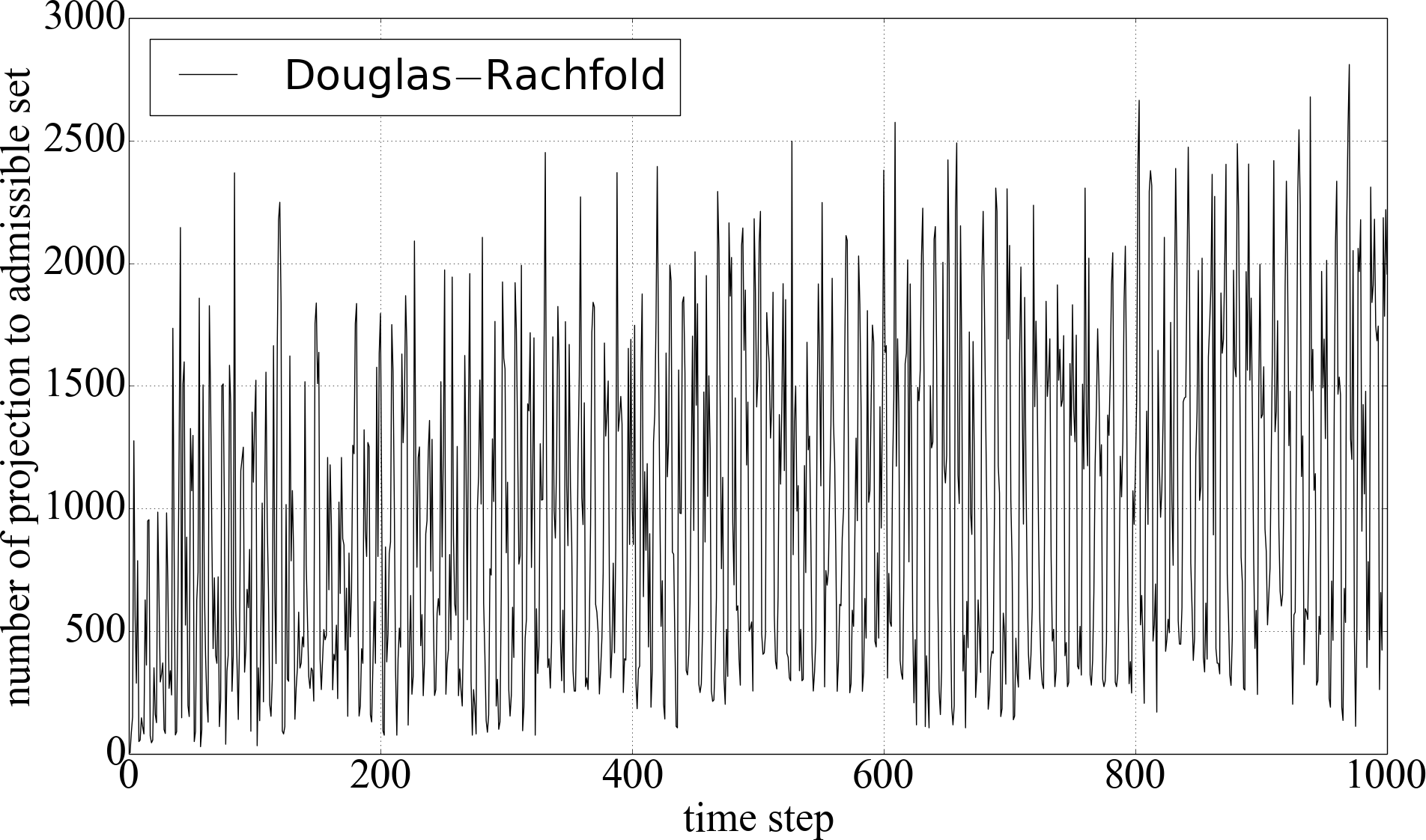} &
\includegraphics[width=0.32\textwidth]{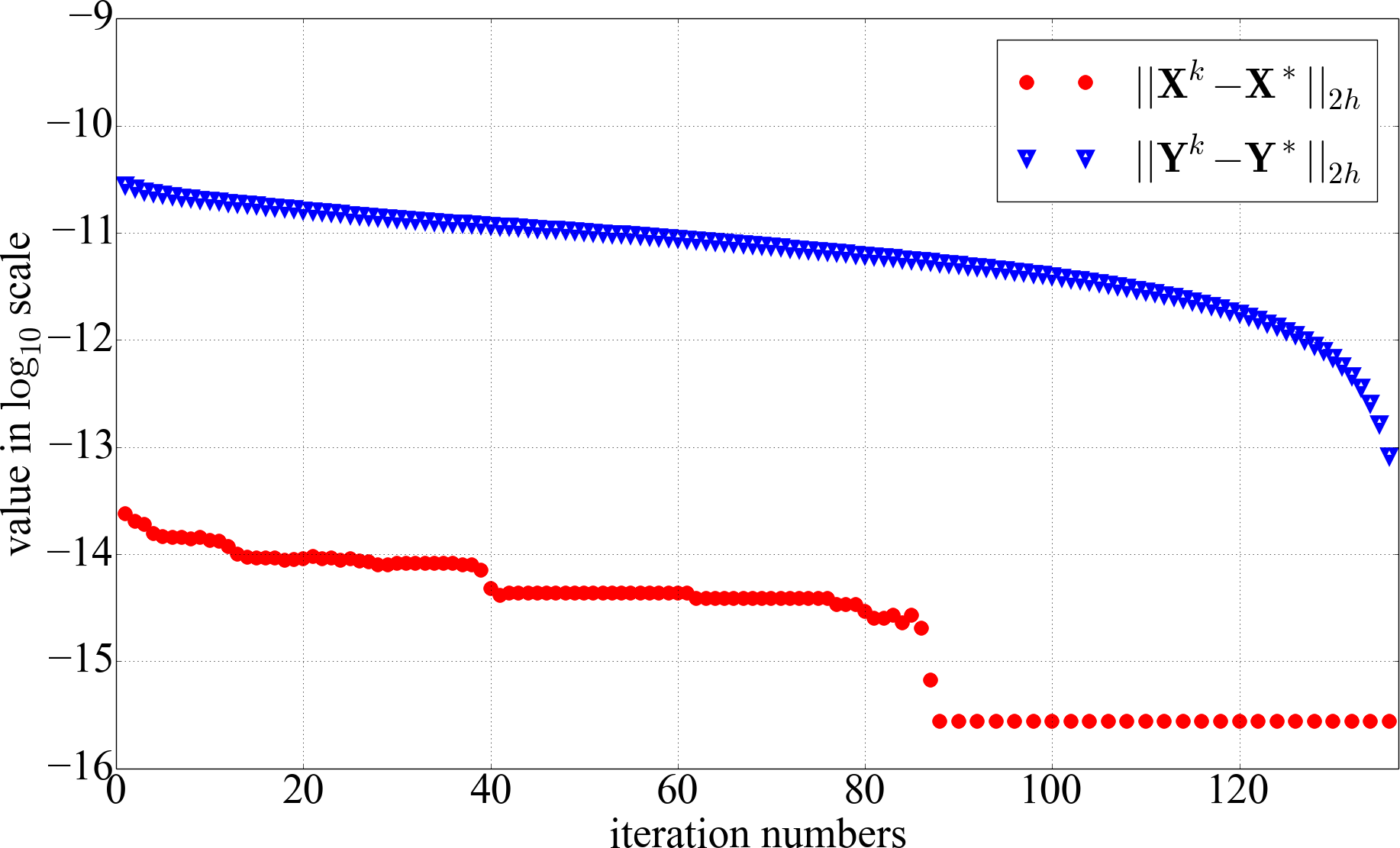} \\
\end{tabularx}
\caption{Top row: DYS method solving $\ell^2$ model. Bottom row: DRS method nested with DYS of solving $\ell^1$ model.  
From left to right: number of iterations for data at different time steps, number of projections, and the convergence error in processing one particular data set at time step $1000$. Minimizers and fixed points, denoted by superscript star, are computed numerically by using sufficiently many iterations of splitting methods. }
\label{fig:linear_advection_1D_2}
\end{figure}

\paragraph{\bf Example~4.2 (Perturbation of Lax shock tube data)}
We uniformly partition the computational domain $\Omega = [-5,5]$ into $400$ cells. To generate data violating the invariant domain, we compute cell averages of the exact solution at time $t = 1.3$ to the Lax shock tube problem for compressible Euler equations \cite[Example 2]{zhang2017positivity} and apply perturbations near the shock location.
Specifically, we perturb $5\%$ of the cells, i.e., $20$ cells in total. For the $10$ cells in pre-shock region, we subtract the following random values sampled from a uniform distribution: on each cell $K$, we modify the cell average by
\begin{align}
\on{\overline{\vec{U}}}{K} -
\begin{bmatrix}
0.1\max\abs{\rho}\,\xi_\rho\\
0.01\max\abs{m}\,\xi_m\\
0.1\max\abs{E}\,\xi_E\\
\end{bmatrix},
\quad\text{where}\quad
\xi_\rho, \xi_m, \xi_E \overset{\mathrm{iid}}{\sim} \mathrm{Uniform}[1,2].
\end{align}
To ensure conservation, we add the same values back to the $10$ cells in the post-shock region. In this way, we obtain one out-of-bound data set. To test our optimization solvers, we repeat and create $1000$ manufactured data set.
\par
We apply DYS method to solve the $\ell^2$ model \eqref{eq:invariant_domain_limiter3} and DRS method with $\gamma = 10^{-4}$  to solve the $\ell^1$ model \eqref{eq:opt_model_l1_Euler}, respectively. After postprocessing, all cell averages are enforced within the bounds. For the $\ell^2$ model, DYS method converges within $20$ iterations on all data sets. Compared to the number of iterations in solving the $\ell^2$ model, the $\ell^1$ model takes more DRS iterations and converges within $200$ iterations on all data sets. But, when compared with the total number of projections to the admissible set $G^\varepsilon$, the required projection numbers are significantly higher for the $\ell^1$ model due to its inner DYS iterations, see Figure~\ref{fig:lax_shock_tube_1D}.
We observe asymptotic linear convergence for both DYS and DRS methods, when iteration sequences are sufficiently close to the fixed point. See Figure~\ref{fig:lax_shock_tube_1D} for an illustration.
\begin{figure}[ht!]
\centering
\begin{tabularx}{\linewidth}{@{}c@{~}c@{~}c@{~}c@{}}
\begin{sideways}{$\hspace{0.95cm} \text{\small $\ell^2$ limiter} \quad$}\end{sideways} &
\includegraphics[width=0.32\textwidth]{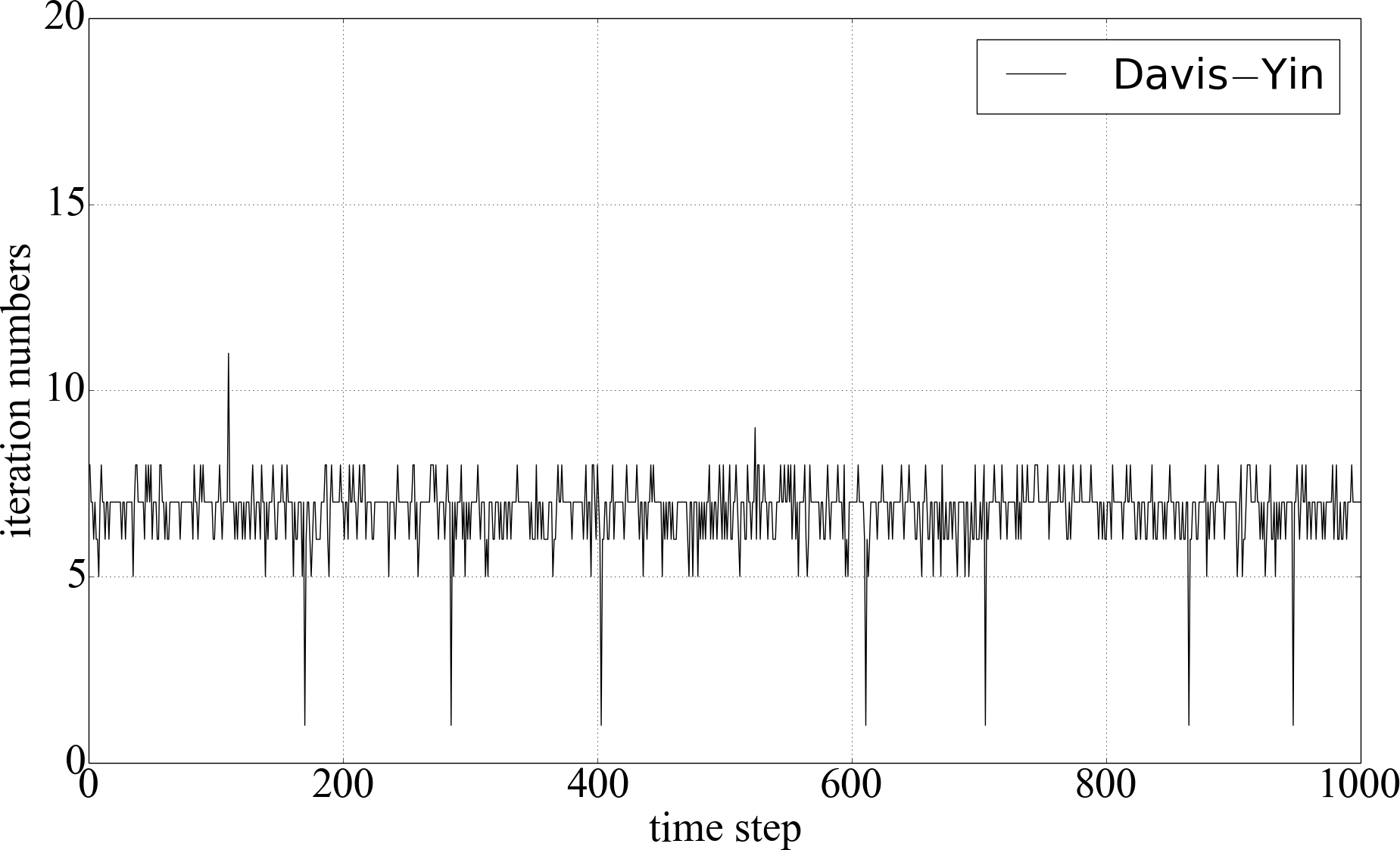} &
\includegraphics[width=0.32\textwidth]{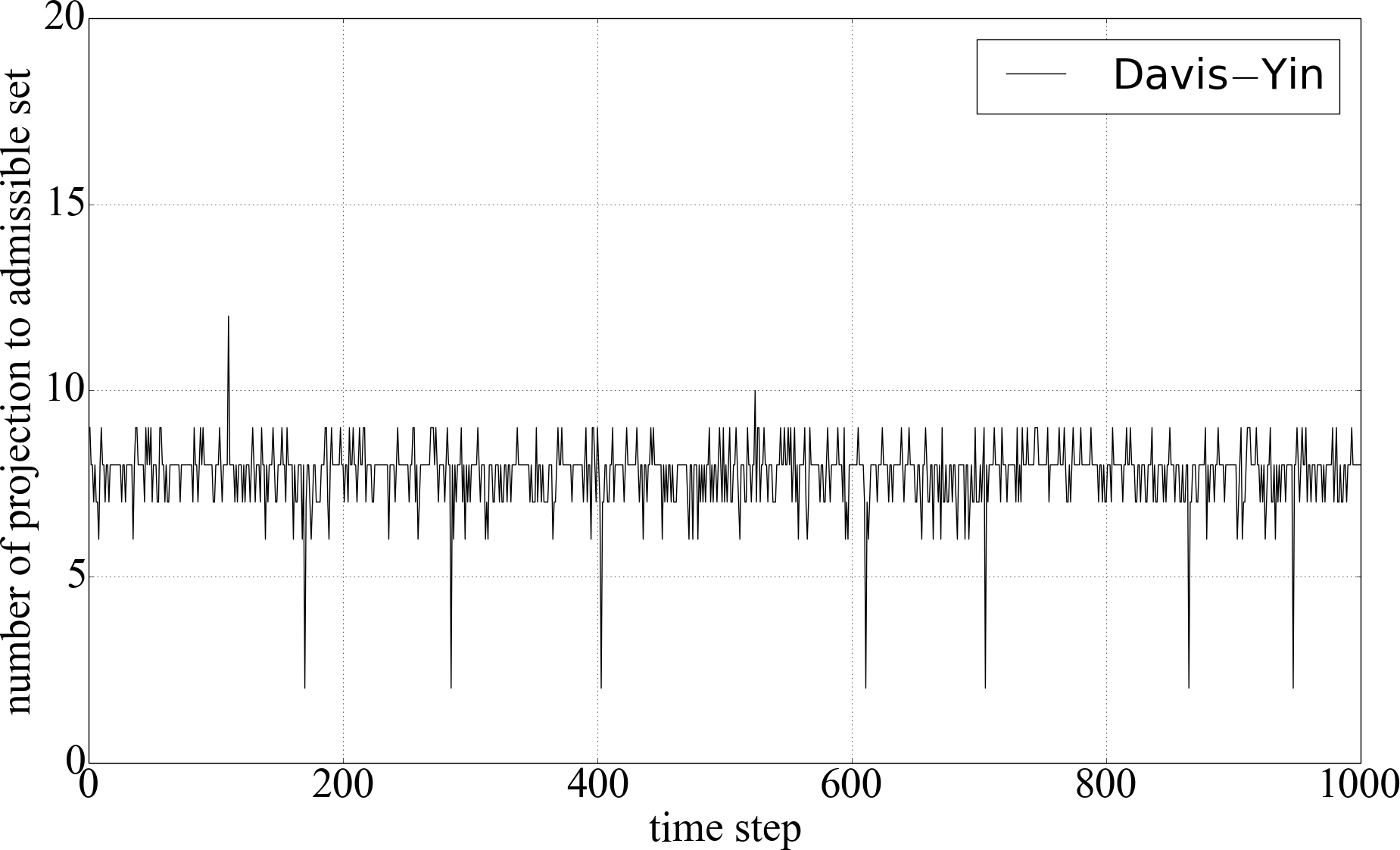} &
\includegraphics[width=0.32\textwidth]{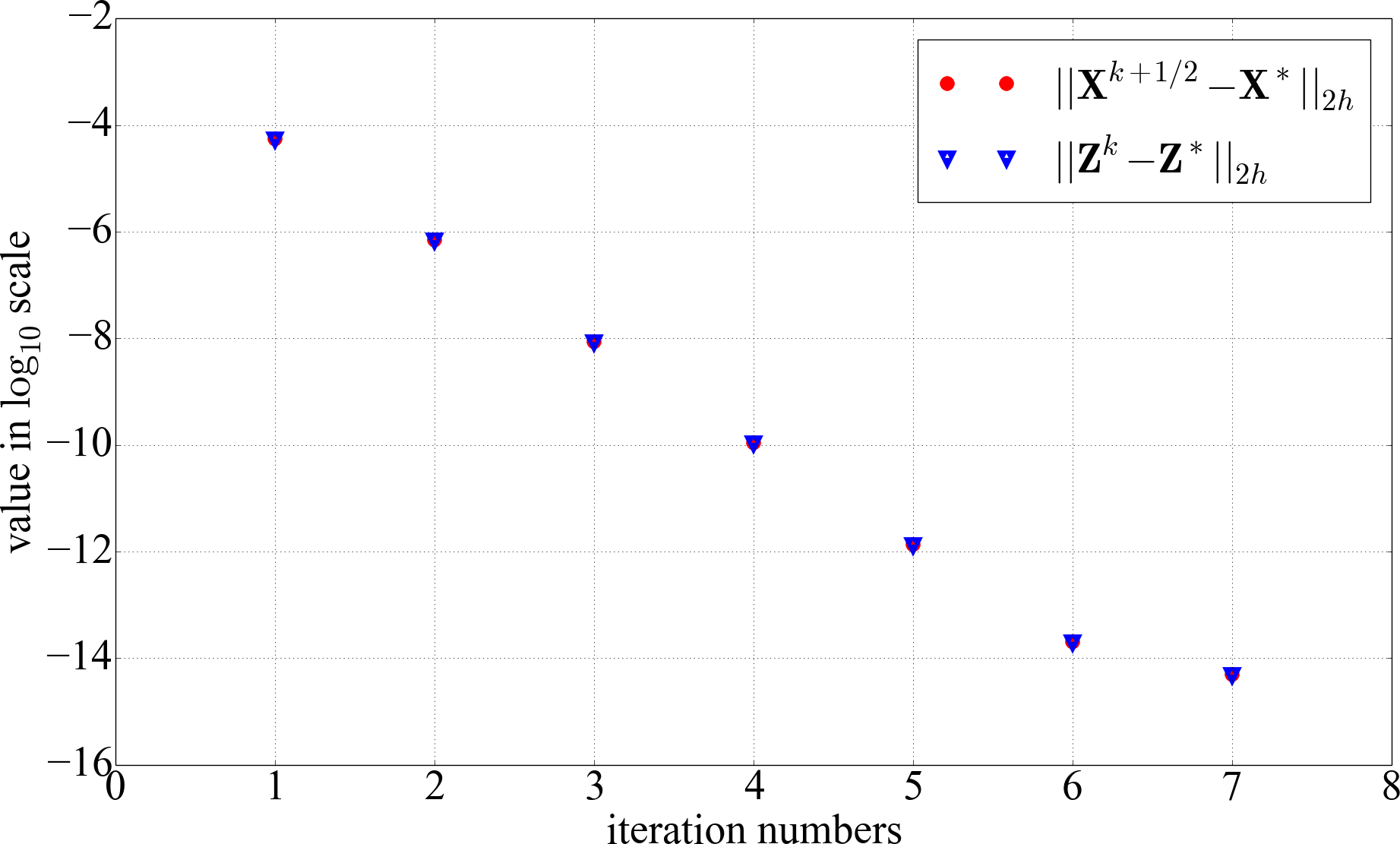} \\
\begin{sideways}{$\hspace{0.95cm} \text{\small $\ell^1$ limiter} \quad$}\end{sideways} &
\includegraphics[width=0.32\textwidth]{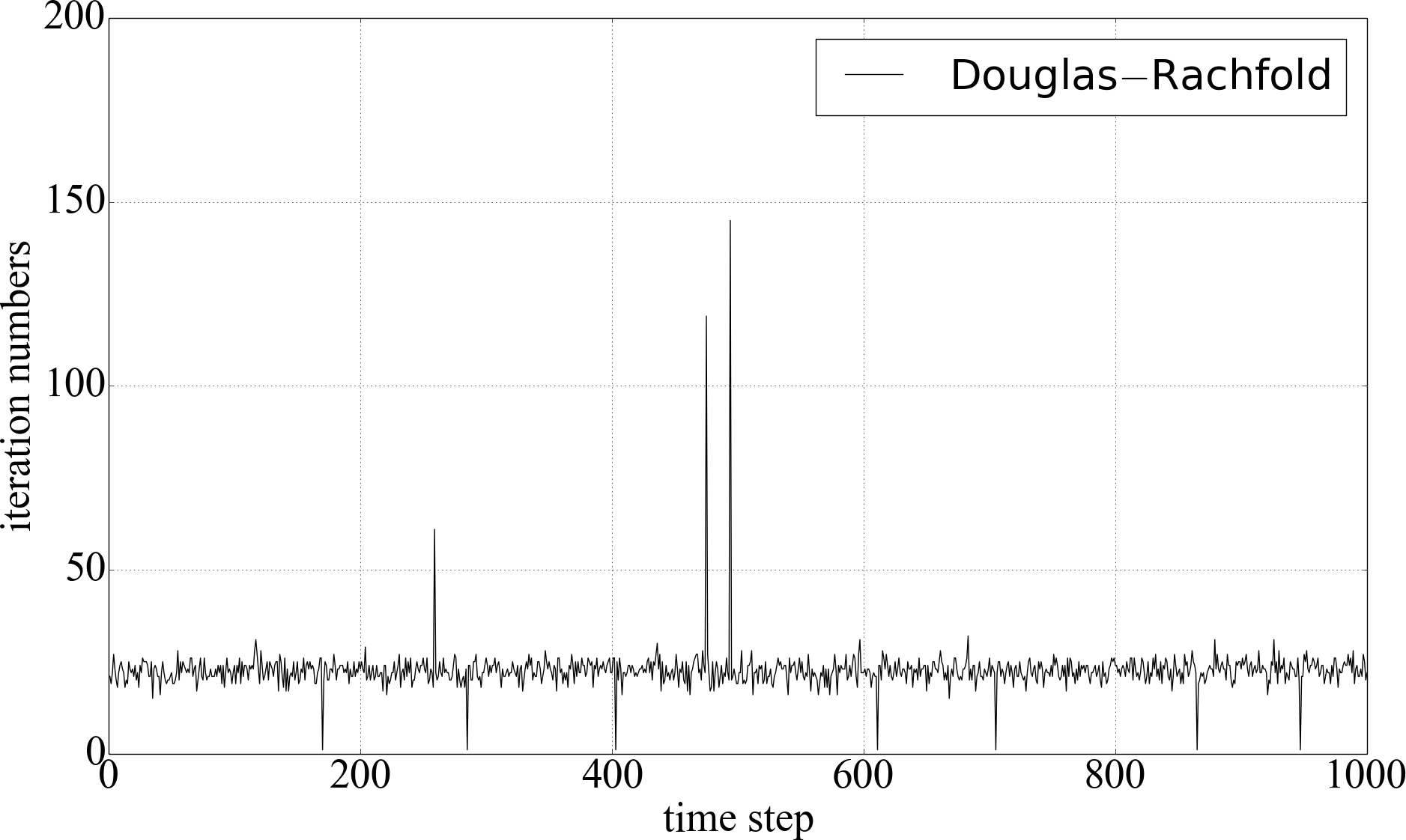} &
\includegraphics[width=0.32\textwidth]{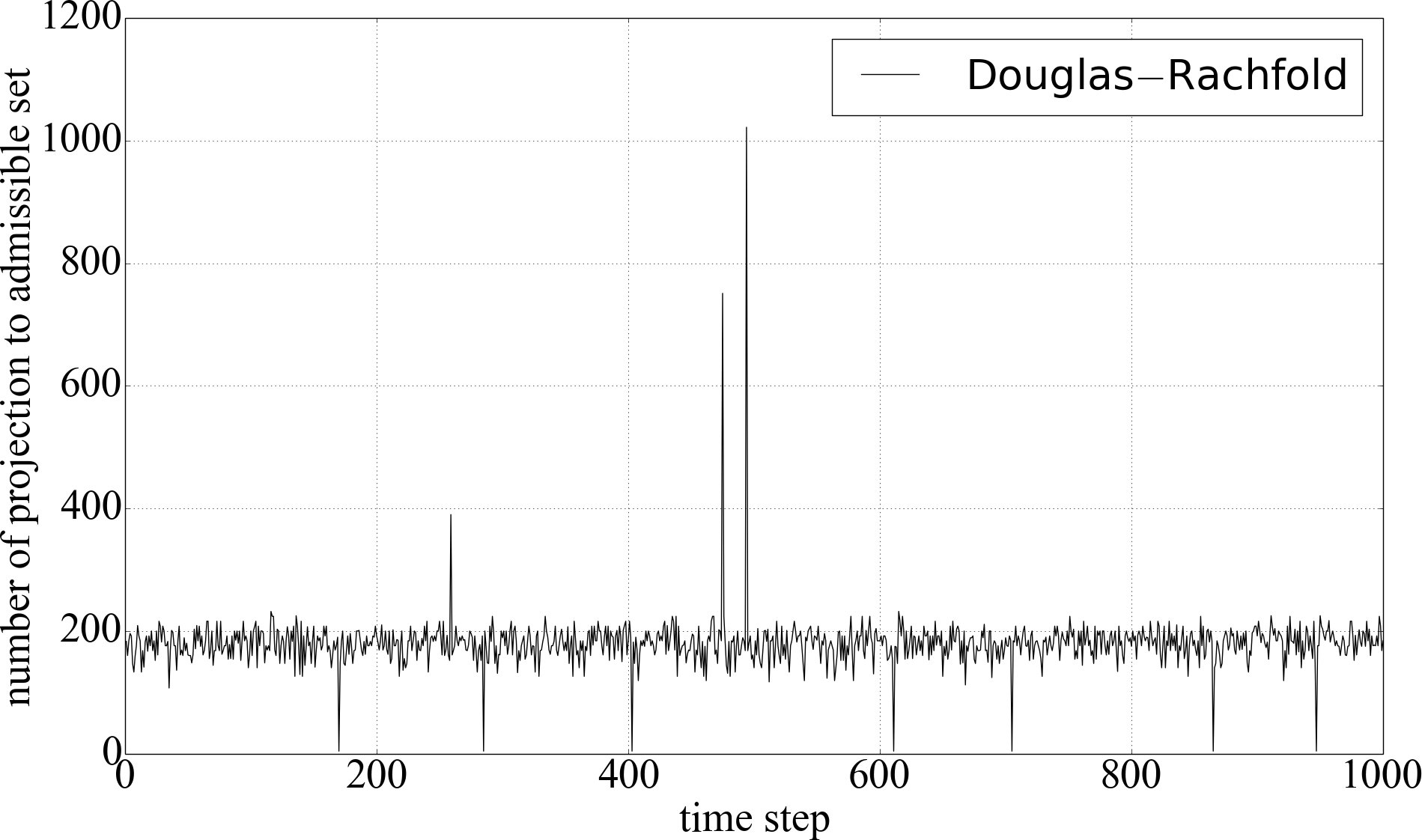} &
\includegraphics[width=0.32\textwidth]{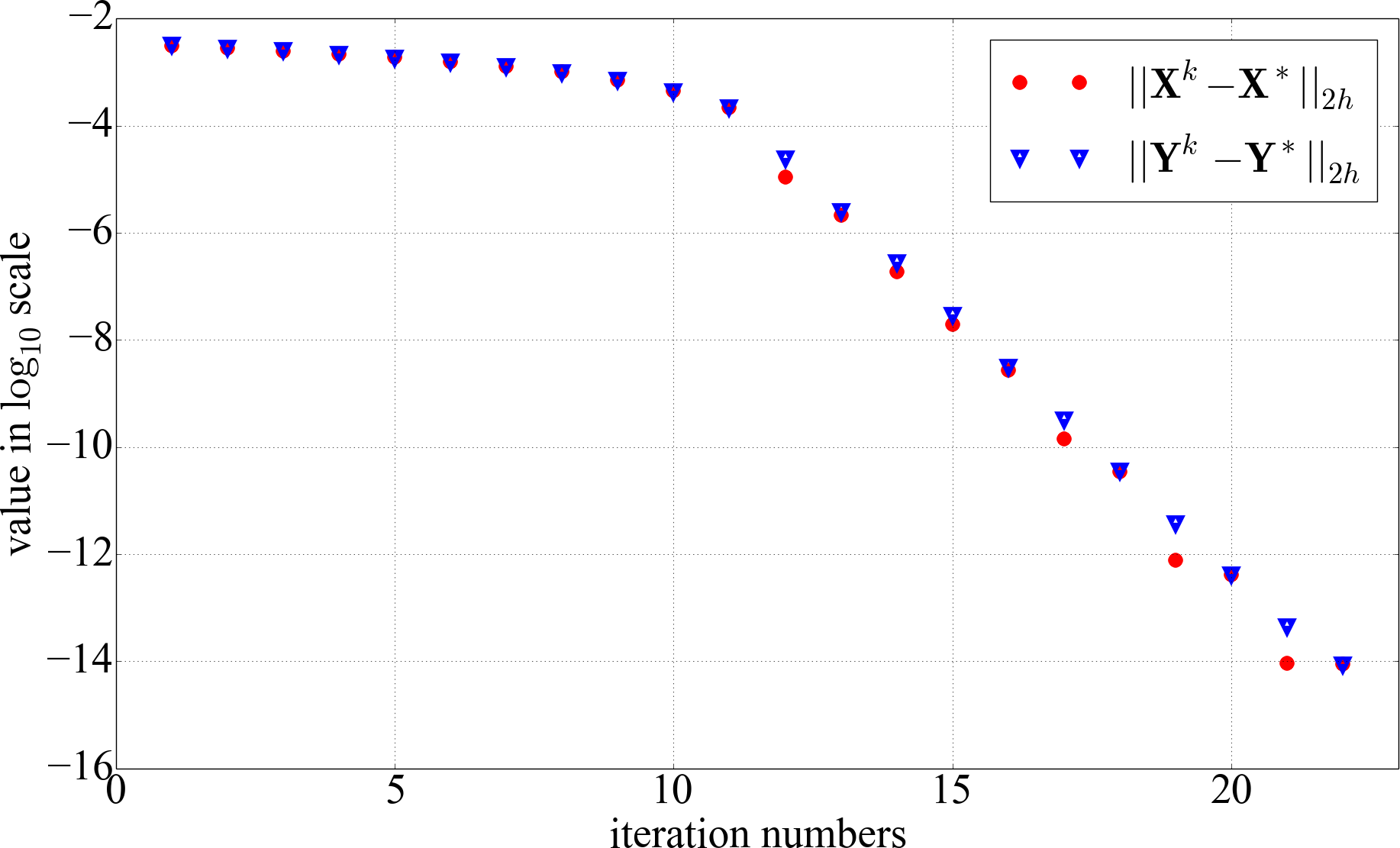} \\
\end{tabularx}
\caption{Top row: DYS method solving $\ell^2$ model. Bottom row: DRS method nested with DYS of solving $\ell^1$ model. From left to right: number of iterations for data at different time steps, total number of projections to admissible set, and the convergence error in processing the data set at time step $1000$. Minimizers and fixed points, denoted by superscript star, are computed numerically by taking sufficiently many iterations.  }
\label{fig:lax_shock_tube_1D}
\end{figure}

\subsection{Benchmark tests in two dimension}\label{sec:numercal_experiment:tests_2D}
In this part, we apply limiters to DG schemes solving demanding test cases, including the Sedov blast wave and a high speed astrophysical jet. Our optimization-based postprocessing procedure preserves invariant domains effectively while producing high-quality solutions for challenging practical applications.
\par
In all two-dimensional simulations, we apply the optimization-based cell average limiter after each stage of RK method, whenever there is a cell average out of the admissible set. 
Only when the cell averages are in the invariant domain,    the Zhang-Shu limiter \cite{zhang2010positivity} can be further used to ensure quadrature points are also in $G$.
For all test cases except the astrophysical jet, we choose $\varepsilon = 10^{-13}$ in $G^\varepsilon$ and the convergence tolerance $\epsilon = 10^{-13}$ for DYS and DRS methods. For the astrophysical jet, we use $\varepsilon = 10^{-8}$ and $\epsilon = 10^{-8}$ due to its extremely high speed.
For solving the $\ell^1$ limiter, DRS method with a step size $\gamma = 10^{-7}$ is used in all two dimensional examples.

\paragraph{\bf Example~4.3 (Convergence study)}
We utilize the manufactured solution method on domain $\Omega = [0,1]^2$ with end time $T = 0.1$ to study the spatial convergence rate of solving Euler equations. The prescribed non-polynomial solutions are: 
\begin{subequations}
\begin{align}
\rho &= \sin^{16}{\pi(x + y - 2t)} + 10^{-13},\\
\vec{u} &= \begin{bmatrix}
1 \\ 1
\end{bmatrix},\\
e &= \displaystyle \frac{10^{-13}}{(\gamma - 1)}\cdot\frac{1}{\sin^{16}{\pi(x + y - 2t)} + 10^{-13}}.
\end{align}
\end{subequations}
The right-hand side functions and the initial and boundary conditions are defined by the manufactured solutions. 
The domain $\Omega$ is uniformly partitioned into square cells. For spatial discretization, we employ $\IP^k$ ($k=2$ and $3$) DG methods, where the bases are constructed by orthonormal Legendre polynomials. The numerical integrations are evaluated by $(k+1)^2$-point Gauss quadrature. 
\par
Let $\mathtt{err}_{\Delta x}$ denote the error on a grid associated with mesh resolution $\Delta x$. The discrete $L^2_h$ and $L^1_h$ errors of density are computed by
\begin{subequations}
\begin{align}
L^2_h~\text{error:}\qquad
\|\rho_h^{n} - \rho(t^n)\|_{L^2_h}^2 &= 
{\Delta x}^2 \sum_{i} \sum_{\nu}\omega_\nu \Big|\sum_{j} \rho_{ij}^n\,\varphi_{ij}(\vec{q}_\nu) - \rho(t^n\!, \vec{q}_\nu)\Big|^2, \\
L^1_h~\text{error:}\qquad
\|\rho_h^{n}-\rho(t^n)\|_{L^1_h} &= 
{\Delta x}^2 \sum_{i} \sum_{\nu}\omega_\nu \Big|\sum_{j} \rho_{ij}^n\,\varphi_{ij}(\vec{q}_\nu) - \rho(t^n\!, \vec{q}_\nu)\Big|,
\end{align}
\end{subequations}
where $\varphi_{ij}$ denotes the $j$-th basis on cell $i$, and $\omega_\nu$ and $\vec{q}_\nu$ are quadrature weights and points. The errors for momentum and total energy are measured similarly.
Then, the convergence rate is defined by $\ln(\mathtt{err}_{\Delta x}/\mathtt{err}_{\Delta x/2})/\ln(2)$.
\par
To evaluate convergence rates, we successively refine mesh of $\Delta{x} = 1/25, 1/50, 1/100$ and employ a fourth-order RK method with fixed time step size $\Delta{t} = 5\times 10^{-4}$. The time step size is small enough, which guarantees that the spatial error is dominated.
We apply $\ell^2$ limiter \eqref{eq:invariant_domain_limiter3} when measuring the convergence rates with respect to $L_h^2$ norm; and apply $\ell^1$ limiter \eqref{eq:opt_model_l1_Euler} when measuring the convergence rates with respect to $L_h^1$ norm, see Table~\ref{tab:convergence_Euler_test_l2} and Table~\ref{tab:convergence_Euler_test_l1}. 
The optimization-based cell average limiters are triggered and optimal convergence rates are obtained.
\begin{table}[ht!]
\centering
\begin{tabularx}{0.9\linewidth}{@{~~}c@{~~}|C@{~}C@{~}|C@{~}C@{~}}
\toprule
\multicolumn{1}{c|}{~} & \multicolumn{2}{c|}{$\IP^2$ basis} & \multicolumn{2}{c}{$\IP^3$ basis}  \\
\toprule
$\Delta x$ & $\|\vec{U}_h^{N_T} - \vec{U}(T)\|_{L_h^2}$ & rate & $\|\vec{U}_h^{N_T} - \vec{U}(T)\|_{L_h^2}$ & rate \\
\midrule
$1/25$  & $3.116\times 10^{-3}$ & ---   & $6.514\times 10^{-3}$ & ---   \\
$1/50$  & $3.534\times 10^{-4}$ & 3.141 & $4.846\times 10^{-5}$ & 7.071 \\
$1/100$ & $4.400\times 10^{-5}$ & 3.006 & $1.302\times 10^{-6}$ & 5.218 \\
\bottomrule
\end{tabularx}
\caption{The discrete $L_h^2$ errors and convergence rates for $\IP^2$ and $\IP^3$ spaces. The $\ell^2$ limiter is solved by DYS method.}
\label{tab:convergence_Euler_test_l2}
\end{table}
\begin{table}[ht!]
\centering
\begin{tabularx}{0.9\linewidth}{@{~~}c@{~~}|C@{~}C@{~}|C@{~}C@{~}}
\toprule
\multicolumn{1}{c|}{~} & \multicolumn{2}{c|}{$\IP^2$ basis} & \multicolumn{2}{c}{$\IP^3$ basis}  \\
\toprule
$\Delta x$ & $\|\vec{U}_h^{N_T} - \vec{U}(T)\|_{L_h^1}$ & rate & $\|\vec{U}_h^{N_T} - \vec{U}(T)\|_{L_h^1}$ & rate \\
\midrule
$1/25$  & $3.687\times 10^{-3}$ & ---   & $5.040\times 10^{-3}$ & ---   \\
$1/50$  & $3.455\times 10^{-4}$ & 3.416 & $3.710\times 10^{-5}$ & 7.086 \\
$1/100$ & $4.248\times 10^{-5}$ & 3.024 & $1.177\times 10^{-6}$ & 4.979 \\
\bottomrule
\end{tabularx}
\caption{The discrete $L_h^1$ errors and convergence rates for $\IP^2$ and $\IP^3$ spaces. The $\ell^1$ limiter is solved by DRS method with parameter $\gamma = 10^{-3}$}.
\label{tab:convergence_Euler_test_l1}
\end{table}

\paragraph{\bf Example~4.4 (Sedov blast wave)}
This test describes a strong shock expanding from a point-source explosion in a uniform medium and involves a strong shock and low density, which makes it of great utility as a verification test for the robustness of a simulator for compressible Euler equations \cite{LSZ2025cRKDG}.
\par
We consider the following configuration on computational domain $\Omega = [0, 1.1]^2$ with the simulation end time $T = 1$. 
Uniformly partition the domain $\Omega$ into square cells with mesh resolution $\Delta x = 1.1/160$.
The initials are defined as piecewise constants: the density $\rho^0 = 1$, the momentum $\vec{m}^0 = \transpose{(0,0)}$, and the total energy $E^0$ is set to $10^{-12}$ everywhere, except in the cell located at the lower left corner, where $0.244816/\Delta{x}^2$ is used.
Reflective boundary conditions are imposed on the left and bottom boundaries, and outflow boundary conditions are imposed on the right and top boundaries.
\par
Our optimization-based limiters can enforce invariant domain for SSP RKDG schemes with large time step sizes. 
Compared to the Zhang-Shu method for enforcing positivity \cite{zhang2010positivity}, 
the advantage of integrating this optimization-based limiter with SSP RKDG methods is that the restrictive CFL from weak positivity in Zhang-Shu method \cite{zhang2017positivity} is no longer needed thus  implementation is simplified.
In this example, we employ a three-stage third-order SSP RK method in conjunction with $\IP^2$ modal DG to solve the compressible Euler equation \eqref{eq:compressible_Euler}. Our time marching scheme is as follows:
\begin{subequations}\label{eq:SSPRK33}
\begin{align}
\vec{U}^{(2)} &= \vec{U}^n - \Delta t \div{\vec{F}^\mathrm{a}(\vec{U}^n)},\\
\vec{U}^{(3)} &= \frac{3}{4}\vec{U}^n + \frac{1}{4}\Big[\vec{U}^{(2)} - \Delta t \div{\vec{F}^\mathrm{a}(\vec{U}^{(2)})}\Big],\\
\vec{U}^{n+1} &= \frac{1}{3}\vec{U}^n + \frac{2}{3}\Big[\vec{U}^{(3)} - \Delta t \div{\vec{F}^\mathrm{a}(\vec{U}^{(3)})}\Big].
\end{align}
\end{subequations}
We use the Lax–Friedrichs flux \cite{zhang2010positivity,zhang2017positivity}. The $\IP^2$ basis functions are constructed using orthogonal Legendre polynomials. Numerical integrations are evaluated by the tensor product of $3$-point Gauss quadrature.
\par
When combining a DG method of $k$-th degree basis with a $(k+1)$-th order explicit RK scheme, the following expression provides an estimate of the effective CFL number for linear stability.
\begin{align}\label{eq:num_test_CFL}
\mathrm{CFL} = \frac{1}{2k+1}, \quad\text{for}~\IP^k~\text{or}~\IQ^k~\text{scheme}. 
\end{align}
Here, we use $\mathrm{CFL} = 0.2$ for $\IP^2$ scheme, which permits greater time step sizes than required for weak positivity, that is, negative cell averages are produced during the simulation.
After each stage in \eqref{eq:SSPRK33}, if out-of-bound cell averages appear, we apply our optimization-based postprocessing to modify the DG polynomial. 

An ad-hoc strategy to improve efficiency and robustness of the optimization-based limiters is to avoid modifying regions of the computational domain which is not reached by the shock wave. 
Following the approach in \cite[Remark 5]{LZ2024CNS}, we define a subset of domain $\Omega$ as follows: 
\begin{align}
\Omega^\mathrm{lim} = \left\{\mathrm{cell}~K_i: ~\mathrm{either}~~ \overline{\vec{U}_h}|_{K_i}\notin G^\varepsilon 
~~\mathrm{or}~~ \overline{E_h}|_{K_i} - \frac{1}{2}\,\frac{\norm{\overline{\vec{m}_h}|_{K_i}}{2}^2}{\overline{\rho_h}|_{K_i}} \geq 10^{-10}   \right\}.
\end{align}
In this example, the optimization-based postprocessing is applied only to the cells in $\Omega^\mathrm{lim}$, thereby avoiding the redistribution of mass or energy into regions where the shock wave has not yet propagated to.

\par
Figure~\ref{fig:sedov_2D} shows snapshot of the density  at time $T=1$. 
Our scheme preserves global conservation and the invariant domain. The shock location is correctly captured. The optimization cell average limiter is triggered only at one time step during time evaluation, and the performance of the splitting methods solving the minimizations is shown in Figure~\ref{fig:sedov_2D}. 
Besides the optimization based cell average limiter and Zhang-Shu limiter \cite{zhang2010positivity} for correcting quadrature point values, no additional limiters are used to remove oscillations.

\begin{remark}
    We can see that the $\ell^1$ limiter is much more expensive to compute than the $\ell^2$ limiter for 2D Sedov blast wave test. However, we will see that  $\ell^1$ limiter may not always be so expensive in the next example of Mach 2000 astrophysical jet. 
\end{remark}

\begin{figure}[ht!]
\centering
\begin{tabularx}{0.925\linewidth}{@{}c@{~~}c@{~}c@{~~}c@{}}
\begin{sideways}{$\hspace{1.75cm} \text{\small $\ell^2$ limiter} \quad$}\end{sideways} &
\includegraphics[width=0.3\textwidth]{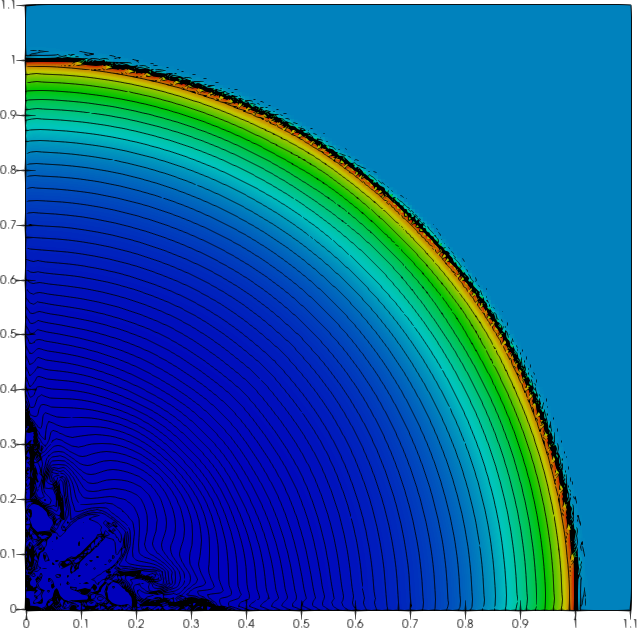} &
\includegraphics[width=0.044\textwidth]{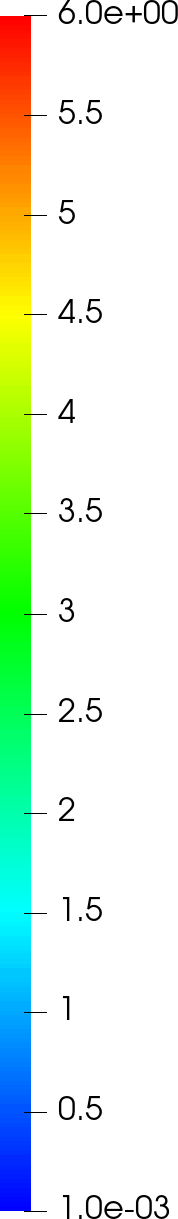} &
\includegraphics[width=0.49\textwidth]{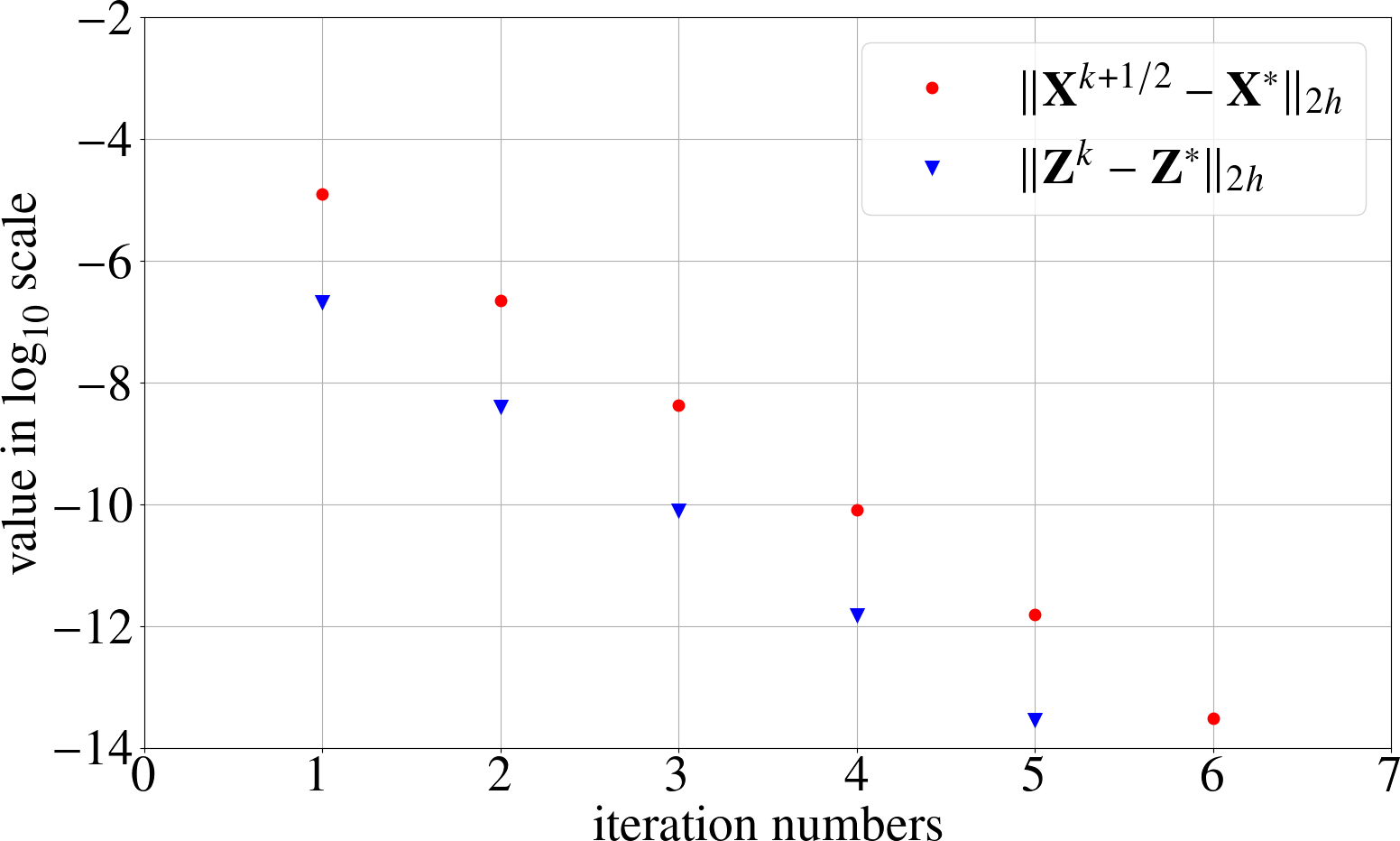} \\
\begin{sideways}{$\hspace{1.75cm} \text{\small $\ell^1$ limiter} \quad$}\end{sideways} &
\includegraphics[width=0.3\textwidth]{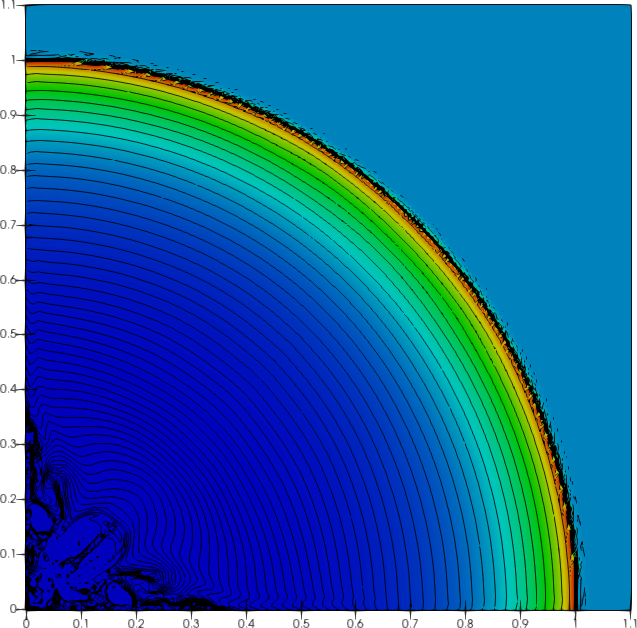} &
\includegraphics[width=0.044\textwidth]{color_bar_sedov.png} &
\includegraphics[width=0.49\textwidth]{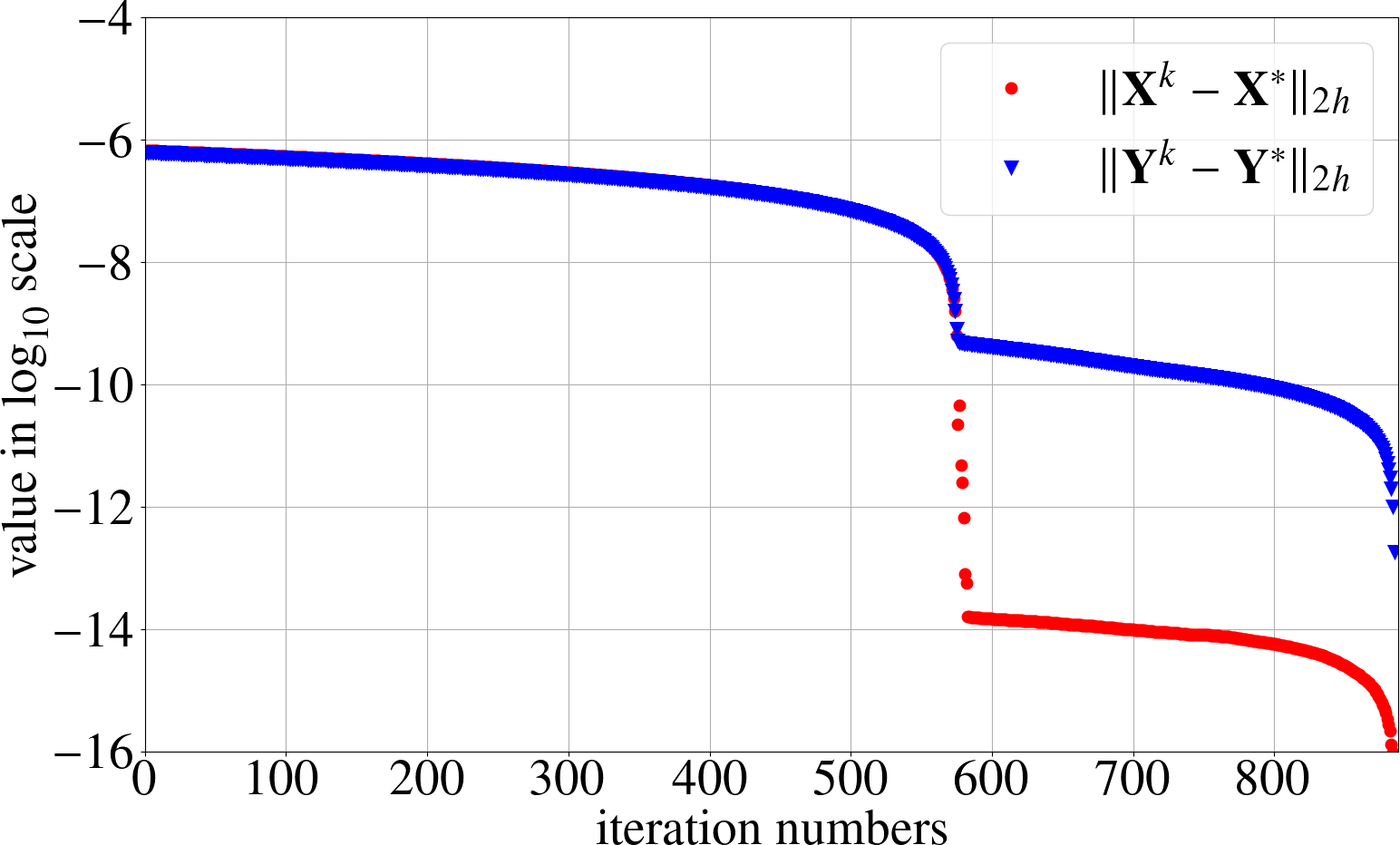} \\
\end{tabularx}
\caption{Sedov blast wave. Left: snapshot of density at $T = 1$. Plot of density: $50$ exponentially distributed contour lines of density from $0.001$ to $6$. Right: number of DYS and DRS iterations in processing out-of-bound data of $\ell^2$ and $\ell^1$ model in the second RK stage at time step $74669$, which is the only time step (during the whole time evolution) where the cell averages violate the invariant domain.}
\label{fig:sedov_2D}
\end{figure}

\paragraph{\bf Example~4.5 (High Mach number astrophysical jet)}
The numerical replication of gas flows and shock wave patterns observed in Hubble Space Telescope images presents a demanding benchmark for computational methods, particularly when extreme Mach numbers are involved \cite{gardner2009numerical,zhang2010positivity}. We validate the robustness of our method by simulating Mach $2000$ astrophysical jets governed by compressible Euler equations without radiative cooling.
\par
Let us consider a high Mach number astrophysical jet configuration on the computational domain $\Omega = [0,1]\times[-0.5,0.5]$. We set the simulation end time $T = 0.001$.
We take the ideal gas constant $\gamma = 5/3$. The initial density $\rho^0 = 0.5$, momentum $\vec{m}^0 = \vec{0}$, and pressure $p^0 = 0.4127$.
The left edge of domain $\Omega$ is prescribed with the following inflow boundary condition:
\begin{align}
\transpose{[\rho, m_x, m_y, p]} = 
\begin{cases}
\transpose{[5,\, 4000,\, 0,\, 0.4127]} & \text{if}~x=0~\text{and}~\abs{y}\leq 0.05,\\
\transpose{[0.5,\, 0,\, 0,\, 0.4127]} & \text{if}~x=0~\text{and}~\abs{y}> 0.05,
\end{cases}
\end{align}
while the top, right, and bottom edges of domain $\Omega$ are treated as outflow boundaries.
\par
One advantage of our optimization-based invariant-domain-preserving approach is its applicability to numerical schemes without strong stability preserving (SSP) properties.   
In this example, we employ the classical  four-stage fourth-order RK (RK$4$) method, which is not a SSP RK method, in conjunction with the $\IQ^3$ DG spectral element method (SEM) to discretize the Euler equations \eqref{eq:compressible_Euler}.
\begin{subequations}\label{eq:RK4}
\begin{align}
\vec{U}^{(2)} &= \vec{U}^n - \frac{\Delta t}{2} \div{\vec{F}^\mathrm{a}(\vec{U}^n)},\\
\vec{U}^{(3)} &= \vec{U}^n - \frac{\Delta t}{2} \div{\vec{F}^\mathrm{a}(\vec{U}^{(2)})},\\
\vec{U}^{(4)} &= \vec{U}^n - \Delta{t} \div{\vec{F}^\mathrm{a}(\vec{U}^{(3)})},\\
\vec{U}^{n+1} &= \vec{U}^n - \Delta{t} \Big[\frac{1}{6}\div{\vec{F}^\mathrm{a}(\vec{U}^n)} + \frac{1}{3}\div{\vec{F}^\mathrm{a}(\vec{U}^{(2)})} + \frac{1}{3}\div{\vec{F}^\mathrm{a}(\vec{U}^{(3)})} + \frac{1}{6}\div{\vec{F}^\mathrm{a}(\vec{U}^{(4)})} \Big],
\end{align}
\end{subequations}
We uniformly partition domain $\Omega$ into square cells with mesh resolution $\Delta x = 1/640$. For spatial discretization, we employ the Lax–Friedrichs flux. The Lagrange bases of $\IQ^3$ DG SEM are constructed via the tensor product of $4$-point Gauss-Lobatto point and numerical integrations are evaluated by the tensor product of $4$-point Gauss-Lobatto rule.
\par
The fourth-order RK method is not SSP, and when combined with $\IQ^3$ DG SEM, the resulting discretization lacks weak positivity (see \cite{zhang2017positivity} for its definition), i.e., negative cell averages may appear during the simulation. To enforce invariant domain without losing conservation and high order accuracy, we apply our optimization-based postprocessing to modify the DG polynomial with a CFL number $\frac{1}{7}$ as determined by \eqref{eq:num_test_CFL} after each stage in \eqref{eq:RK4} whenever any out-of-bound cell average is produced.  The optimization-based limiters are applied to the whole computational domain for this test. 
\par
 Figure~\ref{fig:astrophysical_jet_density} shows snapshots of the density field at time $T = 0.001$. Compared to the DYS method in $\ell^2$ limiter, it requires more number of projections to admissible set for DRS method to solve the minimization problem in $\ell^1$ limiter. When combined with Zhang-Shu limiter, both optimization-based cell average limiter ($\ell^2$ and $\ell^1$) robustly produce simulation results with correct shock location without other limiters for reducing oscillations. These computational results are consistent with high order DG methods stabilized by conventional non-optimization based limiters, see \cite{LZ2022CNS}.
\begin{figure}[ht!]
\centering
\begin{tabularx}{0.975\linewidth}{@{}c@{~}c@{~}c@{~}c@{}}
\begin{sideways}{$\hspace{1.2cm} \text{\small $\ell^2$ limiter} \quad$}\end{sideways} &
\includegraphics[width=0.48\textwidth]{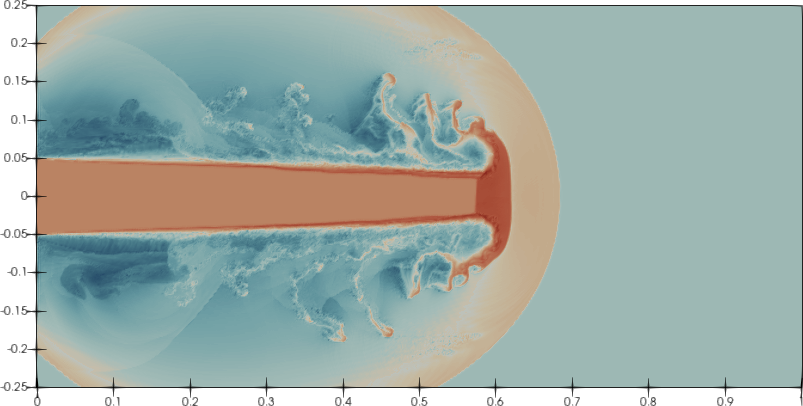} &
\includegraphics[width=0.058\textwidth]{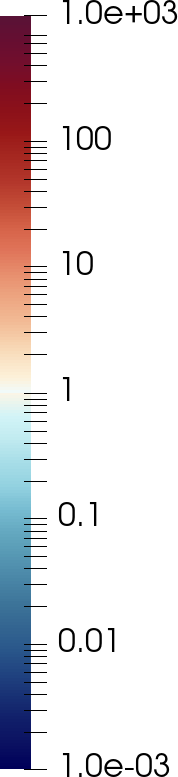} &
\includegraphics[width=0.41\textwidth]{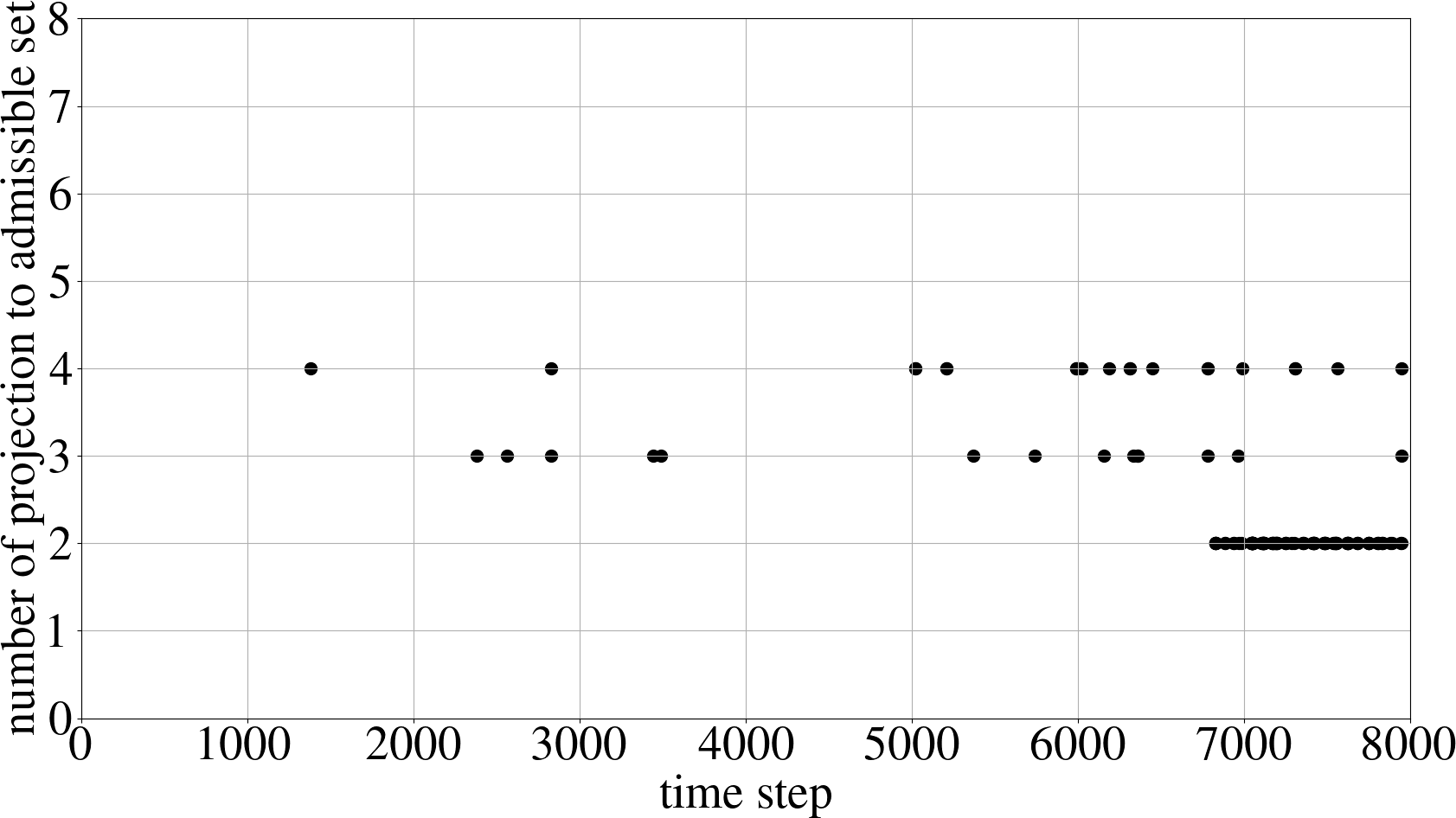} \\
\begin{sideways}{$\hspace{1.2cm} \text{\small $\ell^1$ limiter} \quad$}\end{sideways} &
\includegraphics[width=0.48\textwidth]{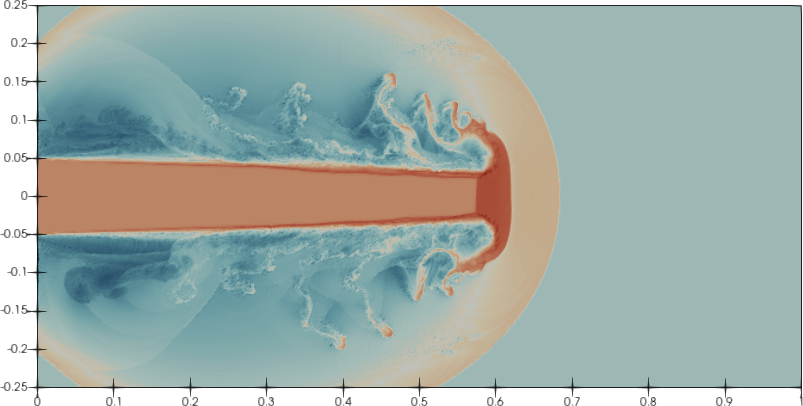} &
\includegraphics[width=0.058\textwidth]{astro_color_bar.png} &
\includegraphics[width=0.41\textwidth]{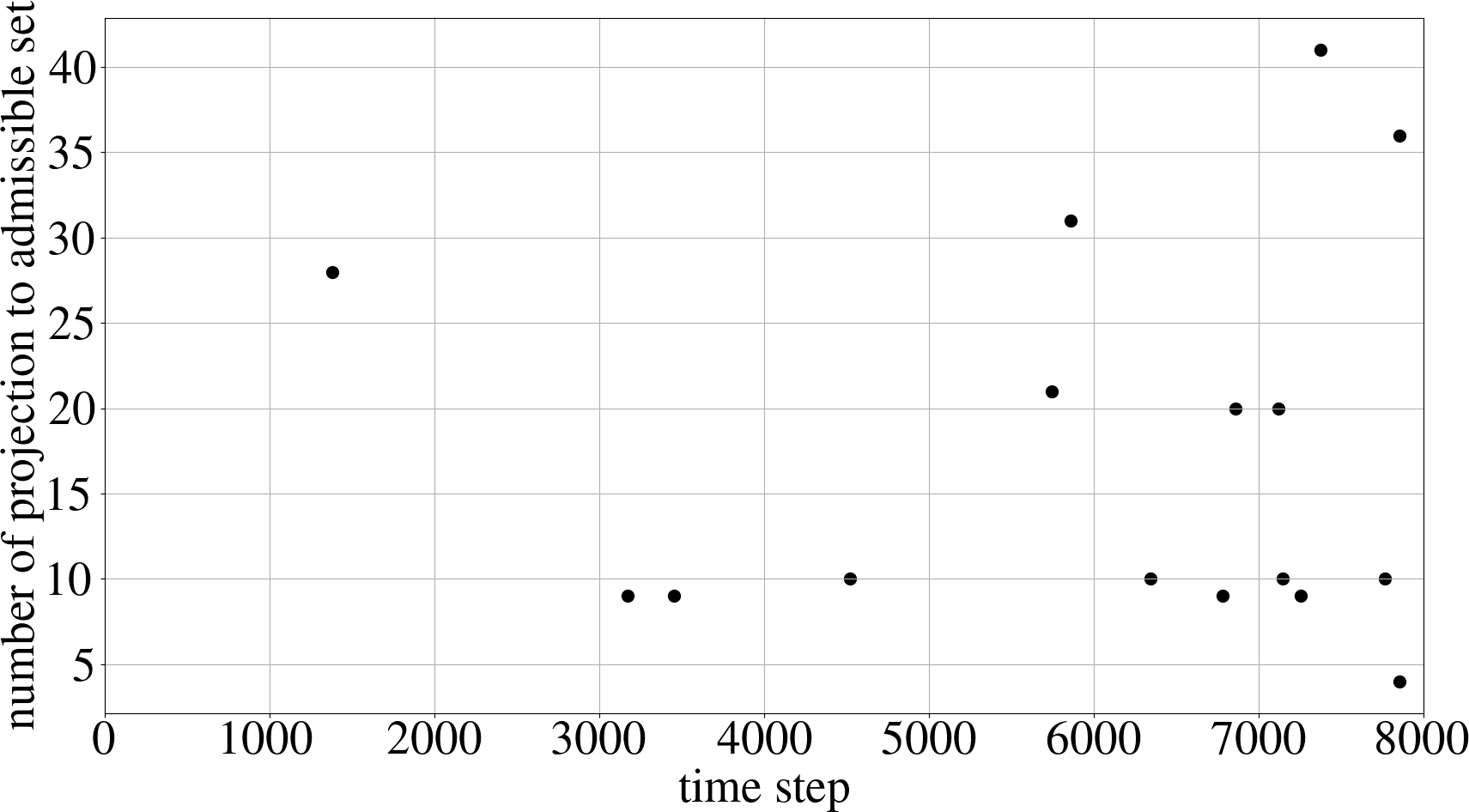} \\
\end{tabularx}
\caption{The Mach 2000 astrophysical jet.
Top row: $\ell^2$-norm limiter solved by with DYS. Bottom row: $\ell^1$-norm limiter solved by DRS nested with DYS.
Left: snapshot of density at $T=0.001$ with logarithmic color scale. Right: number of projections to admissible set is number of computations of proximal operators to the indicator function $\iota_{\Lambda_2}(\vecc{X})$ in the optimization solvers during each step of the fourth-order RK time marching, which is a fair way to compare computational cost between two optimization-based limiters.}
\label{fig:astrophysical_jet_density}
\end{figure}

\begin{remark}
For the  astrophysical jet problem as shown in Figure \ref{fig:astrophysical_jet_density}, the computational cost of solving the $\ell^1$  limiter \eqref{eq:opt_model_l1_Euler} is  still higher than the $\ell^2$ limiter \eqref{eq:invariant_domain_limiter3} whenever they are triggered. However, we can see that the cell average limiter is triggered less if using the $\ell^1$  limiter, even though the $\ell^1$  minimizer is not sparse as explained in previous sections. In other words, the total cost of  the $\ell^1$ limiter does not seem significantly higher than the $\ell^2$ limiter in this test, and more importantly the $\ell^1$ limiter is triggered less during the time evolution, which is more desirable. 
\end{remark}
%%%%%%%%%%%%%%%%%%%%%%%%%%%%%%%%%%%%%%%%%%%%%%%%%%%%%%%%%%%%%%%%%%%%%%%%%%%%%%%%%%%%%%%%%%%%%%%%%%%%%%%%%%%%%%%%%%%%%%%
\section{Concluding remarks}\label{sec:concluding_remarks}
%%%%%%%%%%%%%%%%%%%%%%%%%%%%%%%%%%%%%%%%%%%%%%%%%%%%%%%%%%%%%%%%%%%%%%%%%%%%%%%%%%%%%%%%%%%%%%%%%%%%%%%%%%%%%%%%%%%%%%%
In this paper, we have developed effective splitting methods for implementing $\ell^1$-norm and $\ell^2$-norm 
optimization based limiters  for enforcing the invariant domain in gas dynamics.
The proposed postprocessing for high order DG schemes consists of first applying  an optimization-based limiter to cell averages then by the Zhang-Shu limiter to modify quadrature point values. It is broadly applicable to DG methods with  various time-marching schemes, e.g., non-SSP Runge-Kutta methods.  Even though the optimization-based cell average limiters only preserve global conservation but not local conservation, the numerical tests suggest that it does not produce wrong shock locations for demanding problems such as blast waves and high speed flows. In practice, the $\ell^1$-norm limiter produces almost the same result as the  $\ell^2$-norm limiter. In general, $\ell^2$-norm limiter is cheaper to implement since it can be solved efficiently by the DYS method. On the other hand, for the astrophysical jet problem, we observe that the $\ell^1$-norm limiter is triggered less during time evolution, thus the $\ell^1$-norm limiter seems a better choice for this kind of problems.

%%%%%%%%%%%%%%%%%%%%%%%%%%%%%%%%%%%%%%%%%%%%%%%%%%%%%%%%%%%%%%%%%%%%%%%%%%%%%%%%%%%%%%%%%%%%%%%%%%%%%%%%%%%%%%%%%%%%%%%
\section*{CRediT authorship contribution statement}
\textbf{Chen Liu:} Writing - Review \& Editing, Writing – original draft, Methodology, Software, Visualization, Validation, Formal analysis, Investigation, Funding acquisition.
\textbf{Dionysis Milesis:} Writing - Review \& Editing, Formal analysis, Investigation.
\textbf{Xiangxiong Zhang:} Writing - Review \& Editing, Methodology, Formal analysis, Investigation, Conceptualization, Funding acquisition, Supervision, Project administration.
\textbf{Chi-Wang Shu:} Writing - Review \& Editing, Supervision, Methodology, Funding acquisition.

\section*{Data availability}
Data will be made available on request.

\section*{Declaration of competing interest}
The authors declare that they have no known competing financial interests or personal relationships that could have appeared to influence the work reported in this paper.

%%%%%%%%%%%%%%%%%%%%%%%%%%%%%%%%%%%%%%%%%%%%%%%%%%%%%%%%%%%%%%%%%%%%%%%%%%%%%%%%%%%%%%%%%%%%%%%%%%%%%%%%%%%%%%%%%%%%%%%
\section*{Acknowledgements}
C. Liu is supported by NSF DMS-2513106 and the Division of Research and Innovation (DRI) startup fund from University of Arkansas. C.-W. Shu is partially supported by NSF DMS-2309249. X. Zhang is partially supported by NSF DMS-2208518.

%% Appendix.
\appendix
\renewcommand{\thesection}{Appendix \Alph{section}}  % Set appendix format Appendix A.

%%%%%%%%%%%%%%%%%%%%%%%%%%%%%%%%%%%%%%%%%%%%%%%%%%%%%%%%%%%%%%%%%%%%%%%%%%%%%%%%%%%%%%%%%%%%%%%%%%%%%%%%%%%%%%%%%%%%%%%
\section{Computing proximal operator in $\ell^1$ scalar model}\label{appendix:compute_prox}
%%%%%%%%%%%%%%%%%%%%%%%%%%%%%%%%%%%%%%%%%%%%%%%%%%%%%%%%%%%%%%%%%%%%%%%%%%%%%%%%%%%%%%%%%%%%%%%%%%%%%%%%%%%%%%%%%%%%%%%
\begin{lemma}\label{lem:1}
For any given numbers $\gamma>0$ and $a\in\IR$, define $\mathrm{S}_\gamma(a) = \mathrm{sgn}(a)\max\{\abs{a}-\gamma, 0\}$. 
Let $f(z) = \gamma\abs{z} + \frac{1}{2}(z-a)^2$. We have: if $z > \mathrm{S}_\gamma(a)$, then $f(z)$ is strictly increasing; if $z < \mathrm{S}_\gamma(a)$, then $f(z)$ is strictly decreasing.
\end{lemma}
\begin{proof}
If $z > \mathrm{S}_\gamma(a)$, we have the following scenarios:
\begin{itemize}
\item Case~1: $a\geq0$ and $\abs{a}>\gamma$. \\
Then, we have $z>a-\gamma>0~\Rightarrow~f(z)=\gamma z + \frac{1}{2}(z-a)^2$, namely $f'(z) = z+\gamma-a>0$.
\item Case~2: $a\geq0$ and $\abs{a}\leq\gamma$. \\
Then, we have $z>0~\Rightarrow~f(z)=\gamma z + \frac{1}{2}(z-a)^2$, namely $f'(z) = z+\gamma-a>0$.
\item Case~3: $a<0$ and $\abs{a}>\gamma$. Then, we have $z>a+\gamma$. We discuss this in two sub-cases
\begin{itemize}
\item Sub-case~1: $z\geq0~\Rightarrow~f(z)=\gamma z + \frac{1}{2}(z-a)^2$, namely $f'(z) = z+\gamma-a>0$.
\item Sub-case~2: $a+\gamma<z<0\Rightarrow~f(z)=-\gamma z + \frac{1}{2}(z-a)^2$, namely $f'(z) = z-\gamma-a>0$.
\end{itemize}
\item Case~4: $a<0$ and $\abs{a}\leq\gamma$. \\
Then, we have $z>0~\Rightarrow~f(z)=\gamma z + \frac{1}{2}(z-a)^2$, namely $f'(z) = z+\gamma-a>0$.
\end{itemize}
Above all, when $z > \mathrm{S}_\gamma(a)$, the function $f(z)$ is increasing.
Similarly, if $z < \mathrm{S}_\gamma(a)$, we have the following scenarios:
\begin{itemize}
\item Case~1: $a\geq0$ and $\abs{a}>\gamma$. \\
Then, we have $z<a-\gamma$. We discuss this in two sub-cases
\begin{itemize}
\item Sub-case~1: $z\leq0~\Rightarrow~f(z)=-\gamma z + \frac{1}{2}(z-a)^2$, namely $f'(z) = z-\gamma-a<0$.
\item Sub-case~2: $0<z<a-\gamma\Rightarrow~f(z)=\gamma z + \frac{1}{2}(z-a)^2$, namely $f'(z) = z+\gamma-a<0$.
\end{itemize}
\item Case~2: $a\geq0$ and $\abs{a}\leq\gamma$. \\
Then, we have $z\leq0~\Rightarrow~f(z)=-\gamma z + \frac{1}{2}(z-a)^2$, namely $f'(z) = z-\gamma-a<0$.
\item Case~3: $a<0$ and $\abs{a}>\gamma$. \\
Then, we have $z<a+\gamma<0~\Rightarrow~f(z)=-\gamma z + \frac{1}{2}(z-a)^2$, namely $f'(z) = z-\gamma-a<0$.
\item Case~4: $a<0$ and $\abs{a}\leq\gamma$. \\
Then, we have $z<0~\Rightarrow~f(z)=-\gamma z + \frac{1}{2}(z-a)^2$, namely $f'(z) = z-\gamma-a<0$.
\end{itemize}
Above all, when $z < \mathrm{S}_\gamma(a)$, the function $f(z)$ is decreasing.
\end{proof}
\begin{lemma}
For scalar conservation laws, recall the set $\Lambda_2 = \{\vec{x}\in\IR^N\!:\, x_i \in [m, M],~ \forall i = 1,\cdots, N\}$ and in \eqref{eq:scalar_l1_gh_alter} the function $g(\vec{x}) = \norm{\vec{x}-\vec{u}}{1} + \iota_{\Lambda_2}(\vec{x})$. We have
\begin{align}
[\prox_g^\gamma(\vec{x})]_i = \max\{\min\{u_i + \mathrm{S}_\gamma(x_i-u_i), M\}, m\},
\end{align}
where the subscript $i$ denotes the $i$-th component in corresponding vector.
\end{lemma}
\begin{proof}
By definition
\begin{align}
\prox_g^\gamma(\vec{x}) 
&= \underset{\vec{z}}{\mathrm{argmin}}\, \Big\{\gamma\norm{\vec{z}-\vec{u}}{1} + \gamma \iota_{\{\vec{z}:\,m \leq\vec{z}\leq M\}} + \frac{1}{2}\norm{\vec{z}-\vec{x}}{2}^2\Big\} \nonumber\\
&= \underset{\vec{z}}{\mathrm{argmin}}\, \Big\{\gamma\norm{\vec{z}-\vec{u}}{1} + \frac{1}{2}\norm{\vec{z}-\vec{u}-(\vec{x}-\vec{u})}{2}^2 + \gamma \iota_{\{\vec{z}:\,m-\vec{u} \leq\vec{z}-\vec{u}\leq M-\vec{u}\}}\Big\}\nonumber\\
&= \vec{u} + \underset{\vec{z}}{\mathrm{argmin}}\, \Big\{\gamma\norm{\vec{z}}{1} + \frac{1}{2}\norm{\vec{z}-(\vec{x}-\vec{u})}{2}^2 + \gamma \iota_{\{\vec{z}:\,m-\vec{u} \leq\vec{z}\leq M-\vec{u}\}}\Big\}\nonumber\\
&= \vec{u} + \underset{\vec{z}}{\mathrm{argmin}}\, \sum_{i=1}^n\Big(\gamma\abs{z_i} + \frac{1}{2}\abs{z_i-(x_i-u_i)}^2 + \gamma\iota_{\{z_i:\,m-u_i \leq z_i\leq M-u_i\}}\Big).
\end{align}
Therefore, we have
\begin{align}
[\prox_g^\gamma(\vec{x})]_i 
=  u_i + \underset{z_i}{\mathrm{argmin}}\,\Big\{ \gamma\abs{z_i} + \frac{1}{2}\abs{z_i-(x_i-u_i)}^2 + \gamma\iota_{\{z_i:\,m-u_i \leq z_i\leq M-u_i\}}\Big\}.
\end{align}
The following result is known, see \cite[page~209 soft-thresholding operator]{demanet2016eventual}. 
\begin{align}
\underset{z_i}{\mathrm{argmin}}\,\Big\{ \gamma\abs{z_i} + \frac{1}{2}\abs{z_i-(x_i-u_i)}^2\Big\} = \mathrm{S}_\gamma(x_i-u_i).
\end{align}
Therefore, we have
\begin{itemize}
\item If $\mathrm{S}_\gamma(x_i-u_i)\!\in\![m-u_i,M - u_i]$, namely $u_i+\mathrm{S}_\gamma(x_i-u_i) \!\in\! [m,M]$, then $[\prox_g^\gamma(\vec{x})]_i = u_i + \mathrm{S}_\gamma(x_i-u_i)$.
\item If $\mathrm{S}_\gamma(x_i-u_i) < m-u_i$, namely $u_i + \mathrm{S}_\gamma(x_i-u_i) < m$, then by Lemma~\ref{lem:1}, we know that the function $f(z) = \gamma\abs{z} + \frac{1}{2}\abs{z-(x_i-u_i)}^2$ is strictly increasing when $z>\mathrm{S}_\gamma(x_i-u_i)$. Notice $\mathrm{S}_\gamma(x_i-u_i) < m-u_i < M-u_i$, we have $[\prox_g^\gamma(\vec{x})]_i = u_i + (m - u_i) = m$.
\item If $\mathrm{S}_\gamma(x_i-u_i) > M-u_i$, namely $u_i + \mathrm{S}_\gamma(x_i-u_i) > M$, then by Lemma~\ref{lem:1}, we know that the function $f(z) = \gamma\abs{z} + \frac{1}{2}\abs{z-(x_i-u_i)}^2$ is strictly decreasing when $z<\mathrm{S}_\gamma(x_i-u_i)$. Notice $m-u_i < M-u_i < \mathrm{S}_\gamma(x_i-u_i)$. We have $[\prox_g^\gamma(\vec{x})]_i = u_i + (M - u_i) = M$.
\end{itemize}
Above all, we conclude the proof.
\end{proof}

\section{Projection to numerical admissible set in one dimension}\label{sec:1D_proj}
\subsection{Derivation from KKT conditions}
Let $\transpose{[u, v, w]}$ denote an out-of-bound cell average produced by a simulation in one-dimensional domain.
For a small prescribed $\varepsilon>0$, the set $G^\varepsilon$ in \eqref{invariant-domain} becomes 
\begin{align}
G^\varepsilon = \left\{\vec{U} = \transpose{[\rho,m,E]}\!:~ \rho \geq \varepsilon,~~ \rho e(\vec{U}) = E - \frac{m^2}{2\rho} \geq \varepsilon\right\}.
\end{align} 
Since the set $G^\varepsilon$ is closed and convex, the projection is unique. To find the projection of $\transpose{[u, v, w]}$ on set $G^\varepsilon$, we need to solve the following minimization problem: find $\transpose{[\rho, m, E]}$ that
\begin{align}
\min_{\rho,m,E}&~~ \frac{1}{2}\big(\abs{\rho - u}^2 + \abs{m - v}^2 + \abs{E - w}^2\big) \nonumber\\
\text{subject~to:}&~~ \varepsilon - \rho \leq 0 
\quad\text{and}\quad \varepsilon - E + \frac{m^2}{2\rho} \leq 0.
\end{align}

The exact solution can be derived from the Lagrangian 
\[L(\rho, m, E, \lambda, \mu)=\frac{1}{2}\big(\abs{\rho - u}^2 + \abs{m - v}^2 + \abs{E - w}^2\big)+\lambda (\varepsilon-\rho)+\mu(\varepsilon - E + \frac{m^2}{2\rho}),\]
and its  KKT conditions:
\begin{subequations}\label{eq:kkt_1d}
\begin{itemize}
\item Stationarity condition
\begin{align}
\rho - u - \lambda - \mu\frac{m^2}{2\rho^2} = 0,\label{eq:kkt_1d_1a}\\
m - v + \mu \frac{m}{\rho} = 0,\label{eq:kkt_1d_1b}\\
E - w - \mu = 0.\label{eq:kkt_1d_1c}
\end{align}
\item Complementary slackness
\begin{align}
\lambda(\varepsilon - \rho) = 0,\label{eq:kkt_1d_1d}\\
\mu\Big(\varepsilon - E + \frac{m^2}{2\rho}\Big) = 0.\label{eq:kkt_1d_1e}
\end{align}
\item Primal feasibility
\begin{align}
\varepsilon - \rho \leq 0, \label{eq:kkt_1d_f}\\
\varepsilon - E + \frac{m^2}{2\rho} \leq 0.\label{eq:kkt_1d_1g}
\end{align}
\item Dual feasibility
\begin{align}
\lambda \geq 0, \label{eq:kkt_1d_1h}\\
\mu \geq 0.\label{eq:kkt_1d_1i}
\end{align}
\end{itemize}
\end{subequations}
From the dual feasibility, depending on whether $\lambda$ and $\mu$ are zero or positive, there are four cases:
\paragraph{\bf Case~1 ($\mu = 0$ and $\lambda > 0$)}
When $\mu$ is zero, the \eqref{eq:kkt_1d_1a}-\eqref{eq:kkt_1d_1c} gives $\rho = u + \lambda$, $m = v$, and $E = w$. When $\lambda > 0$, \eqref{eq:kkt_1d_1d} implies $\rho = \varepsilon$. From \eqref{eq:kkt_1d_1a}, we get $\lambda = \varepsilon - u$. 
Thus, if $\transpose{[\varepsilon, v, w]}$ is a solution of \eqref{eq:kkt_1d}, the constraints $\varepsilon > u$ and \eqref{eq:kkt_1d_1g} need to be satisfied.

\paragraph{\bf Case~2 ($\mu = 0$ and $\lambda = 0$)}
When both $\mu$ and $\lambda$ are zero, the \eqref{eq:kkt_1d_1a}-\eqref{eq:kkt_1d_1c} gives: $\rho = u$, $m = v$, and $E = w$. 
Thus, if $\transpose{[u, v, w]}$ is a solution of \eqref{eq:kkt_1d}, the given point should belong to the set $G^\varepsilon$, i.e., the projection point is itself.

\paragraph{\bf Case~3 ($\mu > 0$ and $\lambda > 0$)}
When both $\mu$ and $\lambda$ are positive, the \eqref{eq:kkt_1d_1d} gives $\rho = \varepsilon$ and \eqref{eq:kkt_1d_1e} becomes
\begin{align}\label{eq:1d_case3_E}
E = \frac{1}{2\varepsilon}m^2 + \varepsilon.
\end{align}
If $v = 0$, by \eqref{eq:kkt_1d_1b}, we have $m = 0$, which further implies $E = \varepsilon$. From \eqref{eq:kkt_1d_1c}, we get $\mu = \varepsilon - w$ and from \eqref{eq:kkt_1d_1a}, we get $\lambda = \varepsilon - u$.
Thus, if $\transpose{[\varepsilon, 0, \varepsilon]}$ is a solution of \eqref{eq:kkt_1d}, the constraints $\varepsilon > u$ and $\varepsilon > w$ need to be satisfied.
\par
If $v \neq 0$, then $m \neq 0$. Otherwise, \eqref{eq:kkt_1d_1b} gives $v = 0$, which leads to a contradiction.
Substituting $\rho = \varepsilon$ into \eqref{eq:kkt_1d_1b} to solve $\mu$, and further substituting the result into \eqref{eq:kkt_1d_1a}, we have 
\begin{align}\label{eq:1d_case3_mu_lamb}
\mu = \varepsilon\Big(\frac{v}{m} - 1\Big)
\quad\text{and}\quad
\lambda = \varepsilon - u - \frac{1}{2\varepsilon}m(v - m).
\end{align}
Combine \eqref{eq:kkt_1d_1c} with \eqref{eq:1d_case3_E} and \eqref{eq:1d_case3_mu_lamb}, we obtain
\begin{align}
\frac{1}{2\varepsilon}m^2 + \varepsilon = E = \mu + w = \varepsilon\Big(\frac{v}{m} - 1\Big) + w.
\end{align}
Multiplying $2\varepsilon m$ on both sides above, we get a cubic equation with respect to $m$ as follows.
\begin{align}
m^3 + (4\varepsilon^2 - 2\varepsilon w)m - 2\varepsilon^2 v = 0.
\end{align}
After getting all the real roots of using the method in \ref{appendix:cubic_root}, we compute $E$ from \eqref{eq:1d_case3_E}, and $\mu, \lambda$ from \eqref{eq:1d_case3_mu_lamb}.
Thus, if the result here is a solution of \eqref{eq:kkt_1d}, the constraints $\mu > 0$ and $\lambda > 0$ for \eqref{eq:1d_case3_mu_lamb} need to be satisfied.

\paragraph{\bf Case~4 ($\mu > 0$ and $\lambda = 0$)}
When $\mu$ is positive and $\lambda$ is zero, the KKT condition \eqref{eq:kkt_1d} simplify to $\rho \geq \varepsilon$ and
\begin{subequations}\label{eq:kkt_1d_2}
\begin{align}
\rho - u - \mu\frac{m^2}{2\rho^2} &= 0,\label{eq:kkt_1d_2a}\\
m - v + \mu \frac{m}{\rho} &= 0,\label{eq:kkt_1d_2b}\\
E - w - \mu &= 0,\label{eq:kkt_1d_2c}\\
\varepsilon - E + \frac{m^2}{2\rho} &= 0.\label{eq:kkt_1d_2d}
\end{align}
\end{subequations}
If $v = 0$, by \eqref{eq:kkt_1d_2b}, we have $m = 0$. By \eqref{eq:kkt_1d_2a} and \eqref{eq:kkt_1d_2d}, we get $\rho = u$ and $E = \varepsilon$. From \eqref{eq:kkt_1d_2c}, we get $\mu = \varepsilon - w$.
Thus, if $\transpose{[u, 0, \varepsilon]}$ is a solution of \eqref{eq:kkt_1d}, the constraints $\varepsilon > w$ and $\varepsilon \leq u$ need to be satisfied.
\par
If $v \neq 0$, then $m \neq 0$. Otherwise, \eqref{eq:kkt_1d_2b} gives $v = 0$, which leads to a contradiction.
Substituting into \eqref{eq:kkt_1d_2b} and \eqref{eq:kkt_1d_2a}, we have
\begin{align}
\rho - u - \frac{m}{2\rho}\Big(\mu\frac{m}{\rho}\Big) = 0 
\quad\Rightarrow\quad
\frac{m}{\rho}(v - m) = 2(\rho - u),
\end{align}
which is a quadratic equation with respect to $m$, namely $m^2 - vm + 2\rho^2 - 2u\rho = 0$. Thus, let $\Delta = -8\rho^2 + 8u\rho + v^2$ denote the discriminant. We have
\begin{align}\label{eq:1d_case4_roots_m}
m = \frac{v}{2} - \frac{1}{2}\sqrt{\Delta}
\quad\text{or}\quad
m = \frac{v}{2} + \frac{1}{2}\sqrt{\Delta}.
\end{align}
From \eqref{eq:kkt_1d_2c} and \eqref{eq:kkt_1d_2d}, we get 
\begin{align}
E &= \varepsilon + \frac{m^2}{2\rho},\\
\mu &= \varepsilon + \frac{m^2}{2\rho} - w.\label{eq:1d_case4_mu}
\end{align}
Therefore, once we compute $\rho$, the \eqref{eq:kkt_1d_2} is solved.
Substitute \eqref{eq:1d_case4_mu} into \eqref{eq:kkt_1d_2b}. After multiplying $2\rho^2$, we end up with a quadratic equation with respect to $\rho$, as follows. Notice that $v \neq 0$ implies $m \neq v$.
\begin{align}\label{eq:1d_case4_4}
2(m - v)\rho^2 + 2(\varepsilon - w)m\rho + m^3 = 0.
\end{align}
To solve $\rho$, let us first substitute $m = \frac{v}{2} - \frac{1}{2}\sqrt{\Delta}$ into above. The other root in \eqref{eq:1d_case4_roots_m} can be addressed in the same manner. We get 
\begin{align}\label{eq:1d_case4_5}
\sqrt{\Delta}(3v^2 + \Delta + 8(\varepsilon-w)\rho + 8\rho^2) = v^3 + 3v\Delta + 8(\varepsilon - w)v\rho - 8v\rho^2.
\end{align}
Recall the discriminant $\Delta = -8\rho^2 + 8u\rho + v^2$. Square both sides above and after some simplification, we obtain
\begin{align}\label{eq:case4_1d_equation_rho}
4\big(2v^2 + (\varepsilon + u - w)^2\big)\rho^2 
- 4u\big(2v^2 + (\varepsilon + u - w)^2\big)\rho 
+ 2uv^2(w-\varepsilon) - v^4 = 0.
\end{align}
The two roots of above quadratic equation are:
\begin{subequations}\label{eq:case4_1d_rho}
\begin{align}
\rho &= \frac{u}{2} + \frac{1}{2} \frac{\sqrt{u^2(2v^2 + (\varepsilon + u - w)^2) - 2uv^2(w-\varepsilon) + v^4}}{\sqrt{2v^2 + (\varepsilon + u - w)^2}},\\
\quad\text{or}\quad
\rho &= \frac{u}{2} - \frac{1}{2} \frac{\sqrt{u^2(2v^2 + (\varepsilon + u - w)^2) - 2uv^2(w-\varepsilon) + v^4}}{\sqrt{2v^2 + (\varepsilon + u - w)^2}}.
\end{align}
\end{subequations}
For the second root in \eqref{eq:1d_case4_roots_m}, substitute $m = \frac{v}{2} + \frac{1}{2}\sqrt{\Delta}$ into \eqref{eq:1d_case4_4}. After simplification, we end up with exactly the same equation as \eqref{eq:case4_1d_equation_rho}, namely the expressions in \eqref{eq:case4_1d_rho} are still the roots.
Thus, if the result here is a solution of \eqref{eq:kkt_1d}, the constraints $\mu > 0$ for \eqref{eq:1d_case4_mu} and $\rho \geq \varepsilon$ for \eqref{eq:case4_1d_rho} need to be satisfied.

\subsection{Algorithm of projection onto the invariant domain set}
\label{appendix:1D_projection}
Given $\transpose{[x,y,z]}\notin G^\varepsilon$, find the projection of $\transpose{[x,y,z]}$ on set $G^\varepsilon$.\\
Step~1. Compute candidate points from KKT conditions.
\begin{itemize}
\item The point $\transpose{[\rho,m,E]} = \transpose{[\varepsilon, y, z]}$ is a candidate, if $x < \varepsilon$ and $z-\frac{y^2}{2\varepsilon} \geq \varepsilon$.
\item The point $\transpose{[\rho,m,E]} = \transpose{[\varepsilon, 0, \varepsilon]}$ is a candidate, if $x < \varepsilon$, $y = 0$, and $z < \varepsilon$.
\item If $y\neq 0$, solve cubic equation $m^3 + (4\varepsilon^2 - 2\varepsilon z)m - 2\varepsilon^2 y = 0$ to obtain all real roots. Examine each real root individually. 
Let $m_r$ denote a real root and let $E_r = \frac{m_r^2}{2\varepsilon} + \varepsilon$. 
The point $\transpose{[\rho,m,E]} = \transpose{[\varepsilon, m_r, E_r]}$ is a candidate, if $\frac{y}{m_r} > 1$ and $2\varepsilon x + m_r(y-m_r) < 2\varepsilon^2$.
\item The point $\transpose{[\rho,m,E]} = \transpose{[x, 0, \varepsilon]}$ is a candidate, if $x\geq\varepsilon$, $y = 0$, and $z < \varepsilon$.
\item Compute $\rho_1$ and $\rho_2$ using the following formulas only if they are real.  
\begin{align*}
\rho_1 &= \frac{x}{2} + \frac{1}{2} \frac{\sqrt{x^2(2y^2 + (\varepsilon + x - z)^2) - 2xy^2(z-\varepsilon) + y^4}}{\sqrt{2y^2 + (\varepsilon + x - z)^2}},\\
\rho_2 &= \frac{x}{2} - \frac{1}{2} \frac{\sqrt{x^2(2y^2 + (\varepsilon + x - z)^2) - 2xy^2(z-\varepsilon) + y^4}}{\sqrt{2y^2 + (\varepsilon + x - z)^2}}.
\end{align*}
For real values of $\rho_1$ and $\rho_2$, compute $m_1$ and $m_2$ using the following formulas only if they are real.
\begin{align*}
m_1(\rho) = \frac{1}{2}y - \frac{1}{2}\sqrt{-8\rho^2 + 8x\rho + y^2}
\quad\text{and}\quad
m_2(\rho) = \frac{1}{2}y + \frac{1}{2}\sqrt{-8\rho^2 + 8x\rho + y^2}.
\end{align*}
Then, for real point:
\begin{itemize}
\item The point $\transpose{[\rho,m,E]} = \transpose{[\rho_1, m_1(\rho_1), \varepsilon + \frac{m_1(\rho_1)^2}{2\rho_1}]}$ is a candidate, if $\rho_1\geq\varepsilon$ and $\varepsilon + \frac{m_1(\rho_1)^2}{2\rho_1} > z$.
\item The point $\transpose{[\rho,m,E]} = \transpose{[\rho_2, m_1(\rho_2), \varepsilon + \frac{m_1(\rho_2)^2}{2\rho_2}]}$ is a candidate, if $\rho_2\geq\varepsilon$ and $\varepsilon + \frac{m_1(\rho_2)^2}{2\rho_2} > z$.
\item The point $\transpose{[\rho,m,E]} = \transpose{[\rho_1, m_2(\rho_1), \varepsilon + \frac{m_2(\rho_1)^2}{2\rho_1}]}$ is a candidate, if $\rho_1\geq\varepsilon$ and $\varepsilon + \frac{m_2(\rho_1)^2}{2\rho_1} > z$.
\item The point $\transpose{[\rho,m,E]} = \transpose{[\rho_2, m_2(\rho_2), \varepsilon + \frac{m_2(\rho_2)^2}{2\rho_2}]}$ is a candidate, if $\rho_2\geq\varepsilon$ and $\varepsilon + \frac{m_2(\rho_2)^2}{2\rho_2} > z$.
\end{itemize}
\end{itemize}
Step~2. Substitute all candidate points into the objective function
\begin{align*}
\abs{\rho - x}^2 + \abs{m - y}^2 + \abs{E - z}^2.
\end{align*}
Pick the point $\transpose{[\rho,m,E]}$ from the candidate points such that minimizes the value of the objective function, which gives the projection point.

\section{Projection to numerical admissible set in two dimensions}\label{sec:2D_proj}
\subsection{Derivation from KKT conditions}
Let $\transpose{[u, v_1, v_2, w]}$ denote an out-of-bound cell average produced by a simulation in two-dimensional domain.
For a small prescribed $\varepsilon>0$, the numerical admissible set $G^\varepsilon$ in \eqref{invariant-domain} becomes 
\begin{align}
G^\varepsilon = \left\{\vec{U} = \transpose{[\rho, m_1, m_2, E]}\!:~ \rho \geq \varepsilon,~~ \rho e(\vec{U}) = E - \frac{m_1^2 + m_2^2}{2\rho} \geq \varepsilon\right\}.
\end{align} 
The projection point on the closed convex set $G^\varepsilon$ is unique.
To find the projection of $\transpose{[u,v_1,v_2,w]}$, we need to solve the following minimization problem: find $\transpose{[\rho, m_1, m_2, E]}$ that
\begin{align}
\min_{\rho,m_1,m_2,E}&~~ \frac{1}{2}\big(\abs{\rho - u}^2 + \abs{m_1 - v_1}^2 + \abs{m_2 - v_2}^2 + \abs{E - w}^2\big) \nonumber\\
\text{subject~to:}&~~ \varepsilon - \rho \leq 0 
\quad\text{and}\quad \varepsilon - E + \frac{m_1^2 + m_2^2}{2\rho} \leq 0.
\end{align}
We derive an exact solution using the following KKT conditions:
\begin{subequations}\label{eq:2d_KKT}
\begin{itemize}
\item Stationarity condition
\begin{align}
\rho - u - \lambda - \mu\frac{m_1^2 + m_2^2}{2\rho^2} = 0,\label{eq:1a}\\
m_1 - v_1 + \mu \frac{m_1}{\rho} = 0,\label{eq:1b}\\
m_2 - v_2 + \mu \frac{m_2}{\rho} = 0,\label{eq:1b2}\\
E - w - \mu = 0.\label{eq:1c}
\end{align}
\item Complementary slackness
\begin{align}
\lambda(\varepsilon - \rho) = 0,\label{eq:1d}\\
\mu\Big(\varepsilon - E + \frac{m_1^2 + m_2^2}{2\rho}\Big) = 0.\label{eq:1e}
\end{align}
\item Primal feasibility
\begin{align}
\varepsilon - \rho \leq 0, \label{eq:f}\\
\varepsilon - E + \frac{m_1^2 + m_2^2}{2\rho} \leq 0.\label{eq:1g}
\end{align}
\item Dual feasibility
\begin{align}
\lambda \geq 0,\label{eq:1h}\\
\mu \geq 0.\label{eq:1i}
\end{align}
\end{itemize}
\end{subequations}
From the dual feasibility, depending on $\lambda$ and $\mu$ are zero or positive, we have four different scenarios. Let us discuss case by case.

\paragraph{\bf Case~1 ($\mu = 0$ and $\lambda > 0$)}
When $\mu$ is zero, the \eqref{eq:1a}-\eqref{eq:1c} gives $\rho = u + \lambda$, $m_1 = v_1$, $m_2 = v_2$, and $E = w$. When $\lambda > 0$, the \eqref{eq:1d} implies $\rho = \varepsilon$. From \eqref{eq:1a}, we get $\lambda = \varepsilon - u$. 
Thus, if $\transpose{[\varepsilon, v_1, v_2, w]}$ is a solution of \eqref{eq:2d_KKT}, the constraints $\varepsilon > u$ and \eqref{eq:1g} need to be satisfied.

\paragraph{\bf Case~2 ($\mu = 0$ and $\lambda = 0$)}
When both $\mu$ and $\lambda$ are zero, the \eqref{eq:1a}-\eqref{eq:1c} gives $\rho = u$, $m_1 = v_1$, $m_2 = v_2$, and $E = w$. 
Thus, if $\transpose{[u, v_1, v_2, w]}$ is a solution of \eqref{eq:2d_KKT}, the given point already belongs to the set $G^\varepsilon$, i.e., the projection point is itself.

\paragraph{\bf Case~3 ($\mu > 0$ and $\lambda > 0$)}
When both $\mu$ and $\lambda$ are positive, the \eqref{eq:1d} gives $\rho = \varepsilon $ and \eqref{eq:1e} becomes
\begin{align}\label{eq:2d_case3_E}
E = \frac{1}{2\varepsilon}(m_1^2 + m_2^2) + \varepsilon.
\end{align}
If $v_1 = 0$ and $v_2 = 0$, by \eqref{eq:1b} and \eqref{eq:1b2}, we have $m_1 = 0$ and $m_2 = 0$, which further implies $E = \varepsilon$. From \eqref{eq:1c}, we get $\mu = \varepsilon - w$ and from \eqref{eq:1a}, we get $\lambda = \varepsilon - u$. 
Thus, if $\transpose{[\varepsilon, 0, 0, \varepsilon]}$ is a solution of \eqref{eq:2d_KKT}, the constraints $\varepsilon > u$ and $\varepsilon > w$ need to be satisfied.
\par
If $v_1 = 0$ and $v_2 \neq 0$, then \eqref{eq:1b} gives $m_1 = 0$ and \eqref{eq:1b2} gives $m_2 \neq 0$. Therefore, \eqref{eq:1a}-\eqref{eq:1c} become
\begin{subequations}
\begin{align}
\rho - u - \lambda - \mu\frac{m_2^2}{2\rho^2} = 0,\\
m_2 - v_2 + \mu \frac{m_2}{\rho} = 0,\\
E - w - \mu = 0.
\end{align}
\end{subequations}
Furthermore, the KKT condition \eqref{eq:2d_KKT} degenerates to the associated one-dimensional case, which can be solved by using the method in Section~\ref{sec:1D_proj} Case~3.
\par
If $v_1 \neq 0$, multiply \eqref{eq:1b} by $v_2$ and multiply \eqref{eq:1b2} by $v_1$, respectively. We have
\begin{align}\label{eq:new_old_momentum}
\begin{rcases*}
\Big(1+\dfrac{\mu}{\rho}\Big) m_1 v_2 = v_1 v_2\\
\Big(1+\dfrac{\mu}{\rho}\Big) m_2 v_1 = v_1 v_2
\end{rcases*}
\quad\Rightarrow\quad
m_1 v_2 = m_2 v_1
\quad\Rightarrow\quad
m_2 = \frac{v_2}{v_1}m_1.
\end{align}
When $v_1 \neq 0$, it implies $m_1 \neq 0$. Otherwise, \eqref{eq:1b} gives $v_1 = 0$, which leads to a contradiction. Let $a = 1 + v_2^2/v_1^2$. Substituting $\rho = \varepsilon$ into \eqref{eq:1b} to solve $\mu$, and further substituting the result into \eqref{eq:1a}, we have
\begin{align}\label{eq:2d_case3_mu_lamb}
\mu = \varepsilon\Big(\frac{v_1}{m_1} - 1\Big)
\quad\text{and}\quad
\lambda = \varepsilon - u - \frac{a}{2\varepsilon}m_1(v_1 - m_1).
\end{align}
By substituting \eqref{eq:new_old_momentum} into \eqref{eq:2d_case3_E} and combining the result with \eqref{eq:1c} and \eqref{eq:2d_case3_mu_lamb}, we obtain
\begin{align}\label{eq:2d_case3_E2}
\frac{a}{2\varepsilon}m_1^2 + \varepsilon = E = \mu + w = \varepsilon\Big(\frac{v_1}{m_1} - 1\Big) + w.
\end{align}
Multiplying $2\varepsilon m_1$ on both sides above, we get a cubic equation with respect to $m_1$ as follows.
\begin{align*}
a m_1^3 + (4\varepsilon^2 - 2\varepsilon w)m_1 - 2\varepsilon^2 v_1 = 0.
\end{align*}
After getting all the real roots of using the method in \ref{appendix:cubic_root}, we compute $m_2$ form \eqref{eq:new_old_momentum}, $E$ from \eqref{eq:2d_case3_E2}, and $\mu, \lambda$ from \eqref{eq:2d_case3_mu_lamb}. 
Thus, if the result here is a solution of \eqref{eq:2d_KKT}, the constraints $\mu > 0$ and $\lambda > 0$ for \eqref{eq:2d_case3_mu_lamb} need to be satisfied.

\paragraph{\bf Case~4 ($\mu > 0$ and $\lambda = 0$)}
When $\mu$ is positive and $\lambda$ is zero, the KKT condition \eqref{eq:2d_KKT} simplify to $\rho \geq \varepsilon$ and
\begin{subequations}\label{eq:case4_2d_kkt}
\begin{align}
\rho - u - \mu\frac{m_1^2 + m_2^2}{2\rho^2} &= 0,\label{eq:2a}\\
m_1 - v_1 + \mu \frac{m_1}{\rho} &= 0,\label{eq:2b}\\
m_2 - v_2 + \mu \frac{m_2}{\rho} &= 0,\label{eq:2b1}\\
E - w - \mu &= 0,\label{eq:2c}\\
\varepsilon - E + \frac{m_1^2 + m_2^2}{2\rho} &= 0.\label{eq:2d}
\end{align}
\end{subequations}
If $v_1 = 0$ and $v_2 = 0$, we have $m_1 = 0$ and $m_2 = 0$. By \eqref{eq:2a} and \eqref{eq:2d}, we get $\rho = u$ and $E = \varepsilon$. From \eqref{eq:2c}, we get $\mu = \varepsilon - w$.
Thus, if $\transpose{[u, 0, 0, \varepsilon]}$ is a solution of \eqref{eq:2d_KKT}, the constraints $\varepsilon > w$ and $\varepsilon \leq u$ need to be satisfied.
\par
If $v_1 = 0$ and $v_2 \neq 0$, then \eqref{eq:2b} gives $m_1 = 0$ and \eqref{eq:2b1} gives $m_2 \neq 0$. Thus, \eqref{eq:2a}-\eqref{eq:2d} become
\begin{subequations}
\begin{align}
\rho - u - \mu\frac{m_2^2}{2\rho^2} &= 0,\\
m_2 - v_2 + \mu \frac{m_2}{\rho} &= 0,\\
E - w - \mu &= 0,\label{eq:2c_2}\\
\varepsilon - E + \frac{m_2^2}{2\rho} &= 0,
\end{align}
\end{subequations}
which can be solved by using the method in Section~\ref{sec:1D_proj} Case~4.
\par
If $v_1 \neq 0$, recall $a = 1 + v_2^2/v_1^2$ and $m_2v_1 = m_1v_2$, by eliminating $m_2$ in \eqref{eq:2a} and substituting \eqref{eq:2b} into \eqref{eq:2a}, we have
\begin{align*}
\rho - u - a\frac{m_1}{2\rho} \Big(\mu\frac{m_1}{\rho}\Big) = 0 
\quad\Rightarrow\quad
a\frac{m_1}{\rho} (v_1 - m_1) = 2(\rho - u),
\end{align*}
which is a quadratic equation with respect to $m_1$, namely $am_1^2 - av_1m_1 + 2\rho^2 - 2u\rho = 0$.
Thus, let $\Delta = -8a\rho^2 + 8au\rho + a^2v_1^2$ denote the discriminant. We have
\begin{align}\label{eq:case4_2d_m1}
m_1 = \frac{v_1}{2} - \frac{1}{2a}\sqrt{\Delta}
\quad\text{or}\quad
m_1 = \frac{v_1}{2} + \frac{1}{2a}\sqrt{\Delta}.
\end{align}
Similarly, using $m_2 = m_1 v_2/v_1$ to eliminate $m_2$ in \eqref{eq:2d} and adding \eqref{eq:2c_2} and \eqref{eq:2d} together, we have 
\begin{align}
E &= \varepsilon + a\frac{m_1^2}{2\rho}, \\
\mu &= \varepsilon + a\frac{m_1^2}{2\rho} - w. \label{eq:3b}
\end{align}
Therefore, once we compute $\rho$, the \eqref{eq:case4_2d_kkt} is solved.
Substitute \eqref{eq:3b} into \eqref{eq:2b}. After multiplying $2\rho^2$, we end up with a quadratic equation with respect to $\rho$, as follows. Notice that $v_1\neq0$ implies $m_1 \neq v_1$. 
\begin{align}\label{eq:4}
2(m_1 - v_1)\rho^2 + 2(\varepsilon - w)m_1\rho + am_1^3 = 0.
\end{align}
To solve $\rho$, let us first substitute $m_1 = \frac{v_1}{2} - \frac{1}{2a}\sqrt{\Delta}$ into above. The other root in \eqref{eq:case4_2d_m1} can be addressed in the same manner. We get 
\begin{align}\label{eq:5}
\sqrt{\Delta}\Big(3v_1^2 + \frac{\Delta}{a^2} + 8(\varepsilon-w)\frac{\rho}{a} + \frac{8}{a}\rho^2\Big) = av_1^3 + \frac{3}{a}v_1\Delta + 8(\varepsilon - w)v_1\rho - 8v_1\rho^2.
\end{align}
Recall the discriminant $\Delta = -8a\rho^2 + 8au\rho + a^2v_1^2$. Square both sides above and after some simplification, we obtain
\begin{align}\label{eq:case4_2d_equation_rho}
4\Big(2v_1^2 + \frac{1}{a}(\varepsilon + u - w)^2\Big)\rho^2 
- 4u\Big(2v_1^2 + \frac{1}{a}(\varepsilon + u - w)^2\Big)\rho 
+ 2uv_1^2(w-\varepsilon) - av_1^4 = 0.
\end{align}
The two roots of above quadratic equation are:
\begin{subequations}\label{eq:case4_2d_rho}
\begin{align}
\rho &= \frac{u}{2} + \frac{1}{2} \frac{\sqrt{u^2(2v_1^2 + \frac{1}{a}(\varepsilon + u - w)^2) - 2uv_1^2(w-\varepsilon) + av_1^4}}{\sqrt{2v_1^2 + \frac{1}{a}(\varepsilon + u - w)^2}},\\
\quad\text{or}\quad
\rho &= \frac{u}{2} - \frac{1}{2} \frac{\sqrt{u^2(2v_1^2 + \frac{1}{a}(\varepsilon + u - w)^2) - 2uv_1^2(w-\varepsilon) + av_1^4}}{\sqrt{2v_1^2 + \frac{1}{a}(\varepsilon + u - w)^2}}.
\end{align}
\end{subequations}
For the second root in \eqref{eq:case4_2d_m1}, substitute $m_1 = \frac{v_1}{2} + \frac{1}{2a}\sqrt{\Delta}$ into \eqref{eq:4}. After simplification, we end up with exactly the same equation as \eqref{eq:case4_2d_equation_rho}, namely the expressions in \eqref{eq:case4_2d_rho} are still the roots.
Thus, if the result here is a solution of \eqref{eq:2d_KKT}, the constraints $\mu > 0$ for \eqref{eq:3b} and $\rho \geq \varepsilon$ for \eqref{eq:case4_2d_rho} need to be satisfied.
\par 

Finally, among all the points listed above, we pick the point $\transpose{[\rho, m_1, m_2, E]}$  which gives the smallest value of the objective function, and it is the projection point.

\subsection{Algorithm of projection onto the invariant domain set}\label{appendix:2D_projection}
Given $\transpose{[x, y_1, y_2, z]}\notin G^\varepsilon$, find the projection $\transpose{[\rho,m_1,m_2,E]}$ on set $G^\varepsilon$. 
Without loss of generality, let us assume $\abs{y_1} \geq \abs{y_2}$. A summary of the resulting algorithm is listed below. \\
Step~1. Compute candidate points from KKT conditions.
\begin{itemize}
%% Case 1.
\item The point $\transpose{[\rho,m_1,m_2,E]} = \transpose{[\varepsilon, y_1, y_2, z]}$ is a candidate, if $x < \varepsilon$ and $z - \frac{1}{2\varepsilon}(y_1^2 + y_2^2) \geq \varepsilon$.
%% Case 3.
\item The point $\transpose{[\rho,m_1,m_2,E]} = \transpose{[\varepsilon, 0, 0, \varepsilon]}$ is a candidate, if $x < \varepsilon$, $y_1 = 0$, $y_2 = 0$, and $z < \varepsilon$.
\item If $y_1 \neq0$, solve the cubic equation $am_1^3 + (4\varepsilon^2 - 2\varepsilon z)m_1 - 2\varepsilon^2 y_1 = 0$ with $a = 1 + y_2^2/y_1^2$ to obtain all real roots. Examine each real root individually. 
Let $m_{1r}$ denote a real root. Then the point 
\begin{align*}
\transpose{[\rho,m_1,m_2,E]} = \transpose{[\varepsilon,m_{1r},m_{2r},E_r]}, 
~~\text{where}~~
m_{2r} = \frac{y_2}{y_1}m_{1r} ~~\text{and}~~ 
E_r = \frac{a}{2\varepsilon}m_{1r}^2 + \varepsilon,
\end{align*}
is a candidate, if $\frac{y_1}{m_{1r}} > 1$ and $2\varepsilon x + a m_{1r}(y_1 - m_{1r}) < 2\varepsilon^2$.
%% Case 4.
\item The point $\transpose{[\rho,m_1,m_2,E]} = \transpose{[x, 0, 0, \varepsilon]}$ is a candidate, if $x\geq\varepsilon$, $y_1 = 0$, $y_2 = 0$, and $z < \varepsilon$.
\item Compute $\rho_1$ and $\rho_2$ using the following formulas only if they are real.    
\begin{align*}
\rho_1 &= \frac{x}{2} + \frac{1}{2} \frac{\sqrt{x^2(2y_1^2 + \frac{1}{a}(\varepsilon + x - z)^2) - 2xy_1^2(z-\varepsilon) + ay_1^4}}{\sqrt{2y_1^2 + \frac{1}{a}(\varepsilon + x - z)^2}},\\
\rho_2 &= \frac{x}{2} - \frac{1}{2} \frac{\sqrt{x^2(2y_1^2 + \frac{1}{a}(\varepsilon + x - z)^2) - 2xy_1^2(z-\varepsilon) + ay_1^4}}{\sqrt{2y_1^2 + \frac{1}{a}(\varepsilon + x - z)^2}}.
\end{align*}
For real values of $\rho_1$ and $\rho_2$, compute $m_{1\alpha}$ and $m_{1\beta}$ using the following formulas only if they are real.
\begin{align*}
m_{1\alpha}(\rho) = \frac{1}{2}y_1 - \frac{1}{2a}\sqrt{-8a\rho^2 + 8ax\rho + a^2 y_1^2}
\quad\text{and}\quad
m_{1\beta}(\rho) = \frac{1}{2}y_1 + \frac{1}{2a}\sqrt{-8a\rho^2 + 8ax\rho + a^2 y_1^2}.
\end{align*}
Then, for real points:
\begin{itemize}
\item The point $\transpose{[\rho,m_1,m_2,E]} = \transpose{[\rho_1, m_{1\alpha}(\rho_1), \frac{y_2}{y_1}m_{1\alpha}(\rho_1), \varepsilon + a\frac{m_{1\alpha}(\rho_1)^2}{2\rho_1}]}$ is a candidate, if $\rho_1\geq\varepsilon$ and $\varepsilon + a\frac{m_{1\alpha}(\rho_1)^2}{2\rho_1} > z$.
\item The point $\transpose{[\rho,m_1,m_2,E]} = \transpose{[\rho_2, m_{1\alpha}(\rho_2), \frac{y_2}{y_1}m_{1\alpha}(\rho_2), \varepsilon + a\frac{m_{1\alpha}(\rho_2)^2}{2\rho_2}]}$ is a candidate, if $\rho_2\geq\varepsilon$ and $\varepsilon + a\frac{m_{1\alpha}(\rho_2)^2}{2\rho_2} > z$.
\item The point $\transpose{[\rho,m_1,m_2,E]} = \transpose{[\rho_1, m_{1\beta}(\rho_1), \frac{y_2}{y_1}m_{1\beta}(\rho_1), \varepsilon + a\frac{m_{1\beta}(\rho_1)^2}{2\rho_1}]}$ is a candidate, if $\rho_1\geq\varepsilon$ and $\varepsilon + a\frac{m_{1\beta}(\rho_1)^2}{2\rho_1} > z$.
\item The point $\transpose{[\rho,m_1,m_2,E]} = \transpose{[\rho_2, m_{1\beta}(\rho_2), \frac{y_2}{y_1}m_{1\beta}(\rho_2), \varepsilon + a\frac{m_{1\beta}(\rho_2)^2}{2\rho_2}]}$ is a candidate, if $\rho_2\geq\varepsilon$ and $\varepsilon + a\frac{m_{1\beta}(\rho_2)^2}{2\rho_2} > z$.
\end{itemize}
\end{itemize}
Step~2. Substitute all candidate points into the objective function
\begin{align*}
\abs{\rho - x}^2 + \abs{m_1 - y_1}^2 + \abs{m_2 - y_2}^2 + \abs{E - z}^2.
\end{align*}
Pick the point $\transpose{[\rho, m_1, m_2, E]}$ from the candidate points such that minimizes the value of the objective function, which gives the projection point.
\begin{remark}
For numerical stability, the computation of the factor $a = 1 + y_2^2/y_1^2$ must avoid division by small denominator. Notice, the two momentum components play symmetric roles. If $\abs{y_2} > \abs{y_1}$, we instead work with the variable $m_2$. The derivation of corresponding formulas follows the same procedure and is omitted for brevity.
\end{remark}

%%%%%%%%%%%%%%%%%%%%%%%%%%%%%%%%%%%%%%%%%%%%%%%%%%%%%%%%%%%%%%%%%%%%%%%%%%%%%%%%%%%%%%%%%%%%%%%%%%%%%%%%%%%%%%%%%%%%%%%
\section{Projection to numerical admissible set in three dimensions}\label{sec:3D_proj}
%%%%%%%%%%%%%%%%%%%%%%%%%%%%%%%%%%%%%%%%%%%%%%%%%%%%%%%%%%%%%%%%%%%%%%%%%%%%%%%%%%%%%%%%%%%%%%%%%%%%%%%%%%%%%%%%%%%%%%%
Let $\transpose{[u, v_1, v_2, v_3, w]}$ denote an out-of-bound cell average produced by a simulation in three-dimensional domain.
For a small prescribed $\varepsilon>0$, the numerical admissible set $G^\varepsilon$ in \eqref{invariant-domain} becomes 
\begin{align}
G^\varepsilon = \left\{\vec{U} = \transpose{[\rho, m_1, m_2, m_3, E]}\!:~ \rho \geq \varepsilon,~~ \rho e(\vec{U}) = E - \frac{m_1^2 + m_2^2 + m_3^2}{2\rho} \geq \varepsilon\right\}.
\end{align} 
The projection point on the closed convex set $G^\varepsilon$ is unique.
To find the projection of $\transpose{[u,v_1,v_2,v_3,w]}$, we need to solve the following minimization problem: find $\transpose{[\rho, m_1, m_2, m_3, E]}$ that
\begin{align}
\min_{\rho,m_1,m_2,m_3,E}&~~ \frac{1}{2}\big(\abs{\rho - u}^2 + \abs{m_1 - v_1}^2 + \abs{m_2 - v_2}^2 + \abs{m_3 - v_3}^2 + \abs{E - w}^2\big) \nonumber\\
\text{subject~to:}&~~ \varepsilon - \rho \leq 0 
\quad\text{and}\quad \varepsilon - E + \frac{m_1^2 + m_2^2 + m_3^2}{2\rho} \leq 0.
\end{align}
By KKT conditions, we have
\begin{subequations}\label{eq:3d_KKT}
\begin{itemize}
\item Stationarity condition
\begin{align}
\rho - u - \lambda - \mu\frac{m_1^2 + m_2^2 + m_3^2}{2\rho^2} = 0,\label{eq:1a_3D}\\
m_1 - v_1 + \mu \frac{m_1}{\rho} = 0,\label{eq:1b_3D}\\
m_2 - v_2 + \mu \frac{m_2}{\rho} = 0,\label{eq:1b2_3D}\\
m_3 - v_3 + \mu \frac{m_3}{\rho} = 0,\label{eq:1b3_3D}\\
E - w - \mu = 0.\label{eq:1c_3D}
\end{align}
\item Complementary slackness
\begin{align}
\lambda(\varepsilon - \rho) = 0,\label{eq:1d_3D}\\
\mu\Big(\varepsilon - E + \frac{m_1^2 + m_2^2 + m_3^2}{2\rho}\Big) = 0.\label{eq:1e_3D}
\end{align}
\item Primal feasibility
\begin{align}
\varepsilon - \rho \leq 0, \label{eq:f_3D}\\
\varepsilon - E + \frac{m_1^2 + m_2^2 + m_3^2}{2\rho} \leq 0.\label{eq:1g_3D}
\end{align}
\item Dual feasibility
\begin{align}
\lambda \geq 0,\label{eq:1h_3D}\\
\mu \geq 0.\label{eq:1i_3D}
\end{align}
\end{itemize}
\end{subequations}
From the dual feasibility, depending on $\lambda$ and $\mu$ are zero or positive, we have four different scenarios. Let us discuss case by case.

\paragraph{\bf Case~1 ($\mu = 0$ and $\lambda > 0$)}
When $\mu$ is zero, the \eqref{eq:1a_3D}-\eqref{eq:1c_3D} gives $\rho = u + \lambda$, $m_1 = v_1$, $m_2 = v_2$, $m_3 = v_3$, and $E = w$. When $\lambda > 0$, the \eqref{eq:1d_3D} implies $\rho = \varepsilon$. From \eqref{eq:1a_3D}, we get $\lambda = \varepsilon - u$. 
Thus, if $\transpose{[\varepsilon, v_1, v_2, v_3, w]}$ is a solution of \eqref{eq:3d_KKT}, the constraints $\varepsilon > u$ and \eqref{eq:1g_3D} need to be satisfied.

\paragraph{\bf Case~2 ($\mu = 0$ and $\lambda = 0$)}
When both $\mu$ and $\lambda$ are zero, the \eqref{eq:1a_3D}-\eqref{eq:1c_3D} gives $\rho = u$, $m_1 = v_1$, $m_2 = v_2$, $m_3 = v_3$, and $E = w$. 
Thus, if $\transpose{[u, v_1, v_2, v_3, w]}$ is a solution of \eqref{eq:3d_KKT}, the given point already belongs to the set $G^\varepsilon$, i.e., the projection point is itself.

\paragraph{\bf Case~3 ($\mu > 0$ and $\lambda > 0$)}
When both $\mu$ and $\lambda$ are positive, the \eqref{eq:1d_3D} gives $\rho = \varepsilon$ and \eqref{eq:1e_3D} becomes
\begin{align}\label{eq:3d_case3_E}
E = \frac{1}{2\varepsilon}(m_1^2 + m_2^2 + m_3^2) + \varepsilon.
\end{align}
If $v_1 = 0$, $v_2 = 0$, and $v_3 = 0$, by \eqref{eq:1b_3D}-\eqref{eq:1b3_3D}, we have $m_1 = 0$, $m_2 = 0$, and $m_3 = 0$, which further implies $E = \varepsilon$. From \eqref{eq:1c_3D}, we get $\mu = \varepsilon - w$ and from \eqref{eq:1a_3D}, we get $\lambda = \varepsilon - u$. 
Thus, if $\transpose{[\varepsilon, 0, 0, 0, \varepsilon]}$ is a solution of \eqref{eq:3d_KKT}, the constraints $\varepsilon > u$ and $\varepsilon > w$ need to be satisfied.
\par
If $v_3 = 0$ and $v_1$, $v_2$ are not both zero, this scenario degenerate to the 2D projection, which has been solved in \ref{sec:2D_proj}.
\par
If $v_3 \neq 0$, then $m_3 \neq 0$. Otherwise, \eqref{eq:1b3_3D} gives $v_3 = 0$, which leads to a contradiction. Multiply \eqref{eq:1b_3D} by $v_3$ and multiply \eqref{eq:1b3_3D} by $v_1$, respectively. We have
\begin{align}\label{eq:new_old_momentum_3D}
\begin{rcases*}
\Big(1+\dfrac{\mu}{\rho}\Big) m_1 v_3 = v_1 v_3\\
\Big(1+\dfrac{\mu}{\rho}\Big) m_3 v_1 = v_1 v_3
\end{rcases*}
\quad\Rightarrow\quad
m_1 v_3 = m_3 v_1
\quad\Rightarrow\quad
m_1 = \frac{v_1}{v_3}m_3.
\end{align}
Similarly, multiply \eqref{eq:1b2_3D} by $v_3$ and multiply \eqref{eq:1b3_3D} by $v_2$, respectively. We have
\begin{align}\label{eq:new_old_momentum2_3D}
\begin{rcases*}
\Big(1+\dfrac{\mu}{\rho}\Big) m_2 v_3 = v_2 v_3\\
\Big(1+\dfrac{\mu}{\rho}\Big) m_3 v_2 = v_2 v_3
\end{rcases*}
\quad\Rightarrow\quad
m_2 v_3 = m_3 v_2
\quad\Rightarrow\quad
m_2 = \frac{v_2}{v_3}m_3.
\end{align}
Let $a = 1 + v_1^2/v_3^2 + v_2^2/v_3^2$. Substituting $\rho = \varepsilon$ into \eqref{eq:1b3_3D} to solve $\mu$, and further substituting the result into \eqref{eq:1a_3D}, we have
\begin{align}\label{eq:3d_case3_mu_lamb}
\mu = \varepsilon\Big(\frac{v_3}{m_3} - 1\Big)
\quad\text{and}\quad
\lambda = \varepsilon - u - \frac{a}{2\varepsilon}m_3(v_3 - m_3).
\end{align}
By substituting \eqref{eq:new_old_momentum_3D} and \eqref{eq:new_old_momentum2_3D} into \eqref{eq:3d_case3_E} and combining the result with \eqref{eq:1c_3D} and \eqref{eq:3d_case3_mu_lamb}, we obtain
\begin{align}\label{eq:3d_case3_E2}
\frac{a}{2\varepsilon}m_3^2 + \varepsilon = E = \mu + w = \varepsilon\Big(\frac{v_3}{m_3} - 1\Big) + w.
\end{align}
Multiplying $2\varepsilon m_3$ on both sides above, we get a cubic equation with respect to $m_3$ as follows.
\begin{align*}
a m_3^3 + (4\varepsilon^2 - 2\varepsilon w)m_3 - 2\varepsilon^2 v_3 = 0.
\end{align*}
After getting all the real roots of the cubic equation above, we compute $m_1$ form \eqref{eq:new_old_momentum_3D}, $m_2$ form \eqref{eq:new_old_momentum2_3D}, $E$ from \eqref{eq:3d_case3_E2}, and $\mu, \lambda$ from \eqref{eq:3d_case3_mu_lamb}. 
Thus, if the result here is a solution of \eqref{eq:3d_KKT}, the constraints $\mu > 0$ and $\lambda > 0$ for \eqref{eq:3d_case3_mu_lamb} need to be satisfied.

\paragraph{\bf Case~4 ($\mu > 0$ and $\lambda = 0$)}
When $\mu$ is positive and $\lambda$ is zero, the KKT condition \eqref{eq:3d_KKT} simplify to $\rho \geq \varepsilon$ and
\begin{subequations}\label{eq:case4_3d_kkt}
\begin{align}
\rho - u - \mu\frac{m_1^2 + m_2^2 + m_3^2}{2\rho^2} &= 0,\label{eq:2a_3D}\\
m_1 - v_1 + \mu \frac{m_1}{\rho} &= 0,\label{eq:2b_3D}\\
m_2 - v_2 + \mu \frac{m_2}{\rho} &= 0,\label{eq:2b1_3D}\\
m_3 - v_3 + \mu \frac{m_3}{\rho} &= 0,\label{eq:2b2_3D}\\
E - w - \mu &= 0,\label{eq:2c_3D}\\
\varepsilon - E + \frac{m_1^2 + m_2^2 + m_3^2}{2\rho} &= 0.\label{eq:2d_3D}
\end{align}
\end{subequations}
If $v_1 = 0$, $v_2 = 0$, and $v_3 = 0$, we have $m_1 = 0$, $m_2 = 0$, and $m_3 = 0$. By \eqref{eq:2a_3D} and \eqref{eq:2d_3D}, we get $\rho = u$ and $E = \varepsilon$. From \eqref{eq:2c_3D}, we get $\mu = \varepsilon - w$.
Thus, if $\transpose{[u, 0, 0, 0, \varepsilon]}$ is a solution of \eqref{eq:3d_KKT}, the constraints $\varepsilon > w$ and $\varepsilon \leq u$ need to be satisfied.
\par
If $v_3 = 0$ and $v_1$, $v_2$ are not both zero, this scenario degenerate to the 2D projection, which has been solved in \ref{sec:2D_proj}.
\par
If $v_3 \neq 0$, then $m_3 \neq 0$. Otherwise, \eqref{eq:2b2_3D} gives $v_3 = 0$, which leads to a contradiction. Recall $a = 1 + v_1^2/v_3^2 + v_2^2/v_3^2$, $m_1v_3 = m_3v_1$ and $m_2v_3 = m_3v_2$, by eliminating $m_1$ and $m_2$ in \eqref{eq:2a_3D} and substituting \eqref{eq:2b2_3D} into \eqref{eq:2a_3D}, we have
\begin{align*}
\rho - u - a\frac{m_3}{2\rho} \Big(\mu\frac{m_3}{\rho}\Big) = 0 
\quad\Rightarrow\quad
a\frac{m_3}{\rho} (v_3 - m_3) = 2(\rho - u),
\end{align*}
which is a quadratic equation with respect to $m_3$, namely $am_3^2 - av_3m_3 + 2\rho^2 - 2u\rho = 0$.
Thus, let $\Delta = -8a\rho^2 + 8au\rho + a^2v_3^2$ denote the discriminant. We have
\begin{align}\label{eq:case4_3d_m1}
m_3 = \frac{v_3}{2} - \frac{1}{2a}\sqrt{\Delta}
\quad\text{or}\quad
m_3 = \frac{v_3}{2} + \frac{1}{2a}\sqrt{\Delta}.
\end{align}
Similarly, using $m_1 = m_3 v_1/v_3$ and $m_2 = m_3 v_2/v_3$ to eliminate $m_1$ and $m_2$ in \eqref{eq:2d_3D} and adding \eqref{eq:2c_3D} and \eqref{eq:2d_3D} together, we have 
\begin{align}
E &= \varepsilon + a\frac{m_3^2}{2\rho}, \\
\mu &= \varepsilon + a\frac{m_3^2}{2\rho} - w. \label{eq:3b_3D}
\end{align}
Therefore, once we compute $\rho$, the \eqref{eq:case4_3d_kkt} is solved.
Substitute \eqref{eq:3b_3D} into \eqref{eq:2b2_3D}. After multiplying $2\rho^2$, we end up with a quadratic equation with respect to $\rho$, as follows. Notice that $v_3\neq0$ implies $m_3 \neq v_3$. 
\begin{align}\label{eq:4_3D}
2(m_3 - v_3)\rho^2 + 2(\varepsilon - w)m_3\rho + am_3^3 = 0.
\end{align}
To solve $\rho$, let us first substitute $m_3 = \frac{v_3}{2} - \frac{1}{2a}\sqrt{\Delta}$ into above. The other root in \eqref{eq:case4_3d_m1} can be addressed in the same manner. We get  
\begin{align}
\sqrt{\Delta}\Big(3v_3^2 + \frac{\Delta}{a^2} + 8(\varepsilon-w)\frac{\rho}{a} + \frac{8}{a}\rho^2\Big) = av_3^3 + \frac{3}{a}v_3\Delta + 8(\varepsilon-w)v_3\rho - 8v_3\rho^2.
\end{align}
Recall the discriminant $\Delta = -8a\rho^2 + 8au\rho + a^2v_3^2$. Square both sides above and after some simplification, we obtain 
\begin{align}\label{eq:case4_3d_equation_rho}
4\Big(2v_3^2 + \frac{1}{a}(\varepsilon + u - w)^2\Big)\rho^2 
- 4u\Big(2v_3^2 + \frac{1}{a}(\varepsilon + u - w)^2\Big)\rho 
+ 2uv_3^2(w-\varepsilon) - av_3^4 = 0.
\end{align}
The two roots of above quadratic equation are:
\begin{subequations}\label{eq:case4_3d_rho}
\begin{align}
\rho &= \frac{u}{2} + \frac{1}{2} \frac{\sqrt{u^2(2v_3^2 + \frac{1}{a}(\varepsilon + u - w)^2) - 2uv_3^2(w-\varepsilon) + av_3^4}}{\sqrt{2v_3^2 + \frac{1}{a}(\varepsilon + u - w)^2}},\\
\quad\text{or}\quad
\rho &= \frac{u}{2} - \frac{1}{2} \frac{\sqrt{u^2(2v_3^2 + \frac{1}{a}(\varepsilon + u - w)^2) - 2uv_3^2(w-\varepsilon) + av_3^4}}{\sqrt{2v_3^2 + \frac{1}{a}(\varepsilon + u - w)^2}}.
\end{align}
\end{subequations}
For the second root in \eqref{eq:case4_3d_m1}, substitute $m_3 = \frac{v_3}{2} + \frac{1}{2a}\sqrt{\Delta}$ into \eqref{eq:4_3D}. After simplification, we end up with exactly the same equation as \eqref{eq:case4_3d_equation_rho}, namely the expressions in \eqref{eq:case4_3d_rho} are still the roots.
Thus, if the result here is a solution of \eqref{eq:3d_KKT}, the constraints $\mu > 0$ for \eqref{eq:3b_3D} and $\rho \geq \varepsilon$ for \eqref{eq:case4_3d_rho} need to be satisfied.
\par 
Finally, among all the points listed above, we pick the point $\transpose{[\rho, m_1, m_2, m_3, E]}$  which gives the smallest value of the objective function, and it is the projection point.
The algorithm for projection onto the invariant domain set can be summarized analogously to the previous. For the sake of conciseness, we omit the detailed description.

%%%%%%%%%%%%%%%%%%%%%%%%%%%%%%%%%%%%%%%%%%%%%%%%%%%%%%%%%%%%%%%%%%%%%%%%%%%%%%%%%%%%%%%%%%%%%%%%%%%%%%%%%%%%%%%%%%%%%%%
\section{Real roots of depressed cubic equation}\label{appendix:cubic_root}
%%%%%%%%%%%%%%%%%%%%%%%%%%%%%%%%%%%%%%%%%%%%%%%%%%%%%%%%%%%%%%%%%%%%%%%%%%%%%%%%%%%%%%%%%%%%%%%%%%%%%%%%%%%%%%%%%%%%%%%
We list the following  formulae to compute all the real roots of a depressed cubic equation in the form of $x^3 + px + q = 0$, where $p, q\in\IR$. For more details, see \cite{fan1989new} and \cite[Appendix~A]{zhao2019new}.
\begin{itemize}
\item[1.] If $p = q = 0$, then the equation has a triple real root.
\begin{align}
x_1 = x_2 = x_3 = 0.
\end{align}
\item[2.] If $4p^3 + 27q^2 > 0$, then the equation only has one real root.
\begin{align}
x_1 = -\frac{1}{3}(\sqrt[3]{Y_1} + \sqrt[3]{Y_2}).
\end{align}
Here $Y_1 = \frac{3}{2}(9q + \sqrt{12p^3 + 81q^2})$ and $Y_2 = \frac{3}{2}(9q - \sqrt{12p^3 + 81q^2})$.
\item[3.] If $4p^3 + 27q^2 = 0$ (but $p\neq0$ or $q\neq0$), then the equation has three real roots, one of which is a double root.
\begin{subequations}
\begin{align}
x_1 &= \frac{3q}{p},\\
x_2 &= x_3 = -\frac{3q}{2p}.
\end{align}
\end{subequations}
\item[4.] If $4p^3 + 27q^2 < 0$ (this implies $p<0$), then the equation has three unequal real roots.
\begin{subequations}
\begin{align}
x_1 &= -\frac{2\sqrt{-3p}}{3} \cos{\frac{\theta}{3}}, \\
x_2 &= \frac{\sqrt{-3p}}{3} (\cos{\frac{\theta}{3}} + \sqrt{3}\sin{\frac{\theta}{3}}), \\
x_3 &= \frac{\sqrt{-3p}}{3} (\cos{\frac{\theta}{3}} - \sqrt{3}\sin{\frac{\theta}{3}}).
\end{align}
\end{subequations}
Here, the $\theta = \arccos{\big(-\frac{3\sqrt{3}q}{2p\sqrt{-p}}\big)}$.
\end{itemize}
The advantage of using the  formulae above is that it allows us to only use real number operations to obtain real roots. Eliminating the need for complex number operations simplifies implementation.

%\section*{References}
\bibliographystyle{elsarticle-harv}
\bibliography{bibliography} % Load bibliography.bib.

\begin{thebibliography}{48}
\expandafter\ifx\csname natexlab\endcsname\relax\def\natexlab#1{#1}\fi
\providecommand{\url}[1]{\texttt{#1}}
\providecommand{\href}[2]{#2}
\providecommand{\path}[1]{#1}
\providecommand{\DOIprefix}{doi:}
\providecommand{\ArXivprefix}{arXiv:}
\providecommand{\URLprefix}{URL: }
\providecommand{\Pubmedprefix}{pmid:}
\providecommand{\doi}[1]{\href{http://dx.doi.org/#1}{\path{#1}}}
\providecommand{\Pubmed}[1]{\href{pmid:#1}{\path{#1}}}
\providecommand{\bibinfo}[2]{#2}
\ifx\xfnm\relax \def\xfnm[#1]{\unskip,\space#1}\fi
%Type = Article
\bibitem[{Anshika et~al.(2024)Anshika, Li, Ghosh and Zhang}]{anshika2024three}
\bibinfo{author}{Anshika, A.}, \bibinfo{author}{Li, J.},
  \bibinfo{author}{Ghosh, D.}, \bibinfo{author}{Zhang, X.},
  \bibinfo{year}{2024}.
\newblock \bibinfo{title}{A three-operator splitting scheme derived from
  three-block {ADMM}}.
\newblock \bibinfo{journal}{arXiv preprint arXiv:2411.00166} .
%Type = Article
\bibitem[{Bochev et~al.(2020)Bochev, Ridzal, D’Elia, Perego and
  Peterson}]{bochev2020optimization}
\bibinfo{author}{Bochev, P.}, \bibinfo{author}{Ridzal, D.},
  \bibinfo{author}{D’Elia, M.}, \bibinfo{author}{Perego, M.},
  \bibinfo{author}{Peterson, K.}, \bibinfo{year}{2020}.
\newblock \bibinfo{title}{Optimization-based, property-preserving finite
  element methods for scalar advection equations and their connection to
  algebraic flux correction}.
\newblock \bibinfo{journal}{Computer Methods in Applied Mechanics and
  Engineering} \bibinfo{volume}{367}, \bibinfo{pages}{112982}.
%Type = Incollection
\bibitem[{Bochev et~al.(2012)Bochev, Ridzal, Scovazzi and
  Shashkov}]{bochev2012constrained}
\bibinfo{author}{Bochev, P.}, \bibinfo{author}{Ridzal, D.},
  \bibinfo{author}{Scovazzi, G.}, \bibinfo{author}{Shashkov, M.},
  \bibinfo{year}{2012}.
\newblock \bibinfo{title}{Constrained-optimization based data transfer: {A} new
  perspective on flux correction}, in: \bibinfo{booktitle}{Flux-Corrected
  Transport: Principles, Algorithms, and Applications}.
  \bibinfo{publisher}{Springer}, pp. \bibinfo{pages}{345--398}.
%Type = Article
\bibitem[{Bochev et~al.(2013)Bochev, Ridzal and Shashkov}]{bochev2013fast}
\bibinfo{author}{Bochev, P.}, \bibinfo{author}{Ridzal, D.},
  \bibinfo{author}{Shashkov, M.}, \bibinfo{year}{2013}.
\newblock \bibinfo{title}{Fast optimization-based conservative remap of scalar
  fields through aggregate mass transfer}.
\newblock \bibinfo{journal}{Journal of Computational Physics}
  \bibinfo{volume}{246}, \bibinfo{pages}{37--57}.
%Type = Article
\bibitem[{Bradley et~al.(2019)Bradley, Bosler, Guba, Taylor and
  Barnett}]{bradley2019communication}
\bibinfo{author}{Bradley, A.M.}, \bibinfo{author}{Bosler, P.A.},
  \bibinfo{author}{Guba, O.}, \bibinfo{author}{Taylor, M.A.},
  \bibinfo{author}{Barnett, G.A.}, \bibinfo{year}{2019}.
\newblock \bibinfo{title}{Communication-efficient property preservation in
  tracer transport}.
\newblock \bibinfo{journal}{SIAM Journal on Scientific Computing}
  \bibinfo{volume}{41}, \bibinfo{pages}{C161--C193}.
%Type = Article
\bibitem[{Briceno-Arias(2015)}]{briceno2015forward}
\bibinfo{author}{Briceno-Arias, L.M.}, \bibinfo{year}{2015}.
\newblock \bibinfo{title}{Forward-{D}ouglas--{R}achford splitting and
  forward-partial inverse method for solving monotone inclusions}.
\newblock \bibinfo{journal}{Optimization} \bibinfo{volume}{64},
  \bibinfo{pages}{1239--1261}.
%Type = Article
\bibitem[{Chen et~al.(2025)Chen, Xiu and Zhang}]{chen2025enforcing}
\bibinfo{author}{Chen, Y.}, \bibinfo{author}{Xiu, D.}, \bibinfo{author}{Zhang,
  X.}, \bibinfo{year}{2025}.
\newblock \bibinfo{title}{On enforcing nonnegativity in polynomial
  approximations in high dimensions}.
\newblock \bibinfo{journal}{SIAM Journal on Scientific Computing}
  \bibinfo{volume}{47}, \bibinfo{pages}{A866--A888}.
%Type = Article
\bibitem[{Dai and Fletcher(2006)}]{dai2006new}
\bibinfo{author}{Dai, Y.H.}, \bibinfo{author}{Fletcher, R.},
  \bibinfo{year}{2006}.
\newblock \bibinfo{title}{New algorithms for singly linearly constrained
  quadratic programs subject to lower and upper bounds}.
\newblock \bibinfo{journal}{Mathematical Programming} \bibinfo{volume}{106},
  \bibinfo{pages}{403--421}.
%Type = Article
\bibitem[{Davis and Yin(2017)}]{davis2017three}
\bibinfo{author}{Davis, D.}, \bibinfo{author}{Yin, W.}, \bibinfo{year}{2017}.
\newblock \bibinfo{title}{A three-operator splitting scheme and its
  optimization applications}.
\newblock \bibinfo{journal}{Set-valued and variational analysis}
  \bibinfo{volume}{25}, \bibinfo{pages}{829--858}.
%Type = Article
\bibitem[{Demanet and Zhang(2016)}]{demanet2016eventual}
\bibinfo{author}{Demanet, L.}, \bibinfo{author}{Zhang, X.},
  \bibinfo{year}{2016}.
\newblock \bibinfo{title}{Eventual linear convergence of the
  {D}ouglas--{R}achford iteration for basis pursuit}.
\newblock \bibinfo{journal}{Mathematics of Computation} \bibinfo{volume}{85},
  \bibinfo{pages}{209--238}.
%Type = Article
\bibitem[{Douglas and Rachford(1956)}]{douglas1956numerical}
\bibinfo{author}{Douglas, J.}, \bibinfo{author}{Rachford, H.H.},
  \bibinfo{year}{1956}.
\newblock \bibinfo{title}{On the numerical solution of heat conduction problems
  in two and three space variables}.
\newblock \bibinfo{journal}{Transactions of the American mathematical Society}
  \bibinfo{volume}{82}, \bibinfo{pages}{421--439}.
%Type = Article
\bibitem[{Einfeldt et~al.(1991)Einfeldt, Munz, Roe and
  Sj{\"o}green}]{einfeldt1991godunov}
\bibinfo{author}{Einfeldt, B.}, \bibinfo{author}{Munz, C.D.},
  \bibinfo{author}{Roe, P.L.}, \bibinfo{author}{Sj{\"o}green, B.},
  \bibinfo{year}{1991}.
\newblock \bibinfo{title}{{On Godunov-type methods near low densities}}.
\newblock \bibinfo{journal}{Journal of computational physics}
  \bibinfo{volume}{92}, \bibinfo{pages}{273--295}.
%Type = Article
\bibitem[{Fan(1989)}]{fan1989new}
\bibinfo{author}{Fan, S.}, \bibinfo{year}{1989}.
\newblock \bibinfo{title}{A new extracting formula and a new distinguishing
  means on the one variable cubic equation}.
\newblock \bibinfo{journal}{Natural science journal of Hainan teachers college}
  \bibinfo{volume}{2}, \bibinfo{pages}{91}.
%Type = Article
\bibitem[{Gardner and Dwyer(2009)}]{gardner2009numerical}
\bibinfo{author}{Gardner, C.L.}, \bibinfo{author}{Dwyer, S.J.},
  \bibinfo{year}{2009}.
\newblock \bibinfo{title}{Numerical simulation of the {XZ} {T}auri supersonic
  astrophysical jet}.
\newblock \bibinfo{journal}{Acta Mathematica Scientia} \bibinfo{volume}{29},
  \bibinfo{pages}{1677--1683}.
%Type = Article
\bibitem[{Gil~Torres et~al.(2025)Gil~Torres, Jacobs and
  Zhang}]{torres2025asymptotic}
\bibinfo{author}{Gil~Torres, E.}, \bibinfo{author}{Jacobs, M.},
  \bibinfo{author}{Zhang, X.}, \bibinfo{year}{2025}.
\newblock \bibinfo{title}{Asymptotic linear convergence of {ADMM} for isotropic
  {TV} norm compressed sensing}.
\newblock \bibinfo{journal}{arXiv preprint arXiv:2505.01240} .
%Type = Article
\bibitem[{Grapsas et~al.(2016)Grapsas, Herbin, Kheriji and
  Latch{\'e}}]{grapsas2016unconditionally}
\bibinfo{author}{Grapsas, D.}, \bibinfo{author}{Herbin, R.},
  \bibinfo{author}{Kheriji, W.}, \bibinfo{author}{Latch{\'e}, J.C.},
  \bibinfo{year}{2016}.
\newblock \bibinfo{title}{An unconditionally stable staggered pressure
  correction scheme for the compressible {N}avier--{S}tokes equations}.
\newblock \bibinfo{journal}{The SMAI journal of computational mathematics}
  \bibinfo{volume}{2}, \bibinfo{pages}{51--97}.
%Type = Article
\bibitem[{Guba et~al.(2014)Guba, Taylor and St-Cyr}]{guba2014optimization}
\bibinfo{author}{Guba, O.}, \bibinfo{author}{Taylor, M.},
  \bibinfo{author}{St-Cyr, A.}, \bibinfo{year}{2014}.
\newblock \bibinfo{title}{Optimization-based limiters for the spectral element
  method}.
\newblock \bibinfo{journal}{Journal of Computational Physics}
  \bibinfo{volume}{267}, \bibinfo{pages}{176--195}.
%Type = Article
\bibitem[{Guermond et~al.(2021)Guermond, Maier, Popov and
  Tomas}]{guermond2021second}
\bibinfo{author}{Guermond, J.L.}, \bibinfo{author}{Maier, M.},
  \bibinfo{author}{Popov, B.}, \bibinfo{author}{Tomas, I.},
  \bibinfo{year}{2021}.
\newblock \bibinfo{title}{Second-order invariant domain preserving
  approximation of the compressible {N}avier--{S}tokes equations}.
\newblock \bibinfo{journal}{Computer Methods in Applied Mechanics and
  Engineering} \bibinfo{volume}{375}, \bibinfo{pages}{113608}.
%Type = Article
\bibitem[{Guermond et~al.(2018)Guermond, Nazarov, Popov and
  Tomas}]{guermond2018second}
\bibinfo{author}{Guermond, J.L.}, \bibinfo{author}{Nazarov, M.},
  \bibinfo{author}{Popov, B.}, \bibinfo{author}{Tomas, I.},
  \bibinfo{year}{2018}.
\newblock \bibinfo{title}{Second-order invariant domain preserving
  approximation of the {E}uler equations using convex limiting}.
\newblock \bibinfo{journal}{SIAM Journal on Scientific Computing}
  \bibinfo{volume}{40}, \bibinfo{pages}{A3211--A3239}.
%Type = Article
\bibitem[{Guermond and Popov(2017)}]{guermond2017invariant}
\bibinfo{author}{Guermond, J.L.}, \bibinfo{author}{Popov, B.},
  \bibinfo{year}{2017}.
\newblock \bibinfo{title}{Invariant domains and second-order continuous finite
  element approximation for scalar conservation equations}.
\newblock \bibinfo{journal}{SIAM Journal on Numerical Analysis}
  \bibinfo{volume}{55}, \bibinfo{pages}{3120--3146}.
%Type = Article
\bibitem[{Ha et~al.(2005)Ha, Gardner, Gelb and Shu}]{ha2005numerical}
\bibinfo{author}{Ha, Y.}, \bibinfo{author}{Gardner, C.L.},
  \bibinfo{author}{Gelb, A.}, \bibinfo{author}{Shu, C.W.},
  \bibinfo{year}{2005}.
\newblock \bibinfo{title}{Numerical simulation of high mach number
  astrophysical jets with radiative cooling}.
\newblock \bibinfo{journal}{Journal of Scientific Computing}
  \bibinfo{volume}{24}, \bibinfo{pages}{29--44}.
%Type = Article
\bibitem[{Hajduk(2021)}]{hajduk2021monolithic}
\bibinfo{author}{Hajduk, H.}, \bibinfo{year}{2021}.
\newblock \bibinfo{title}{Monolithic convex limiting in discontinuous
  {G}alerkin discretizations of hyperbolic conservation laws}.
\newblock \bibinfo{journal}{Computers \& Mathematics with Applications}
  \bibinfo{volume}{87}, \bibinfo{pages}{120--138}.
%Type = Article
\bibitem[{Kuzmin(2020)}]{kuzmin2020monolithic}
\bibinfo{author}{Kuzmin, D.}, \bibinfo{year}{2020}.
\newblock \bibinfo{title}{Monolithic convex limiting for continuous finite
  element discretizations of hyperbolic conservation laws}.
\newblock \bibinfo{journal}{Comput. Methods Appl. Mech. Engrg.}
  \bibinfo{volume}{361}, \bibinfo{pages}{112804}.
%Type = Book
\bibitem[{Kuzmin and Hajduk(2024)}]{kuzmin2024property}
\bibinfo{author}{Kuzmin, D.}, \bibinfo{author}{Hajduk, H.},
  \bibinfo{year}{2024}.
\newblock \bibinfo{title}{Property-preserving numerical schemes for
  conservation laws}.
\newblock \bibinfo{publisher}{World Scientific}.
%Type = Article
\bibitem[{Lin and Chan(2024)}]{lin2024high}
\bibinfo{author}{Lin, Y.}, \bibinfo{author}{Chan, J.}, \bibinfo{year}{2024}.
\newblock \bibinfo{title}{High order entropy stable discontinuous {G}alerkin
  spectral element methods through subcell limiting}.
\newblock \bibinfo{journal}{Journal of Computational Physics}
  \bibinfo{volume}{498}, \bibinfo{pages}{112677}.
%Type = Article
\bibitem[{Lions and Mercier(1979)}]{lions1979splitting}
\bibinfo{author}{Lions, P.L.}, \bibinfo{author}{Mercier, B.},
  \bibinfo{year}{1979}.
\newblock \bibinfo{title}{Splitting algorithms for the sum of two nonlinear
  operators}.
\newblock \bibinfo{journal}{SIAM Journal on Numerical Analysis}
  \bibinfo{volume}{16}, \bibinfo{pages}{964--979}.
%Type = Article
\bibitem[{Liu et~al.(2024a)Liu, Buzzard and Zhang}]{LZ2024CNS}
\bibinfo{author}{Liu, C.}, \bibinfo{author}{Buzzard, G.T.},
  \bibinfo{author}{Zhang, X.}, \bibinfo{year}{2024}a.
\newblock \bibinfo{title}{An optimization based limiter for enforcing
  positivity in a semi-implicit discontinuous {G}alerkin scheme for
  compressible {N}avier--{S}tokes equations}.
\newblock \bibinfo{journal}{Journal of Computational Physics}
  \bibinfo{volume}{519}, \bibinfo{pages}{113440}.
%Type = Article
\bibitem[{Liu et~al.(2025a)Liu, Hu, Taitano and Zhang}]{LHTZ2024FP}
\bibinfo{author}{Liu, C.}, \bibinfo{author}{Hu, J.}, \bibinfo{author}{Taitano,
  W.T.}, \bibinfo{author}{Zhang, X.}, \bibinfo{year}{2025}a.
\newblock \bibinfo{title}{An optimization-based positivity-preserving limiter
  in semi-implicit discontinuous {G}alerkin schemes solving {F}okker--{P}lanck
  equations}.
\newblock \bibinfo{journal}{Computers and Mathematics with Applications}
  \bibinfo{volume}{192}, \bibinfo{pages}{54--71}.
%Type = Article
\bibitem[{Liu et~al.(2024b)Liu, Riviere, Shen and Zhang}]{liu2024simple}
\bibinfo{author}{Liu, C.}, \bibinfo{author}{Riviere, B.},
  \bibinfo{author}{Shen, J.}, \bibinfo{author}{Zhang, X.},
  \bibinfo{year}{2024}b.
\newblock \bibinfo{title}{A simple and efficient convex optimization based
  bound-preserving high order accurate limiter for
  {C}ahn--{H}illiard--{N}avier--{S}tokes system}.
\newblock \bibinfo{journal}{SIAM Journal on Scientific Computing}
  \bibinfo{volume}{46}, \bibinfo{pages}{A1923--A1948}.
%Type = Article
\bibitem[{Liu et~al.(2025b)Liu, Sun and Zhang}]{LSZ2025cRKDG}
\bibinfo{author}{Liu, C.}, \bibinfo{author}{Sun, Z.}, \bibinfo{author}{Zhang,
  X.}, \bibinfo{year}{2025}b.
\newblock \bibinfo{title}{A bound-preserving {R}unge--{K}utta discontinuous
  {G}alerkin method with compact stencils for hyperbolic conservation laws}.
\newblock \bibinfo{journal}{Journal of Computational Physics} ,
  \bibinfo{pages}{114071}.
%Type = Article
\bibitem[{Liu and Zhang(2023)}]{LZ2022CNS}
\bibinfo{author}{Liu, C.}, \bibinfo{author}{Zhang, X.}, \bibinfo{year}{2023}.
\newblock \bibinfo{title}{A positivity-preserving implicit-explicit scheme with
  high order polynomial basis for compressible {N}avier--{S}tokes equations}.
\newblock \bibinfo{journal}{Journal of Computational Physics}
  \bibinfo{volume}{493}, \bibinfo{pages}{112496}.
%Type = Article
\bibitem[{Pazner(2021)}]{pazner2021sparse}
\bibinfo{author}{Pazner, W.}, \bibinfo{year}{2021}.
\newblock \bibinfo{title}{Sparse invariant domain preserving discontinuous
  {G}alerkin methods with subcell convex limiting}.
\newblock \bibinfo{journal}{Computer Methods in Applied Mechanics and
  Engineering} \bibinfo{volume}{382}, \bibinfo{pages}{113876}.
%Type = Article
\bibitem[{Peaceman and Rachford(1955)}]{peaceman1955numerical}
\bibinfo{author}{Peaceman, D.W.}, \bibinfo{author}{Rachford, Jr, H.H.},
  \bibinfo{year}{1955}.
\newblock \bibinfo{title}{The numerical solution of parabolic and elliptic
  differential equations}.
\newblock \bibinfo{journal}{Journal of the Society for industrial and Applied
  Mathematics} \bibinfo{volume}{3}, \bibinfo{pages}{28--41}.
%Type = Article
\bibitem[{Peterson et~al.(2024)Peterson, Bochev and
  Ridzal}]{peterson2024optimization}
\bibinfo{author}{Peterson, K.}, \bibinfo{author}{Bochev, P.},
  \bibinfo{author}{Ridzal, D.}, \bibinfo{year}{2024}.
\newblock \bibinfo{title}{Optimization-based, property-preserving algorithm for
  passive tracer transport}.
\newblock \bibinfo{journal}{Computers \& Mathematics with Applications}
  \bibinfo{volume}{159}, \bibinfo{pages}{267--286}.
%Type = Article
\bibitem[{Raguet(2019)}]{raguet2019note}
\bibinfo{author}{Raguet, H.}, \bibinfo{year}{2019}.
\newblock \bibinfo{title}{A note on the forward-{D}ouglas--{R}achford splitting
  for monotone inclusion and convex optimization}.
\newblock \bibinfo{journal}{Optimization Letters} \bibinfo{volume}{13},
  \bibinfo{pages}{717--740}.
%Type = Article
\bibitem[{Raguet et~al.(2013)Raguet, Fadili and
  Peyr{\'e}}]{raguet2013generalized}
\bibinfo{author}{Raguet, H.}, \bibinfo{author}{Fadili, J.},
  \bibinfo{author}{Peyr{\'e}, G.}, \bibinfo{year}{2013}.
\newblock \bibinfo{title}{A generalized forward-backward splitting}.
\newblock \bibinfo{journal}{SIAM Journal on Imaging Sciences}
  \bibinfo{volume}{6}, \bibinfo{pages}{1199--1226}.
%Type = Article
\bibitem[{Rueda-Ram{\'\i}rez et~al.(2022)Rueda-Ram{\'\i}rez, Pazner and
  Gassner}]{rueda2022subcell}
\bibinfo{author}{Rueda-Ram{\'\i}rez, A.M.}, \bibinfo{author}{Pazner, W.},
  \bibinfo{author}{Gassner, G.J.}, \bibinfo{year}{2022}.
\newblock \bibinfo{title}{Subcell limiting strategies for discontinuous
  {G}alerkin spectral element methods}.
\newblock \bibinfo{journal}{Computers \& Fluids} \bibinfo{volume}{247},
  \bibinfo{pages}{105627}.
%Type = Article
\bibitem[{Ruppenthal and Kuzmin(2023)}]{ruppenthal2023optimal}
\bibinfo{author}{Ruppenthal, F.}, \bibinfo{author}{Kuzmin, D.},
  \bibinfo{year}{2023}.
\newblock \bibinfo{title}{Optimal control using flux potentials: {A} way to
  construct bound-preserving finite element schemes for conservation laws}.
\newblock \bibinfo{journal}{Journal of Computational and Applied Mathematics}
  \bibinfo{volume}{434}, \bibinfo{pages}{115351}.
%Type = Article
\bibitem[{van~der Vegt et~al.(2019)van~der Vegt, Xia and
  Xu}]{van2019positivity}
\bibinfo{author}{van~der Vegt, J.J.}, \bibinfo{author}{Xia, Y.},
  \bibinfo{author}{Xu, Y.}, \bibinfo{year}{2019}.
\newblock \bibinfo{title}{{Positivity preserving limiters for time-implicit
  higher order accurate discontinuous Galerkin discretizations}}.
\newblock \bibinfo{journal}{SIAM journal on scientific computing}
  \bibinfo{volume}{41}, \bibinfo{pages}{A2037--A2063}.
%Type = Article
\bibitem[{Wang et~al.(2012)Wang, Zhang, Shu and Ning}]{wang2012robust}
\bibinfo{author}{Wang, C.}, \bibinfo{author}{Zhang, X.}, \bibinfo{author}{Shu,
  C.W.}, \bibinfo{author}{Ning, J.}, \bibinfo{year}{2012}.
\newblock \bibinfo{title}{{Robust high order discontinuous Galerkin schemes for
  two-dimensional gaseous detonations}}.
\newblock \bibinfo{journal}{Journal of Computational Physics}
  \bibinfo{volume}{231}, \bibinfo{pages}{653--665}.
%Type = Article
\bibitem[{Wu et~al.()Wu, Zhang and Shu}]{wu2025high}
\bibinfo{author}{Wu, K.}, \bibinfo{author}{Zhang, X.}, \bibinfo{author}{Shu,
  C.W.}, .
\newblock \bibinfo{title}{High order numerical methods preserving invariant
  domain for hyperbolic and related systems}.
\newblock \bibinfo{journal}{to appear in SIAM Review}
  \bibinfo{note}{ArXiv:2512.09116}.
%Type = Article
\bibitem[{Xu(2014)}]{xu2014parametrized}
\bibinfo{author}{Xu, Z.}, \bibinfo{year}{2014}.
\newblock \bibinfo{title}{Parametrized maximum principle preserving flux
  limiters for high order schemes solving hyperbolic conservation laws:
  one-dimensional scalar problem}.
\newblock \bibinfo{journal}{Math. Comp.} \bibinfo{volume}{83},
  \bibinfo{pages}{2213--2238}.
%Type = Article
\bibitem[{Yee et~al.(2020)Yee, Olivier, Haut, Holec, Tomov and
  Maginot}]{yee2020quadratic}
\bibinfo{author}{Yee, B.C.}, \bibinfo{author}{Olivier, S.S.},
  \bibinfo{author}{Haut, T.S.}, \bibinfo{author}{Holec, M.},
  \bibinfo{author}{Tomov, V.Z.}, \bibinfo{author}{Maginot, P.G.},
  \bibinfo{year}{2020}.
\newblock \bibinfo{title}{{A quadratic programming flux correction method for
  high-order DG discretizations of SN transport}}.
\newblock \bibinfo{journal}{Journal of Computational Physics}
  \bibinfo{volume}{419}, \bibinfo{pages}{109696}.
%Type = Article
\bibitem[{Zalesak(1979)}]{zalesak1979fully}
\bibinfo{author}{Zalesak, S.T.}, \bibinfo{year}{1979}.
\newblock \bibinfo{title}{Fully multidimensional flux-corrected transport
  algorithms for fluids}.
\newblock \bibinfo{journal}{Journal of computational physics}
  \bibinfo{volume}{31}, \bibinfo{pages}{335--362}.
%Type = Article
\bibitem[{Zhang(2017)}]{zhang2017positivity}
\bibinfo{author}{Zhang, X.}, \bibinfo{year}{2017}.
\newblock \bibinfo{title}{On positivity-preserving high order discontinuous
  {G}alerkin schemes for compressible {N}avier--{S}tokes equations}.
\newblock \bibinfo{journal}{Journal of Computational Physics}
  \bibinfo{volume}{328}, \bibinfo{pages}{301--343}.
%Type = Article
\bibitem[{Zhang and Shu(2010a)}]{zhang2010maximum}
\bibinfo{author}{Zhang, X.}, \bibinfo{author}{Shu, C.W.},
  \bibinfo{year}{2010}a.
\newblock \bibinfo{title}{On maximum-principle-satisfying high order schemes
  for scalar conservation laws}.
\newblock \bibinfo{journal}{Journal of Computational Physics}
  \bibinfo{volume}{229}, \bibinfo{pages}{3091--3120}.
%Type = Article
\bibitem[{Zhang and Shu(2010b)}]{zhang2010positivity}
\bibinfo{author}{Zhang, X.}, \bibinfo{author}{Shu, C.W.},
  \bibinfo{year}{2010}b.
\newblock \bibinfo{title}{On positivity-preserving high order discontinuous
  {G}alerkin schemes for compressible {E}uler equations on rectangular meshes}.
\newblock \bibinfo{journal}{Journal of Computational Physics}
  \bibinfo{volume}{229}, \bibinfo{pages}{8918--8934}.
%Type = Article
\bibitem[{Zhao et~al.(2019)Zhao, Zhu, Chen and Qiu}]{zhao2019new}
\bibinfo{author}{Zhao, Z.}, \bibinfo{author}{Zhu, J.}, \bibinfo{author}{Chen,
  Y.}, \bibinfo{author}{Qiu, J.}, \bibinfo{year}{2019}.
\newblock \bibinfo{title}{A new hybrid {WENO} scheme for hyperbolic
  conservation laws}.
\newblock \bibinfo{journal}{Computers \& Fluids} \bibinfo{volume}{179},
  \bibinfo{pages}{422--436}.

\end{thebibliography}

\end{document}